\documentclass[12pt]{amsart}
\usepackage{amsfonts, amssymb, amsgen, amsbsy, amstext, amsopn, amsmath,
latexsym, indentfirst,color,longtable}
\usepackage{verbatim} 
\usepackage{hyperref} 
\usepackage[american]{babel}
\renewcommand{\H}{\mathbb{H}}
\newcommand{\B}{\mathbb{B}}

\newcommand{\G}{\mathbb{G}}

\newcommand{\N}{\mathbb{N}}

\newcommand{\R}{\mathbb{R}}

\newcommand{\cA}{\mathcal{A}}

\newcommand{\cC}{\mathcal{C}}

\newcommand{\cF}{\mathcal{F}}

\newcommand{\cH}{\mathcal{H}}
\newcommand{\cI}{\mathcal{I}}
\newcommand{\cL}{\mathcal{L}}

\newcommand{\cO}{\mathcal{O}}
\newcommand{\cP}{\mathcal{P}}

\newcommand{\cS}{\mathcal{S}}

\newcommand{\cM}{\mathcal{M}}

\newcommand{\cV}{\mathcal{V}}

\newcommand{\ep}{\varepsilon}

\newcommand{\ph}{\varphi}

\newcommand{\sm}{\setminus}

\newcommand{\lb}{{\big\lbrace}}
\newcommand{\rb}{{\big\rbrace}}

\newcommand{\LLs}{\bigg(}
\newcommand{\RRs}{\bigg)}

\newcommand{\ls}{\mbox{\large $($}}
\newcommand{\rs}{\mbox{\large $)$}}
\newcommand{\lls}{\big(}
\newcommand{\rrs}{\big)}

\newcommand{\diams}{\mbox{\rm diam}\;\!}
\newcommand{\diam}{\mbox{\rm diam}}

\renewcommand{\exp}{\mbox{\rm exp}\;\!}

\newcommand{\spn}{\mbox{\rm span}}

\newcommand{\dist}{\mbox{dist}}

\newcommand{\vol}{\mbox{\rm vol}}
\newcommand{\sgn}{\mbox{\rm sgn}\,}

\newcommand{\Id}{\mbox{\rm Id}}
\newcommand{\Lie}{\mathrm{Lie}}

\newcommand{\q}{\mathrm q}
\newcommand{\m}{\mathrm m}
\newcommand{\n}{\mathrm n}
\newcommand{\h}{\mathrm{h}}
\newcommand{\rr}{\mathrm{r}}
\newcommand{\NN}{\mathrm N}
\newcommand{\MM}{\mathrm M}
\newcommand{\QQ}{\mathrm Q}

\newcommand{\bcup}{\bigcup}

\newcommand{\res}{\mbox{\LARGE{$\llcorner$}}}

\newcommand{\lra}{\longrightarrow}
\newcommand{\der}{\partial}

\newcommand{\ds}{\displaystyle}

\newcommand{\beqas}{\begin{eqnarray*}}
\newcommand{\eeqas}{\end{eqnarray*}}
\newcommand{\beqa}{\begin{eqnarray}}
\newcommand{\eeqa}{\end{eqnarray}}
\newcommand{\beq}{\begin{equation}}
\newcommand{\eeq}{\end{equation}}
\newcommand{\bce}{\begin{center}}
\newcommand{\ece}{\end{center}}

\newcommand{\pa}[1]{\left( #1 \right)}               
\newcommand{\qa}[1]{\left[ #1 \right]}               
\newcommand{\set}[1]{\left\{ #1 \right\}}            
\newcommand{\pal}[1]{\left| #1 \right|}            
\newcommand{\ban}[1]{\left\langle  #1 \right\rangle}  
\newcommand{\dpar}[2]{\frac{\partial #1}{\partial #2}}

\newcommand{\qs}[1]{\quad\mbox{ #1} \quad}               

\newcommand{\qandq}{\quad\mbox{and}\quad}

\newtheorem{The}{Theorem}[section]
\newtheorem{Lem}[The]{Lemma}
\newtheorem{Def}[The]{Definition}
\newtheorem{Rem}[The]{Remark}
\newtheorem{Pro}[The]{Proposition}
\newtheorem{Cor}[The]{Corollary}
\newtheorem{Exa}[The]{Example}
\newtheorem{Con}{Conjecure}

\newcommand{\bt}{\begin{The}}
\newcommand{\et}{\end{The}}
\newcommand{\bl}{\begin{Lem}}
\newcommand{\el}{\end{Lem}}
\newcommand{\bd}{\begin{Def}\rm}
\newcommand{\ed}{\end{Def}}
\newcommand{\br}{\begin{Rem}\rm}
\newcommand{\er}{\end{Rem}}
\newcommand{\bpr}{\begin{Pro}}
\newcommand{\epr}{\end{Pro}}
\newcommand{\bc}{\begin{Cor}}
\newcommand{\ec}{\end{Cor}}
\newcommand{\bj}{\begin{Con}}
\newcommand{\ej}{\end{Con}}
\newcommand{\bex}{\begin{Exa}}
\newcommand{\eex}{\end{Exa}}

\numberwithin{equation}{section}

\begin{document}

\title[Towards a theory of area in homogeneous groups]
{Towards a theory of area in homogeneous groups}

%
%
%
\author{Valentino Magnani}
\address{Valentino Magnani, Dipartimento di Matematica, Universit\`a di Pisa \\
Largo Bruno Pontecorvo 5 \\ I-56127, Pisa}
\email{valentino.magnani@unipi.it}
\date{\today}
%
%
%
%
\subjclass[2010]{Primary 28A75. Secondary 53C17, 22E30.}
\keywords{homogeneous group, area, spherical measure, submanifolds} 
\date{\today}

\begin{abstract}
A general approach to compute the spherical measure of submanifolds in homogeneous groups is provided. We focus our attention on the {\em homogeneous tangent space}, that is a suitable weighted algebraic expansion of the submanifold. This space plays a central role for the existence of blow-ups. Main applications are area-type formulae for new classes of $C^1$ smooth submanifolds. We also study various classes of distances, showing how their symmetries lead to simpler area and coarea formulas.
Finally, we establish the equality between spherical measure and Hausdorff measure on all horizontal submanifolds.
\end{abstract}

\maketitle

\tableofcontents


\section{Introduction}

The notion of surface area is fundamental in several areas of mathematics, such as geometric analysis, differential geometry and geometric measure theory. 
Area formulae for rectifiable sets in Riemannian manifolds and general metric spaces are well known \cite{Kir94}, \cite{AmbKir2000Rect}. 
When the metric space is not Riemannian, as a noncommutative homogeneous group (Section~\ref{sect:Graded}), even smooth sets need not be rectifiable in the standard metric sense \cite{Federer69}.
Such unrectifiability occurs when the Hausdorff dimension is greater than the topological dimension. In Carnot-Carath\'eodory spaces all smooth submanifolds ``generically'' have this dimensional gap \cite[Section 0.6.B]{Gromov1996}, so several well known tools of geometric measure theory do not apply. The basic question of computing an area formula for the Hausdorff measure remains a difficult task, even for smooth submanifolds.

Hausdorff measure plays a fundamental role in geometric measure theory, as it is well witnessed by the following Federer's words \cite{Federer1978Colloquium}. ``It took five decades, beginning with Carathéodory's fundamental paper on measure theory in 1914, to develop the intuitive conception of an $m$ dimensional surface as a mass distribution into an efficient instrument of mathematical analysis, capable of significant applications in the calculus of
variations. The first three decades were spent learning basic facts on how
subsets of $\R^n$ behave with respect to $m$ dimensional Hausdorff measure
$\cH^m$. During the next two decades this knowledge was fused with many techniques from analysis, geometry and algebraic topology, finally to produce new and sometimes surprising but classically acceptable solutions to old problems.''

Federer's comments remain extremely appealing when applied to the Hausdorff measure in nilpotent groups, that have a more complicated geometric structure.
The wider program of studying analysis and geometry in such groups and
general Carnot-Carath\'eodory spaces already appeared
in the seminal works by H\"ormander \cite{Hormander67},
Folland \cite{Fol75}, Stein \cite{SteinICM1977}, Gromov \cite{Gromov1996}, Rothschild and Stein \cite{RothschildStein76}, Nagel, Stein and Wainger \cite{NagelSteinWainger85}
and many others.
An impressive number of papers prove the always expanding interest on understanding geometric measure theory in such non-Euclidean frameworks.

 Among the many topics that have been studied, we mention
projection theorems, unrectifiability \cite{BDCFMT2013EffectProj}, \cite{BFMT2012ProjSlicing}, \cite{Hovila2014isotropicprojections}, \cite{FasslerHovila2016}, sets of finite h-perimeter, intrinsic regular sets, intrinsic differentiability, rectifiability \cite{Amb01}, \cite{FSSC3} \cite{FSSC4}, \cite{FSSC5}, \cite{FSSC6}, \cite{Mag5}, \cite{AKLD2009TangSpGro}, \cite{Mag14}, \cite{FranchiMarchiSerapioni2014IntDiff}, \cite{FranchiSerapioni2016IntrLip},  
\cite{Mag31}, \cite{MagnaniComiGaussGreen}, differentiation of measures and covering theorems, uniform measures, singular integrals \cite{Mag30}, \cite{LeDR2017Besicovitch}, \cite{ChoTys2015Marstrand}, \cite{ChoMagTys2015}, \cite{ChousionisLi2017Nonnegative} and minimal surfaces  \cite{BASCV2007Bernstein},  \cite{MSSC10}, \cite{MonVitHeightEst2015}, \cite{MonSteLipApp2017},   \cite{CHMY2005MinSurfPseudoherm},
\cite{RitRos2006RotationallyInv}, \cite{CapCitManfr2009RegHeisMinGraphs}, \cite{CapCitManfr2010SmoothLipMinGraphs}, \cite{DGNP2010MinSurfEmbedded}, \cite{HurtRitRos2010ClassifMinSurfHeis}, 
\cite{GalliRitore2015AreaStatC1}, \cite{Montef2015HminHyp},
\cite{SerraCassanoNicolussi2018}.
These works represent only a small part of a vaster and always growing literature.

Aim of the present work is to establish area formulas for the spherical measure of new classes of $C^1$ smooth submanifolds. One of the key tools is the intrinsic blow-up, performed by translations and dilations that are compatible with the metric structure of the group (Section~\ref{sect:notions}).
The blow-up is expected to exist on ``metric regular points''.
Precisely, these are those points having maximum pointwise degree \eqref{d:degree_point}, that is a kind of ``pointwise Hausdorff dimension''.
The pointwise degree was introduced by Gromov in \cite[Section 0.6.B]{Gromov1996}. It was subsequently rediscovered in \cite{Mag13Vit}, through an algebraic definition that also provides the density of the spherical measure.

However, pointwise degree does not possess enough information to describe the local behavior of the submanifold. We show how a more precise local geometric description  is available through the {\em homogeneous tangent space}, in short {\em h-tangent space}.
It is not difficult to find submanifolds of the same topological dimension, having the same pointwise degree at a fixed point, but whose corresponding h-tangent spaces are algebraically different (Remark~\ref{r:comparisondegreehtan}).

The construction of the h-tangent space is purely algebraic. It arises from a formal ``weighted homogeneous expansion'' of the standard tangent space (Definition~\ref{def:homtan}). Such a notion was introduced in \cite{Mag13Vit} for points of maximum degree. 
In the same paper it was proved that for $C^{1,1}$ smooth submanifolds the h-tangent space is always 
a homogeneous subgroup  (Definition~\ref{d:hsubgroup}). 
This kind of ``algebraic regularity'' was then used to establish the blow-up,
in addition to the $C^{1,1}$ smoothness.

New questions arise when we reduce the regularity of submanifolds to $C^1$.
Thus, studying in depth the algebraic structure of the h-tangent space becomes crucial. We focus our attention on {\em algebraically regular points}, i.e.
those points whose h-tangent space is a homogeneous subgroup.
We may a priori distinguish those submanifolds that at least at 
points of maximum degree have the h-tangent space in a specific family of subgroups.
In Sections~\ref{section:Legendrian} and~\ref{Sect:transversal} we focus our attention on horizontal submanifolds and transversal submanifolds, that have such an algebraic characterization.
For these submanifolds we can compute their spherical measure. 

Horizontal submanifolds are defined by having 
the h-tangent space everywhere isomorphic to a {\em horizontal subgroup} (Definition~\ref{d:horizontalsubgroup}).
The crucial relation is the inclusion 
\begin{equation}\label{eq:T_pinclusionH_pI} 
T_p\Sigma\subset H_p\G,
\end{equation}
for the submanifold $\Sigma$ at every point $p$, with horizontal fiber
$H_p\G$ defined in \eqref{d:H_pG}.
This condition is everywhere satisfied by all horizontal submanifolds
(Remark~\ref{r:horiz_points}).
To more easily detect and construct horizontal submanifolds, it is important 
to verify whether the everywhere validity of \eqref{eq:T_pinclusionH_pI}
implies that $\Sigma$ is a horizontal submanifold.

We notice that \eqref{eq:T_pinclusionH_pI} satisfied only at a single point $p$
does not imply that $p$ is horizontal (Example~\ref{ex:charact_Heis}).
If \eqref{eq:T_pinclusionH_pI} holds on an open subset of a $C^2$ submanifold $\Sigma$, then the approach of the classical Frobenius theorem implies that $\Sigma$ is horizontal (Proposition~\ref{pr:horizC^2smooth}).
For $C^1$ smooth submanifolds the situation is more delicate,
since commutators of vector fields are not defined. Surprisingly,
with $C^1$ regularity the classical proof of Frobenius theorem can be
replaced by a differentiability result. 
Indeed, the horizontality condition \eqref{eq:T_pinclusionH_pI} implies a suitable differentiability of the parametrization of $\Sigma$ (Theorem~\ref{t:horizC^1Smooth}), that is the well known as Pansu differentiability.
As a result, the area formulas \eqref{eq:area} and \eqref{eq:cS^N_dAll} hold 
for all $C^1$ smooth horizontal submanifolds 
satisfing \eqref{eq:T_pinclusionH_pI} at every point.
These submanifolds include for instance horizontal curves and Legendrian submanifolds.

Transversal submanifolds can be defined through {\em transversal points}, which are those points whose h-tangent space is a vertical subgroup (Definition~\ref{d:vertical_subgroup}).
Due to this transversality, with arguments similar to those of \cite[Section~4]{Mag12A}, one could see that generically every smooth submanifold is transversal.
All $C^1$ smooth hypersurfaces are special instances of transversal submanifolds.
Every transversal submanifold is characterized by having maximal Hausdorff dimension among all $C^1$ smooth submanifolds with the same topological dimension \cite{Mag26TV}. 
The same condition characterizes vertical subgroups with respect to homogeneous subgroups.

Our first result is Theorem~\ref{t:LocExpSurf}, that establishes the existence 
of the blow-up at an algebraically regular point of a $C^1$ smooth submanifold, under different conditions.
The proof of this theorem, besides including new cases as $C^1$ smooth submanifolds in two step groups and horizontal submanifolds in homogeneous groups, also simplifies 
the previous arguments.
The blow-up of a submanifold $\Sigma$ can be turned into a suitable differentiation of its intrinsic measure $\mu_\Sigma$ (Definition~\ref{d:SRmeasure}).
This measure, first introduced in \cite{Mag13Vit}, takes into account the degree $\NN$ of $\Sigma$ and the graded structure of the group.
Finding the relationship between $\mu_\Sigma$ and the spherical measure of $\Sigma$ corresponds to establish an area formula, due to the explicit form of $\mu_\Sigma$.

We use a suitable differentiation of the intrinsic measure, that works in metric spaces  \cite{Mag30}. 
In Section~\ref{sect:meastheo} we adapt the general differentiation to homogeneous groups. The point is to find an explicit formula for the Federer density $\theta^\NN(\mu_\Sigma,\cdot)$, that is defined by \eqref{eq:FDens} in any metric space and appears in the measure theoretic area formula
\eqref{eq:spharea}. We are lead to our first result.
%
%
%
%
\begin{The}[Upper blow-up theorem]\label{t:UpBlwC}
Let $\Sigma\subset\G$ be a $C^1$ smooth submanifold of topological dimension $\n$
and degree $\NN$. Let $p\in\Sigma$ be an algebraically regular point of maximum degree $\NN$
and let $A_p\Sigma$ be the $\n$-dimensional homogeneous tangent space. 
We assume that one of the following assumptions holds:
\begin{enumerate}
\item
$p$ is a horizontal point 
\item
$\G$ has step two 
\item
$\Sigma$ is a one dimensional submanifold
\item
$p$ is a transversal point.
\end{enumerate}
Then the Federer density satisfies the following formula 
\begin{equation}\label{e:UpBlwSub}
\theta^\NN(\mu_\Sigma,p)= \beta_d\lls A_p\Sigma\rrs\,.
\end{equation}
\end{The}
The degree of $\Sigma$ is the maximum integer $\NN$ among all pointwise degrees of $\Sigma$.
The number $\beta_d(A_p\Sigma)$ is the {\em spherical factor} (Definition~\ref{d:sphericalfactor}) associated to the h-tangent space $A_p\Sigma$ of $\Sigma$ at $p$.
Such a number amounts to the maximal area of the intersection of $A_p\Sigma$ with any metric unit ball that is not too far from the origin.
In Euclidean space, corresponding to the commutative group $\G\approx\R^\n$ with step one and distance $d_E(x,y)=|x-y|$, we get $\beta_{d_E}(A_p\Sigma)\equiv\omega_\n$, that is the volume of the unit ball in $\R^n$.
When the metric unit ball $\B(0,1)$ is convex a simple formula 
for $\beta_d(A_p\Sigma)$ is available on vertical subgroups.
\begin{The}\label{t:homConvBall}
If $d$ is a homogeneous distance whose metric unit ball $\B(0,1)$ is convex and 
$N\subset\G$ is an $\n$-dimensional vertical subgroup of $\G$, then 
\begin{equation}\label{eq:bdN}
 \beta_d(N)=\cH_{|\cdot|}^\n(N\cap\B)\,.
\end{equation}
\end{The}
As an application, joining \eqref{eq:bdN}, Theorem~\ref{t:area} below and 
the results of \cite{FSSC8}, one can prove that on all 
transversal submanifolds the equality between spherical 
measure and centered Hausdorff measure holds, then including hypersurfaces.
The main tool in the proof of Theorem~\ref{t:homConvBall} 
is a suitable concavity property of the area of ``parallel sections'' of convex sets. 
We expect this property to be well known (Theorem~\ref{t:ConvSect_n-dim}).
For reader's convenience, we have provided a proof, 
that is an application of Brunn-Minkowski inequality.
Joining Theorem~\ref{t:UpBlwC} and the measure theoretic area formula \eqref{eq:spharea}, we arrive at our second main result.

\bt[Area formula for all homogeneous distances]\label{t:area}
Let $\Sigma\subset\G$ be a $C^1$ smooth $\n$-dimensional submanifold of degree $\NN$. Suppose that one of the following conditions hold.
\begin{enumerate}
\item
$\Sigma$ is a horizontal submanifold.
\item
$\G$ has step 2, every point of maximum degree is algebraically regular and points of lower degree are $\cS^\NN$ negligible.
\item
$\Sigma$ is a transversal submanifold.
\item
$\Sigma$ is one dimensional.
\end{enumerate}
Then for any Borel set $B\subset\Sigma$ the following area formula holds \beq\label{eq:area}
\mu_\Sigma(B)=\int_B \|\tau^{\tilde g}_{\Sigma,\NN}(p)\|_g\, d\sigma_{\tilde g}(p)
=\int_B \beta_d(A_p\Sigma)\, d\cS^\NN_0(p).
\eeq
\et 
We refer the reader to Section~\ref{sect:meastheo} for the definitions of the projected $\tilde g$-unit tangent $\n$-vector $\tau_{\Sigma,\NN}^{\tilde g}$ and 
the ``nonrenormalized'' spherical measure $\cS_0^\NN$.
Theorem~\ref{t:area} is the union of different results contained in Section~\ref{sect:SphericalMeasure} and Section~\ref{sect:HorizSubmanifolds}. 
The implication from (1) to \eqref{eq:area}
corresponds to Theorem~\ref{t:AreaHoriz}
and in this case $\n=\NN=\deg\Sigma$.
It follows that the area formula \eqref{eq:area}  holds for $C^1$ smooth submanifolds everywhere tangent to the horizontal subbundle, due to Theorem~\ref{t:horizC^1Smooth}.
In particular, we can are able to compute the spherical measure of all $C^1$ smooth Legendrian submanifoldd in any Heisenberg group.

The other implications of Theorem~\ref{t:area} all need a negligibility result for the set of points
of lower degree. If $\Sigma$ has degree $\NN$ greater than its topological dimension, we have to prove that the characteristic set 
\beq\label{eq:S_Sigma-R_SigmaI}
\cC_\Sigma=\set{p\in\Sigma: d_\Sigma(p)<\NN}
\eeq
is $\cS^\NN$ negligible.
The implication from assumption (2) to \eqref{eq:area}
is a consequence of Theorem~\ref{t:area2steps}.
By results of \cite{Mag12B}, when $\Sigma$ is $C^{1,1}$ smooth in a two step group
we have $\cS^\NN(\cS_\Sigma)=0$ and every point of maximum degree is algebraically regular \cite{Mag13Vit}. Thus, assumptions (2) are more general than the conditions required in \cite{Mag12B}. The validity of \eqref{eq:area} from hypothesis (3) is a consequence of Theorem~\ref{t:TransvSub}, where the $\cH^\NN$ negligibility of $\cC_\Sigma$ 
is a nontrivial fact \cite{Mag26TV}.
The implication from (4) to \eqref{eq:area} comes from Theorem~\ref{t:curves},
slightly extending the results of \cite{Mag21Korte}.

Let us point out that \eqref{eq:area} cannot be obtained through $C^{1,1}$ smooth approximation of $C^1$ submanifolds, since continuity theorems for the spherical measure require strong topological constraints. Additional efforts may arise to preserve the degree of the approximating submanifolds
and possible ``isolated submanifolds'' of specific degree \eqref{d:degSigma}
could also appear. Such difficulties justify why working with $C^1$ submanifolds meets a number of difficulties.

Formula \eqref{eq:area} provides an explicit relationship between the intrinsic 
measure and the spherical measure. The latter is constructed by a homogeneous distance, that may be also the sub-Riemannian distance. Then the terminology ``sub-Riemannian measure'' for the intrinsic measure is 
somehow justified.

Besides \eqref{eq:bdN}, another important case is when the spherical factor
\beq\label{eq:betadSigma}
\cF\ni V\to \beta_d(V)
\eeq
is constant on a specific family $\cF$ of $\n$-dimensional subspaces,
giving a standard form to the area formula as in \eqref{eq:cS^N_d}.
We consider precisely $\n$-vertically symmetric distances (Definition~\ref{d:hsym}),
that possess some invariance of the metric unit ball with respect to a precise family of isometries.
Such symmetries allow to have the spherical factor constant on all vertical subgroups (Theorem~\ref{t:constBeta}).
In general,  when $\beta_d$ is constant on a specific family 
$\cF$ of $\n$-dimensional homogeneous subgroups of Hausdorff dimension
$\NN$, we may redefine the spherical measure as follows
\beq\label{eq:cSNN}
\cS^\NN_d=\omega_d(\n,\NN) \cS^\NN_0.
\eeq
This definition yields the standard spherical measure of the Euclidean space
when we consider it as a commutative homogeneous group equipped with
the Euclidean distance.
From the previous results we are arrived at the following theorem.
\bt[Area formula for transversal submanifolds and symmetric distances]\label{t:areaTransversal}
Let $\Sigma\subset\G$ be an $\n$-dimensional transversal submanifold
of degree $\NN$ and let $d$ be an $\n$-vertically symmetric distance. There exists a geometric constant $\omega_d(\n,\NN)$
such that there holds
\beq\label{eq:cS^N_d}
\cS^\NN_d(B)=\int_B\|\tau^{\tilde g}_{\Sigma,\NN}\|_g\, d\sigma_{\tilde g}
\eeq
for any Borel set $B\subset\Sigma$.
\et
The class of $\n$-vertically symmetric distances is rather general,
since it includes also distances whose unit ball is not convex,
as for instance the sub-Riemannian distance in the Heisenberg group \cite{Mag31}. 
As an application, in Corollary~\ref{c:Coarea} we obtain a standard form of the coarea formula, whenever the spherical measure is construced by any $(\q-k)$-vertically symmetric distance.
To have constant spherical factor in more cases, we consider {\em multiradial distances} (Definition~\ref{d:multirad}) that are $\n$-vertically symmetric for every $\n$ (Proposition~\ref{pr:multiradial}). For these distances
the following area formula holds.

\bt[Area formula for multiradial distances]\label{t:areaIIConstBeta}
Let $\G$ be equipped with a multiradial distance $d$ and let $\Sigma\subset\G$ be a $C^1$ smooth $\n$-dimensional submanifold of degree $\NN$. Suppose that one of the following conditions hold.
\begin{enumerate}
\item
$\G$ has step 2, every point of maximum degree is algebraically regular and points of lower degree are $\cS^\NN$ negligible.
\item
$\Sigma$ is one dimensional.
\item
$\Sigma$ is a horizontal submanifold.
\end{enumerate}
Thus, in any of these cases the spherical factor $\beta_d$ is constant on all homogeneous tangent spaces of maximum degree and we set $\beta_d(\cdot)=\omega(\n,\NN)$. As a consequence, defining $\cS^\NN_d$ as in \eqref{eq:cSNN},
for any Borel set $B\subset\Sigma$ the following area formula holds 
\beq\label{eq:cS^N_dAll}
\cS^\NN_d(B)=\int_B\|\tau^{\tilde g}_{\Sigma,\NN}\|_g\, d\sigma_{\tilde g}.
\eeq
\et 
The proof of this theorem follows by joining Theorem~\ref{t:area2steps},
Theorem~\ref{t:curves} and Theorem~\ref{t:AreaHoriz}.
Finally, we provide a first general formula relating spherical measure and
Hausdorff measure for a class of submanifolds in homogeneous groups.

\bt[Hausdorff and spherical measure of horizontal submanifolds]\label{t:HausdSpheric}
Let $d$ be a multiradial distance and let $\Sigma\subset\G$ be a horizontal submanifold. Then the following equality holds
\beq\label{eq:cHcSNN}
\cH_d^\n\res\Sigma=\cS_d^\n\res\Sigma
\eeq
where $\cS^\n_d=\omega(\n,\n)\cS^\n_0$ and $\cH^\n_d=\omega(\n,\n)\cH_0^\n$.
\et
This formula also includes the Euclidean one. 
In general, the constant $\omega(\n,\n)$ is the area of the metric unit ball intersected with an $\n$-dimensional space contained in the first stratum of $\G$.

The results of this paper provide a strong evidence that a unified approach to the area formula in homogeneous groups can be achieved.
However, several questions are still to be understood. 
Whether or not an ``algebraic classification'' of submanifolds is required certainly represents a first question, which may have an independent interest. 
Other issues may arise from the absence of a general negligibility result for points of low degree. These issues, as many others, are a matter for future investigations.

%
%
%
%
%
%
\section{Basic notions}\label{sect:notions}
%
%
%
%
%
%
%

\subsection{Graded nilpotent Lie groups and their metric structure}\label{sect:Graded}
A connected and simply connected {\em graded nilpotent Lie group} can be regarded as a graded linear space $\G=H^1\oplus\cdots\oplus H^\iota$ equipped with a polynomial group operation such that its Lie algebra $\Lie(\G)$ is {\em graded}. This grading corresponds to the following conditions
\beq\label{eq:LieG}
\Lie(\G)=\cV_1\oplus\cdots\oplus\cV_\iota, \qquad [\cV_i,\cV_j]\subset \cV_{i+j}
\eeq
for all integers $i,j\ge0$ and $\cV_{j}=\{0\}$ for all $j>\iota$, with $\cV_\iota\neq\{0\}$. 
The integer $\iota \ge 1$ is the {\em step} of the group.
The graded structure of $\G$ allows us to introduce intrinsic dilations $\delta_r:\G\to\G$ as linear mappings such that
$\delta_r(p)=r^ip$ for each $p\in H^i$, $r>0$ and $i=1,\ldots,\iota$.
The graded nilpotent Lie group $\G$ equipped with intrinsic dilations
is called {\em homogeneous group}, \cite{FS82}.
With the stronger assumption that 
\beq
[\cV_1,\cV_j]=\cV_{j+1}
\eeq
for each $j=1,\ldots,\iota$ and $[\cV_1,\cV_\iota]=\set{0}$, we say that $\G$ is a 
{\em stratified group}. Identifying further $\G$ with the tangent space $T_0\G$ at the origin $0$, we have a canonical isomorphism between $H^j$ and $\cV_j$, that associates to each $v\in H^j$ the unique left invariant vector field $X\in\cV_j$ such that $X(0)=v$.

We may also assume that $\G$ is equipped with a Lie product that induces a Lie algebra structure, where its group operation is given 
through the Baker-Campbell-Hausdorff formula:
\beq\label{eq:BCH}
xy=\sum_{j=1}^\iota c_j(x,y)=x+y+\frac{[x,y]}2+\sum_{j=3}^\iota c_j(x,y)
\eeq
with $x,y\in\G$. Here $c_j$ denote homogeneous polynomials of degree $j$ with respect to the nonassociative Lie product on $\G$.
We will refer to \eqref{eq:BCH} in short as BCH.
It is always possible to have these additional conditions, since the exponential mapping 
\[
\exp:\Lie(\G)\to\G
\]
of any simply connected nilpotent Lie group $\G$ is a bianalytic diffeomorphism.
In addition, the given Lie product and the Lie algebra associated to
the induced group operation are compatible, according to the 
following standard fact.
\bpr\label{pr:GAlgGroupOpe}
Let $G$ be a nilpotent, connected and simply connected Lie group
and consider the new group operation given by \eqref{eq:BCH}.
Then the Lie algebra associated to this Lie group structure
is isomorphic to the Lie algebra of $G$.
\epr
\begin{comment}
Dimostrazione nel file 2012-03-16_22.pdf corrispondente a due
lezioni del mio corso sui gruppi omogenei.
\end{comment}
We will denote by $\q$ the dimension of $\G$, seen as a linear space.

\begin{Def}\label{d:hsubgroup}\rm
A linear subspace $S$ of $\G$ that satisfies $\delta_r(S)\subset S$ for every $r>0$ is a {\em homogeneous subspace} of $\G$.
If in addition $S$ is a Lie subgroup of $\G$ then we say that
$S$ is a {\em homogeneous subgroup} of $\G$.
\ed
Using dilations it is not difficult to check that $S\subset\G$ is a 
homogeneous subspace if and only if we have the direct decomposition
\[
S=S_1\oplus\cdots\oplus S_\iota,
\]
where each $S_j$ is a subspace of $H^j$.
%

A {\em homogeneous distance} $d$ on a graded nilpotent Lie group $\G$ is a left invariant distance with $d(\delta_rx,\delta_ry)=r\,d(p,q)$ for all $p,q\in\G$ and $r>0$. 
We define the open and closed balls \[ B(p,r)=\lb q\in\G: d(q,p)<r\rb\quad\mbox{and}\quad \B(p,r)=\lb q\in\G: d(q,p)\le r \rb\,. \]
The corresponding homogeneous norm is denoted by $\|x\|=d(x,0)$ for all $x\in\G$. 
When the graded nilpotent Lie group is equipped with the corresponding dilations, along with
a homogeneous norm, is called {\em homogeneous group}.

In the special case $\G$ is a stratified group, the distribution of subspaces given by
the so-called {\em horizontal fibers} 
\beq\label{d:H_pG}
H_p\G=\{X(p)\in T_p\G: X\in\cV_1\}
\eeq
with $p\in\G$ satisfies the Lie bracket generating condition. Thus, in view of Chow's theorem, a left invariant metric restricted to horizontal fibers leads us to the well known sub-Riemannian distance, that is an important example of homogeneous distance.
In this case the terminology {\em Carnot group} for $\G$ is well known.
We denote by $H\G$ the {\em horizontal subbundle} of $\G$, whose fibers are precisely the ones of \eqref{d:H_pG}.

\begin{comment} 
{\color{red} 
As mentioned in the introduction, we fix an auxiliary scalar product on $\G$ and we make this choice such that the fixed graded basis is orthonormal.
The restriction of this scalar product to $V_1$ can be translated to the so-called {\em horizontal fibers}
\[
H_x\G=\{X(x)\in T_x\G: X\in\cV_1\}
\]
as $x$ varies in $\G$, hence defining a left invariant sub-Riemannian metric $g$ on $\G$.
We denote by $H\G$ the {\em horizontal subbundle} of $\G$, whose fibers are $H_x\G$.
%
%
%
\begin{Def}[Horizontal subbundle and its sections] \rm
Let $O\subset\G$ be an open set. We denote by $HO$ the restriction of the {\em horizontal subbundle} $H\G$ to the open set $O$, whose {\em horizontal fibers} $H_x\G$ are restricted to all points $x\in O$. We denote by $\cS_c(HO)$ the linear space of all smooth sections of $HO$ with compact support in $O$.
\end{Def}
%
%
%
We also introduce the {\em horizontal gradient}
\[
\nabla_Hf=(X_1f,\ldots,X_\m f).
\]
}
\end{comment}
%
%
%
A {\em graded basis} $(e_1,\ldots,e_\q)$ of a homogeneous group $\G$ is a basis of vectors such that 
\beq
(e_{\m_{j-1}+1},e_{\m_{j-1}+2},\ldots,e_{\m_j})
\eeq
is a basis of $H^j$ for each $j=1,\ldots,\iota$, where 
\beq
\m_j=\sum_{i=1}^j\h_i \qandq \h_j=\dim H^j,
\eeq
we have set $\m_0=0$. We also set $\m=\m_1$ and observe that $\m_\iota=\q$.
A graded basis provides the associated {\em graded coordinates} $x=(x_1,\ldots,x_\q)\in\R^\q$, then defining the unique element $p=\sum_{j=1}^\q x_je_j\in\G$.
\br
It is easy to realize that one can always equip a homogeneous subgroup with graded coordinates.
\er
%
%
Throughout this work, a {\em graded left invariant Riemannian metric} $g$ is fixed on the homogeneous group $\G$. This metric automatically induces a scalar product on $T_0\G$, therefore our identification of $\G$ with $T_0\G$ yields a fixed Euclidean structure in $\G$. The fact that our left invariant Riemannian metric $g$ is ``graded'' means that the induced scalar product on $\G$ is {\em graded}, namely, all subspaces $H^i$ with $i=1,\ldots,\iota$ are orthogonal to each other.  
With a slight abuse of notation, the Euclidean norm on $\G$ and the norm arising from the Riemannian metric $g$ on tangent spaces will be denoted by 
the same symbol $|\cdot|$. 

For the sequel, it is also useful to recall that when a Riemannian metric $\tilde g$ is fixed on $\G$, then a scalar product on $\Lambda_k(T_p\G)$ is automatically induced for every $p\in\G$. The corresponding norm on $k$-vectors 
is denoted by $\|\cdot\|_{\tilde g}$. A {\em $\tilde g$-unit $k$-vector} $v\in \Lambda_k(T_p\G)$ satisfies $\|v\|_{\tilde g}=1$.
\br
One can easily check that when a graded scalar product is fixed, we can 
find a graded basis that is also orthonormal with respect to this scalar product.
\er
\subsection{Degrees, multivectors and projections}
In this section we present a suitable notion of degree and of
projection on $k$-vectors. Let us consider a graded basis
$(e_1,\ldots,e_\q)$ of $\G$ and the corresponding left invariant vector fields 
$X_j\in\Lie(\G)$ such that $X_j(0)=e_j$ for each $j=1,\ldots,\q$.
We have obtained a basis $(X_1,\ldots,X_\q)$ of the Lie algebra $\Lie(\G)$.
If the graded basis is orthonormal with respect to $g$, then our frame automatically becomes orthonormal. In the sequel, we will consider graded orthonormal frames.

If $(x_1,\ldots,x_\q)$ are graded coordinates of graded basis
$(e_1,\ldots,e_\q)$, we assign {\em degree $j$} to each
coordinate $x_i$ such that $e_i\in H^j$.
We analogously assign degree $j$ to each left invariant vector field of $\cV_j$.
In different terms, for each $i\in\set{1,\ldots,\q}$ we consider the unique integer $d_i$ on $\set{1,\dots,\iota}$ such that 
\[
\m_{d_i-1}< i\leq \m_{d_i}.
\]
It is easy to observe that $d_i$ is the degree of both the coordinate $x_i$ and 
the left invariant vector field $X_i$.

We denote by $\cI_{k,\q}$ the family of all multi-index $I=(i_1,\ldots,i_k)\in \set{1,\ldots,\q}^k$ such that $1\le i_1<\cdots<i_k\le\q$.
For each $I\in\cI_{k,\q}$, we define the $k$-vector 
\beq\label{eq:X_I}
X_I=X_{i_1}\wedge\cdots \wedge X_{i_k}\in\Lambda_k\!\pa{\Lie(\G)},
\eeq
whose degree is defined as follows 
\[
d(X_I)=d_{i_1}+\cdots+d_{i_k}.
\]

\br
The set $\set{X_I: I\in\cI_{k,\q}}$ constitutes a basis of $\Lambda_k(\Lie(\G))$.
We also notice that the degree of $X_1\wedge\cdots\wedge X_\q$ is precisely
\[
\QQ=d_1+\cdots+d_\q,
\]
where this number coincides with the Hausdorff dimension of $\G$ with respect to an arbitrary homogeneous distance.
\er
The space $\Lambda_k(\Lie(\G))$ can be identified with the space of {\em left invariant $k$-vector fields}. The sections $\xi$ of the vector bundle $\Lambda_k\G=\bcup_{p\in\G} \Lambda_k(T_p\G)$ are precisely the $k$-vector fields of $\G$. The left invariance of $\xi$ is expressed by the equality
\[
(\Lambda_kl_p)_*(\xi)=\xi
\]
for every $p\in\G$, where $z\to l_pz=pz$ denotes the {\em left translation}
by $p$. On a simple $k$-vector field $Z_1\wedge\cdots\wedge Z_k$ made by the vector fields $Z_1,\ldots,Z_k$ of $G$, we have defined
\[
(\Lambda_kl_p)_*(Z_1\wedge\cdots\wedge Z_k)=(l_p)_*Z_1\wedge\cdots\wedge(l_p)_*Z_k,
\]
where $(l_p)_*Z_j$ is the push-forward of $Z_j$ by $l_p$.
In the sequel, we will automatically identify the space of $k$-vectors $\Lambda_k(\Lie G)$ 
with the space of left invariant $k$-vector fields.
Indeed, whenever $\xi$ is a left invariant $k$-vector field, the mapping that associates
$\xi$ to $\xi(0)\in\Lambda_k(T_0\G)$ is an isomorphism and 
$\Lambda_k(T_0\G)$ is isomorphic to $\Lambda_k(\Lie\G)$.

\begin{Def}[Projections on $k$-vectors]\rm
Let $(X_1,\ldots, X_\q)$ be a graded orthonormal frame, let $1\le k \le \q$ and $1\le M\le \QQ$ be integers. For each left invariant $k$-vector field $\xi\in\Lambda_k(\Lie(\G))$,
written as $\xi=\sum_{I\in \cI_{k,\q}} c_I\, X_I$ for a suitable set of real numbers $\set{c_I}$,
we define the {\em $M$-projection of $\xi$} as follows
\[
\pi_M(\xi)=\sum_{\substack{I\in \cI_{k,\q} \\ d(X_I)=M}} c_I\, X_I\in\Lambda_k(\Lie(\G)).
\]
This defines a mapping $\pi_M:\Lambda_k(\Lie(\G))\to \Lambda_k^M(\Lie(\G))$, where we have set
\[
\Lambda_k^M(\Lie(\G))=\set{\sum_{I\in \cI_{k,\q}} c_I\, X_I: d(X_I)=M,\ c_I\in\R }.
\]
For each $p\in\G$, we can also introduce the fibers
\[
\Lambda_k^M\ls T_p\G\rs=\set{\xi(p)\in\Lambda_k(T_p\G): \xi\in\Lambda_k^M(\Lie(\G))},
\]
along with the following pointwise $M$-projection
\beq\label{d:projM}
\pi_{p,\MM}(z)=\pi_\MM(\xi)(p)\in\Lambda_k(T_p\G),
\eeq
where $z\in \Lambda_k(T_p\G)$ and there exists a unique 
$\xi\in  \Lambda_k(\Lie(\G))$ such that $\xi(p)=z$.
We clearly have $\pi_{p,\MM}:\Lambda_k(T_p\G)\to \Lambda_k^M\ls T_p\G\rs$.
By our identification of $\G$ with $T_0\G$, we introduce
the translated projection of $k$-vectors at a point $p$ to the origin:
\beq\label{eq:Pi_pM0}
\pi_{p,\MM}^0:\Lambda_k(T_p\G)\to \Lambda_k\G.
\eeq
For each $z\in \Lambda_k(T_p\G)$, we consider the unique element $\xi\in\Lambda_k(\Lie(\G))$ such that 
\[
\xi(p)=z\qandq \pi^0_{p,\MM}(z)=\pi_\MM(\xi)(0)\in\Lambda_k(T_0\G)\simeq \Lambda_k(\G).
\]
\ed
%
%
%
%
%
%
%
%
%
\subsection{The homogeneous tangent space}\label{sect:AlgTanSpace}
%
%
%
%
%
%
%
%
%
%
In this section and in the sequel $\Sigma$ denotes an $\n$-dimensional $C^1$ smooth submanifold embedded in a homogeneous group $\G$.
A {\em tangent $\n$-vector} of $\Sigma$ at $p\in\Sigma$ is 
\[
\tau_\Sigma(p)=t_1\wedge\cdots \wedge t_\n\in\Lambda_\n(T_p\Sigma),
\]
where $(t_1,\ldots,t_\n)$ is a basis of $T_p\Sigma$. 
This vector is not uniquely defined, but any other choice of the basis
of $T_p\Sigma$ yields a proportional $\n$-vector.
%
%
%
We define the {\em pointwise degree} $d_\Sigma(p)$ of $\Sigma$ at $p$ as the integer
\beq\label{d:degree_point}
d_\Sigma(p)=\max\set{M\in\N: \pi_{p,M}\pa{\tau_\Sigma(p)}\neq 0}
\eeq
and the {\em degree} of $\Sigma$ is the positive integer
\beq\label{d:degSigma}
d(\Sigma)=\max\{d_\Sigma(p):p\in\Sigma\}\in\N\sm\set{0}.
\eeq 
We say that $p\in\Sigma$ has {\em maximum degree} if $d_\Sigma(p)=d(\Sigma)$.
\begin{Def}[Homogeneous tangent space]\label{def:homtan}\rm
Let $p\in \Sigma$ and set $d_\Sigma(p)=\NN$.
If $\tau_\Sigma(p)$ is a tangent $\n$-vector to $\Sigma$ at $p$
and $\xi_{p,\Sigma}\in\Lambda_\n(\Lie(\G))$ is the unique left invariant $\n$-vector 
field such that $\xi_{p,\Sigma}(p)=\tau_\Sigma(p)$, then
we define the {\em Lie homogeneous tangent space} of $\Sigma$ at $p$,
in short the {\em Lie h-tangent space} as follows
\[
\cA_p\Sigma=\set{X\in\Lie(\G): X\wedge \pi_\NN(\xi_{p,\Sigma})=0}.
\]
We say that $p\in\Sigma$ is {\em algebraically regular} if 
$\cA_p\Sigma$ is a subalgebra of $\Lie(\G)$. 
In this case we call the corresponding subgroup
\[
A_p\Sigma=\exp\cA_p\Sigma
\]
the {\em homogeneous tangent space} of $\Sigma$ at $p$, or simply
the {\em h-tangent space} of $\Sigma$ at $p$.
\end{Def}
\begin{Rem}\rm
It is very important that the h-tangent space can be defined at any point of a smooth submanifold of a graded group. In many cases, it precisely coincides with the blow-up of the 
submanifold, when it is performed by intrinsic dilations and the group operation.
\end{Rem}
Any point of a $C^1$ smooth curve of $\G$ is algebraically regular, 
since any one dimensional linear subspace of a layer $H^j$ 
is automatically a homogeneous subalgebra.
Points that are not algebraically regular may appear in submanifolds
of dimension higher than one, according to the next example.
\begin{Exa}\label{ex:charact_Heis}\rm
Let the first Heisenberg group $\H$ be identifed with $\R^3$ through the coordinates
$(x_1,x_2,x_3)$ such that the group operation reads as follows
\[
(x_1,x_2,x_3)(x_1',x_2',x_3')=(x_1+x_1',x_2+x_2',x_3+x_3'+x_1x_2'-x_2x_1').
\]
Let $\Sigma=\set{(x_1,x_2,x_3)\in\H: x_3=x_1^2+x_2^2}$ be a 2-dimensional submanifold.
Let us show that the origin $p=(0,0,0)\in\Sigma$ is not algebraically regular. 
It is easy to observe that $d_\Sigma(p)=2$. 
We have $T_p\Sigma=\spn\set{e_1,e_2}$, where $(e_1,e_2,e_3)$ is the canonical basis of $\R^3$,
therefore $\tau_\Sigma(p)=e_1\wedge e_2$.
Introducing the left invariant vector fields 
\[
X_1(x)=e_1-x_2e_3 \qandq X_2=e_2+x_1e_3,
\]
have may define $\xi=X_1\wedge X_3$, of degree two, such that $\xi(0)=e_1\wedge e_2$.
This implies that $\pi_{p,2}(e_1\wedge e_2)\neq0$ and $\pi_{p,j}(e_1\wedge e_2)=0$ for all $j\ge 3$. The Lie h-tangent space is
defined as follows
\[
\cA_p\Sigma=\set{X\in\Lie(\G): X\wedge X_1\wedge X_2=0}=\spn\set{X_1,X_2},
\]
that is not a Lie subalgebra of $\Lie(\H)$.
The homogeneous tangent space 
\[
A_p\Sigma=\exp\cA_p\Sigma=\set{(x,y,0)\in\H: x,y\in\R}
\]
is a subspace of $\H$, but it is not a subgroup.
\end{Exa}
\begin{Rem}[Characteristic points]\rm
We observe that in the previous example the origin $p$ is also a 
{\em characteristic point} of $\Sigma$. The general definition states that
a characteristic point $q$ of a $C^1$ smooth hypersurface $\Sigma\subset\G$ 
satisfies $H_q\G\subset T_q\Sigma$.
This kind of point behaves as a singular point with respect 
to the metric strucure of $\G$.

The notion of algebraic regularity fits with this picture
in that characteristic points are not algebraically regular,
as it can be seen arguing as in Example~\ref{ex:charact_Heis} and taking into account the invariance of the pointwise degree under left translations, as shown in
Proposition~\ref{pr:translation}. 
On the other hand, for all $C^1$ smooth hypersurfaces, 
characteristic points are negligible with respect to the $(Q-1)$-dimensional Hausdorff measure \cite{Mag5}. 
\er
\bex\label{ex:Legendrian_Heis^2}\rm
Let $p$ be a point of a 2-dimensional Legendrian submanifold $\Sigma$, see Section~\ref{section:Legendrian}, that is embedded in the second Heisenberg
group $\H^2$.  Then $d_{\Sigma}(p)=2$ and $p$ is an algebraically regular point
whose homogeneous tangent space is a commutative horizontal subgroup of $\H^2$.
Here we consider $\H^2$ as $\R^5$ equipped with the horizontal left invariant
vector fields
\[
X_1(x)=e_1-x_3e_5,\quad X_2=e_2-x_4e_5,\quad X_3(x)=e_3+x_1e_5,\quad X_4=e_4+x_2e_5,
\]
spanning the first layer of the stratified Lie algebra $\Lie(\H^2)$.
\eex

\begin{Rem}\label{r:comparisondegreehtan}\rm
Examples~\ref{ex:charact_Heis} and \ref{ex:Legendrian_Heis^2}
show that one can find different submanifolds of the same dimension 
with points of the same degree, where only one of these points
is algebraically regular. This shows somehow that algebraic regularity
encodes the ``behavior'' of the submanifold around the point.
The pointwise degree clearly provides less information.
\er 
%
%
%
%
%
%
%
%
%
\section{Special coordinates around points of submanifolds}\label{sect:SpecialCoordinates}
%
%
%
%
%
%
%
%
%
%
Throughout this section, the symbol $\Sigma\subset\G$ will denote a $C^1$ smooth submanifold embedded in a homogeneous group $\G$, if not otherwise stated.
To perform the blow-up of $\Sigma$ at a fixed point, finding special coordinates is 
of capital importance. They are also useful to determine degree and homogeneous tangent space of a fixed point. 

From the proof of \cite[Lemma~3.1]{Mag13Vit}, it is not difficult to see that special coordinates can be found around any point
of a submanifold, that need not have maximum degree. 
This is the content of the following theorem.
\begin{The}\label{t:specialcoord}
Let $\Sigma\subset\G$ be a $C^1$ smooth submanifold of topological dimension $\n$ and let $0\in\Sigma$. There exist $\alpha_1,\ldots,\alpha_\iota\in\N$ with $\alpha_j\le\h_j$ for all $j=1,\ldots,\iota$, an orthonormal graded basis $(e_1,\ldots,e_\q)$ with respect to the fixed graded scalar product
on $\G$, a bounded open
neighborhood $U\subset\R^\n$ of the origin
and a $C^1$ smooth embedding $\Psi:U\to\Sigma$ with the following 
properties. There holds $\Psi(0)=0\in\G$, for all $y\in U$
\[
\Psi(y)=\sum_{j=1}^\q \psi_j(y) e_j, \quad \psi(y)=(\psi_1(y),\ldots,\psi_\q(y))
\] 
and the Jacobian matrix of $\psi$ at the origin is
\begin{equation}\label{e:matriceC}
D\psi (0)=\left(\begin{array}{c|c|c|c|c|c}
I_{\alpha_1} & 0 & \cdots & \cdots &\cdots & 0\\
0 & \ast &   \cdots &\cdots & \cdots &\ast\\
\hline 0 & I_{\alpha_2} & 0  & \cdots &\cdots & 0 \\
0 & 0 & \ast & \cdots & \cdots &\ast\\
\hline 0 & 0 & I_{\alpha_3}  & 0 & \cdots &0 \\
0 & 0 & 0 & \ast & \cdots &\ast\\
\hline \vdots & \vdots & \vdots   & \ddots & \ddots & \vdots \\
\hline 0 & 0 & \cdots &\cdots & \cdots &I_{\alpha_\iota}\\
0 & 0 & \cdots & \cdots & \cdots &0
\end{array}\right).
\end{equation}
The blocks containing the identity matrix $I_{\alpha_j}$ have $\h_j$ rows,  for every $j=1,\ldots,\iota$. The blocks $\ast$ are $(\h_j-\alpha_j)\times\alpha_i$ matrices,
for all $j=1,\ldots,\iota-1$ and $i=j+1,\ldots,\iota$. The mapping $\psi$ can be assumed to 
have the special graph form given by the conditions
\beq\label{eq:psicoord}
\psi_s(y)=y_{s-\m_{j-1}+\mu_{j-1}}	 
\eeq
for every $s=\m_{j-1}+1,\ldots,\m_{j-1}+\alpha_j$ and $j=1,\ldots,\iota$, where we have defined
\beq\label{eq:muj}
\mu_0=0\qandq \mu_j=\sum_{i=1}^j\alpha_i\quad \mbox{for}\  j=1,\ldots,\iota.
\eeq
\end{The}
%
%
%
%
\begin{comment}
\begin{proof}
Since in \cite[Corollary 3.8]{Mag13Vit} there are not all details here we sketch a more complete proof. By Lemma~3.1 of \cite{Mag13Vit} there exists a $C^1$ smooth embedding $\Phi:V\to\G$ on an open neighborhood 
of the origin $V\subset\R^\n$ such that $\Phi(V)\subset\Sigma$, $\Phi(0)=0$, 
\[
\Phi(x)=\sum_{j=1}^\q\phi_j(x)e_j
\]
and the Jacobian matrix of $\phi(x)=(\phi_1(x),\ldots,\phi_\q(x))$ at the origin has the form \eqref{e:matriceC}
We define the projection $\pi_0:\R^\q\to\R^\n$ as follows 
\[
\pi_0(x_1,\ldots,x_n)=(x_1,\ldots,x_{\alpha_1},x_{\m_1+1},\ldots, x_{\m_1+\alpha_2},\ldots,x_\n)\,.
\]
Thus, $d(\pi_0\circ\phi)(0)$ is the identity matrix, hence the composition $\phi\circ (\pi_0\circ\phi)^{-1}$
has the required properties ....
\end{proof}
\end{comment}
%
%
\br\label{r:localdegreealpha_i}
The numbers $\alpha_j$ provided by Theorem~\ref{t:specialcoord} are 
uniquely defined and do not depend on the choice of the special coordinates $\psi_j$.
One may also observe that
\beq\label{eq:n_d_Sigma}
\mu_\iota=\n \qandq d_\Sigma(0)=\sum_{i=1}^\iota i\, \alpha_i,
\eeq
where $\n$ is the topological dimension of $\Sigma$ and $d_\Sigma(0)$ is the 
degree of $\Sigma$ at the origin.
\er
\bpr\label{pr:homtansp}
Under the assumptions of Theorem~\ref{t:specialcoord},  the homogeneous tangent space of $\Sigma$ at the origin can be represented as follows
\beq\label{eq:homtansp}
A_0\Sigma=\spn\set{e_1,\ldots,e_{\alpha_1},e_{\m_1+1},\ldots,e_{\m_1+\alpha_2},\ldots,e_{\m_{\iota-1}+1},\ldots,e_{\m_{\iota-1}+\alpha_\iota}}.
\eeq
\epr
\begin{proof}
From the form of the Jacobian matrix \eqref{e:matriceC} and the definition of 
homogeneous tangent space, there holds
\[
\pi_{0,\NN}^0\pa{\der_1\psi(0)\wedge\der_2\psi(0)\wedge\cdots\wedge \der_\n\psi(0)}=
e_1\wedge\cdots\wedge e_{\alpha_1} \cdots \wedge
e_{\m_{\iota-1}+1}\wedge\cdots\wedge e_{\m_{\iota-1}+\alpha_1}.
\]
The unique left invariant $\n$-vector field $\xi\in\Lambda_\n(\Lie(\G))$ such that
\[
\xi(0)=\der_1\psi(0)\wedge\der_2\psi(0)\wedge\cdots\wedge \der_\n\psi(0)
\]
then satisfies
\[
\pi_\NN(\xi)=X_1\wedge\cdots\wedge X_{\alpha_1}\wedge \cdots \wedge X_{\m_{\iota-1}+1}\wedge\cdots\wedge X_{\m_{\iota-1}+\alpha_1}.
\]
As a result, in view of Definition~\ref{def:homtan} our claim is established.
\end{proof}
The special coordinates of Theorem~\ref{t:specialcoord} allow us 
to introduce an ``induced degree'' on $\Sigma$, as in the next definition.
\begin{Def}\rm
In the notation of Theorem~\ref{t:specialcoord}, we define
\beq\label{e:degSigma}
b_i=j\qs{if and only if} \mu_{j-1}<i\le\mu_j
\eeq
for every $i=1,\ldots,\n$. The integer $b_i$ is the {\em induced degree of $y_i$},
with respect to the coordinates $y=(y_1,\ldots,y_\n)$ of $\Sigma$ around
the origin, in Theorem~\ref{t:specialcoord}. We define accordingly the 
{\em induced dilations} $\sigma_r:\R^\n\to\R^\n$ as follows
\begin{equation}\label{d:dilintrinsic}
\sigma_r(t_1,\ldots,t_\n)=(r^{b_1} t_1,\ldots,r^{b_\n} t_\n) \quad \text{where $r>0$}.
\end{equation}
\ed
The coordinates $y$ of Theorem~\ref{t:specialcoord} allow us to act on
the homogeneous tangent space $A_p\Sigma$ of $\Sigma$ through the induced dilations $\sigma_r$. This is an important fact, that will be used in the sequel.
\bc\label{c:matriceCcont}
Under the assumptions of Theorem~\ref{t:specialcoord},
we consider the frame of left invariant vector fields 
\[
X_1,\ldots,X_\q
\]
adapted to the coordinates of the theorem, namely we impose
the condition $X_j(0)=e_j$ for each $j=1,\ldots,\q$.
Then there exist unique continuous coefficients $C_i^s$ such that
\begin{equation}\label{e:derjXj}
\der_i\psi=\sum_{s=1}^\q C^s_i(\psi)\, X_s(\psi)\quad \mbox{for all $i=1,\ldots,\n$}.
\end{equation}
If $0\in\Sigma$ has maximum degree, then the $\q\times\n$ matrix-valued function $C$ of coefficients $C_i^s$ satisfies the following formula 
\begin{equation}\label{eq:matriceCcont}
C=\left(\begin{array}{c|c|c|c|c|c}
I_{\alpha_1}+o(1) & o(1) & \cdots & \cdots &\cdots & o(1)\\
o(1) & \ast &   \cdots &\cdots & \cdots &\ast\\
\hline o(1) & I_{\alpha_2}+o(1) & o(1)  & \cdots &\cdots & o(1) \\
0 & o(1) & \ast & \cdots & \cdots &\ast\\
\hline o(1) & o(1) & I_{\alpha_3}+o(1)  & o(1) & \cdots &o(1) \\
0 & 0 & o(1) & \ast & \cdots &\ast\\
\hline \vdots & \vdots & \vdots   & \ddots & \ddots & \vdots \\
\hline o(1) & o(1) & \cdots &\cdots & \cdots &I_{\alpha_\iota}+o(1)\\
0 & 0 & \cdots & \cdots & \cdots &o(1)
\end{array}\right).
\end{equation}
The symbols $o(1)$ denote a continuous submatrix that vanishes at $0$.
The constantly null submatrices in \eqref{eq:matriceCcont} are denoted by $0$. In the case $0\in\Sigma$ is not of maximum degree these
submatrices are replaced by other matrices $o(1)$ vanishing at $0$.
\ec
\begin{proof}
The form \eqref{eq:matriceCcont} of $C$ follows from \eqref{e:matriceC} 
joined with the assumption that $0\in\Sigma$ has maximum degree.
The assumption on the maximum degree of the origin is needed only to obtain the constantly vanishing submatrices of \eqref{eq:matriceCcont}.
\end{proof}
Theorem~\ref{t:specialcoord} and Corollary~\ref{c:matriceCcont} 
provides special coordinates around any point of a smooth submanifold.
The next proposition shows that translations preserve
the ``algebraic structure of points''.
\bpr\label{pr:translation} 
If $\Sigma$ is a $C^1$ smooth submanifold, $p\in\Sigma$ and 
we define the translated submanifold $\Sigma_p=p^{-1}\Sigma$, then 
\[
d_\Sigma(p)=d_{\Sigma_p}(0), \quad \cA_p\Sigma=\cA_0\Sigma_p
\qandq A_p\Sigma=A_0\Sigma_p.
\]
\epr
\begin{proof}
We consider the tangent $\n$-vector
\[
\tau_\Sigma(p)=\sum_{I\in\cI_{k,\q}} c_I  X_I(p),
\]
where $X_I$ are defined in \eqref{eq:X_I} and $c_I\in\R$
and the translated one
\[
\tau_{\Sigma_p}(0)=dl_{p^{-1}}\pa{\sum_{I\in\cI_{k,\q}} c_I  X_I(p)}=
\sum_{I\in\cI_{k,\q}} c_I  X_I(0).
\]
We have used the left invariance of the basis $(X_1,\ldots,X_\q)$,
that defines the $k$-vectors $X_I$.
This invariance of the coefficients $c_I$ joint with
the definition of degree and of homogeneous tangent
space immediately lead us to our claim.
\end{proof}

As we have previously seen, the continuous matrix \eqref{eq:matriceCcont} is
related to the algebraic structure of the homogeneous tangent space
$A_0\Sigma$ and it plays an important role in the proof of the blow-up
of Theorem~\ref{t:LocExpSurf}. This result considers four distinct cases that correspond to
different ``shapes'' of the submanifold around the blow-up point. 
It is then imporant to make 
the form of the continuous matrix \eqref{eq:matriceCcont} explicit in each of the four cases.

If $\G$ is of step two, the continuous matrix $C$ of \eqref{eq:matriceCcont}
takes the form
\begin{equation}\label{eq:matriceCcontiota2}
C=\left(\begin{array}{c|c}
 I_{\alpha_1}+o(1) & o(1)   \\
o(1) & \ast \\
\hline
 o(1) & I_{\alpha_2}+o(1)   \\
 0 & o(1)  \\
\end{array}\right).
\end{equation}
In the case $\Sigma$ is curve embedded in $\G$, namely $\n=1$, $\alpha_\NN=1$, we have 
\begin{equation}\label{eq:matriceCcontn1}
C=\left(\begin{array}{c}
\vdots \\
\hline
\ast   \\
 \hline
I_{\alpha_\NN}+o(1)  \\
o(1)  \\
\hline
0   \\
\hline
\vdots \\
\hline
0
\end{array}\right),
\end{equation}
where $I_{\alpha_\NN}$ in this case denotes the $1\times 1$ matrix equal to one.
The remaining two cases, related to the special structure of the homogeneous tangent space, need to be treated in more detail.
They are indeed related to specific classes of submanifolds.

\section{Horizontal points and horizontal submanifolds}\label{section:Legendrian}

Horizontal points are a specific class of algebraically regular points, 
associated to a class of subgroups.
The interesting fact is that they have a corresponding class of submanifolds, where all points are horizontal.
\begin{Def}[Horizontal subgroup]\label{d:horizontalsubgroup}\rm
We say that $H\subset\G$ is a {\em horizontal subgroup} 
if it is a homogeneous subgroup contained in the first layer $H^1$ of $\G$.
\ed
Clearly horizontal subgroups are automatically commutative.
\begin{Def}[Horizontal points and horizontal submanifolds]\rm
A {\em horizontal point} $p$ of a $C^1$ smooth submanifold $\Sigma$ embedded in a homogeneous group $\G$ is an algebraically regular whose homogeneous tangent space is a horizontal subgroup. The submanifold $\Sigma$ is {\em horizontal} if all of its points are horizontal.
\ed
Horizontal points determine a special form of the matrix $C$ in Corollary~\ref{c:matriceCcont}, as shown in the next proposition.
\bpr\label{pr:horizC}
In the assumptions of Corollary~\ref{c:matriceCcont}, if the origin $0\in\Sigma$
is a horizontal point, then $\alpha_1=\n$, $\alpha_j=0$ for each $j=2,\ldots,\iota$
and the continuous matrix \eqref{eq:matriceCcont} takes the following form
\begin{equation}\label{eq:matriceCcontHor}
C=\left(\begin{array}{c}
I_{\alpha_1}+o(1) \\
o(1) \\
\hline 
0 \\
\hline 
\vdots \\
\hline
0
\end{array}\right).
\end{equation}
\epr
\begin{proof}
From Proposition~\ref{pr:homtansp}, we have $A_0\Sigma=\spn\set{e_1,\ldots,e_{\alpha_1}}$.
This immediately shows the form of $C$ given in \eqref{eq:matriceCcontHor}.
\end{proof}
\br\label{r:horiz_degree}
Joining the previous proposition with Remark~\ref{r:localdegreealpha_i}
and taking into account the left invariance pointed out in Proposition~\ref{pr:translation}, 
one immediately observes that all points of an $\n$-dimensional horizontal submanifold $\Sigma$ have degree $\n$.
Therefore the degree of $\Sigma$ coincides
with its topological dimension.
\er
\br\label{r:horiz_points}
Proposition~\ref{pr:horizC} shows in particular that a horizontal point
$p$ of a $C^1$ smooth submanifold $\Sigma$ must satisfy the condition
\beq\label{eq:T_pinclusionH_p}
T_p\Sigma\subset H_p\G.
\eeq
Then any $C^1$ smooth horizontal submanifold is
tangent to the horizontal subbundle $H\G$. In different terms,
$\Sigma$ is an integral submanifold of the distribution made by the fibers $H_p\G$.
\er
The inclusion \eqref{eq:T_pinclusionH_p} alone does not imply that $p$ is horizontal, see Example~\ref{ex:charact_Heis}.
\bpr\label{pr:horizC^2smooth}
If $\Sigma$ is a $C^2$ smooth submanifold such that $T_p\Sigma\subset H_p\G$ 
for every $p\in\Sigma$, then $\Sigma$ is a horizontal submanifold.
\epr
\begin{proof}
Fix $p\in\Sigma$ and consider two arbitrary $C^1$ smooth 
sections $X$ and $Y$ of the tangent bundle $T\Sigma$,
which are defined on a neighborhood $U$ of $p$.
There exist $a_j,b_l$ $C^1$ smooth coefficients on $U$ such that
\[
X=\sum_{j=1}^\m a_j X_j\qandq Y=\sum_{j=1}^\m b_j X_j
\]
where $(X_1,\ldots,X_\m)$ is a frame of horizontal left invariant
vector fields, namely a basis of the first layer $\cV_1\subset\Lie(\G)$.
It follows that 
\[
\begin{split}
[X,Y](p)&=\sum_{i,j=1}^\m a_j(p)b_l(p) [X_j,X_l](p) 
+\sum_{i,j=1}^\m a_j(p) X_jb_l(p) X_l(p) \\
&-\sum_{i,j=1}^\m b_l(p)X_la_j(p) X_j(p)\in T_p\Sigma.
\end{split}
\]
This shows that 
\[
\qa{\sum_{j=1}^\m a_j(p) X_j,\sum_{l=1}^\m b_l(p) X_l}(p)=\sum_{i,j=1}^\m a_j(p)b_l(p) [X_j,X_l](p) \in T_p\Sigma.
\]
We could have extended any choice of values $a_j(p)$ and $b_j(p)$
to a neighborhood of $p$ in a $C^1$ smooth way.
This means that we can choose any couple of vectors in $T_p\Sigma$,
consider their associated left invariant vector fields and observe that their
Lie bracket evaluated at the origin is in $dl_{p^{-1}}(T_p\Sigma)\subset H_0\G$,
namely their Lie bracket is in $\cV_1$.
This proves that $\cA_p\Sigma$ is a commutative subalgebra of
$\cV_1$, hence $p$ is a regular point and its homogeneous 
tangent space $A_p\Sigma=\exp\cA_p\Sigma$ is a horizontal
subgroup.
\end{proof}
\bt\label{t:horizC^1Smooth}
If $\Sigma$ is a $C^1$ smooth submanifold such that $T_p\Sigma\subset H_p\G$ 
for every $p\in\Sigma$, then $\Sigma$ is a horizontal submanifold.
\et
\begin{proof}
Let us consider a $C^1$ smooth local chart $\Psi:\Omega\to U$ of the
$C^1$ smooth horizontal submanifold $\Sigma\subset\G$.
Here $\Omega\subset\R^k$ is an open set and $U$ is an open
subset of $\Sigma$. The fact that $\Sigma$ is horizontal
precisely means that
\[
d\Psi(x)(\R^k)\subset H_{\Psi(x)}\G
\]
for a every $x\in\Omega$. These conditions coincides with the validity of contact equations, according to \cite{Mag14}. However, they do not ensure a priori that
a priori the subspace of $\cV_1$ associated to the subspace $d\Psi(x)(\R^k)$
is a commutative subalgebra.
To obtain this information we use \cite[Theorem~1.1]{Mag14}, according
to which $\Psi$ is also differentiable with respect to dilations and the group operation. In particular, this gives the existence of the following limit
\beq\label{eq:limL(v)}
\lim_{t\to0^+}\delta_{1/t}\pa{\Psi(x)^{-1}\Psi(x+tv)}=L_x(v)
\eeq
where $v\in\R^k$ and $L_x:\R^k\to\G$ is a Lie group homomorphism.
We fix now a point $p=\Psi(x_0)\in\Sigma$, observing that  
\[
H_0=L_{x_0}(\R^k)
\]
is a horizontal subgroup of $\G$. We fix a graded basis 
$(e_1,\ldots,e_\q)$ of $\G$, hence we set
\[
\Psi(x)=\sum_{j=1}^\q \psi_j(x) e_j \qandq
L_{x_0}(v)=\sum_{j=1}^\m (L_{x_0})_j(v) e_j.
\]
The Baker-Campbell-Hausdorff formula joined with the limit \eqref{eq:limL(v)} yields
\beq\label{eq:dpsi_jL_j}
d\psi_j(x_0)(v)=(L_{x_0})_j(v) \quad\text{for all $j=1,\ldots,\m$}.
\eeq
The same formula shows that the left invariant vector fields 
$X_1,\ldots,X_\q$ have a special polynomial form. 
Indeed assuming that $X_j(0)=e_j$, with the identification of $\G$ with $T_0\G$,
being $\G$ a linear space, we have
\[
X_j(x)=e_j+\sum_{l=\m+1}^\q a_{jl}(x) e_l,
\]
where $a_{jl}:\G\to\R$ a polynomials.
We have 
\[
\begin{split}
\dpar{\Psi}{x_k}(x)&=\sum_{j=1}^\q \dpar{\psi_j}{x_k}(x) e_j=
\sum_{j=1}^\m \dpar{\psi_j}{x_k}(x) e_j
+\sum_{j=\m+1}^\q \dpar{\psi_j}{x_k}(x) e_j  \\
&=\sum_{j=1}^\m \dpar{\psi_j}{x_k}(x)\, X_j(\Psi(x))
-\sum_{l=\m+1}^\q \sum_{j=1}^\m\dpar{\psi_j}{x_k}(x) a_{jl}(\Psi(x)) e_l
+\sum_{j=\m+1}^\q \dpar{\psi_j}{x_k}(x) e_j  \\
&=\sum_{j=1}^\m \dpar{\psi_j}{x_k}(x)\, X_j(\Psi(x)),
\end{split}
\] 
where in the last equality we have used the fact that 
any $\der_{x_k}\Psi(x)$ must be horizontal, namely
$\der_{x_k}\Psi(x)\in H_{\Psi(x)}\G$ for all $x\in \Omega$.
Applying the definition of algebraically regular point, we consider the left invariant vector fields
\[
Y_k=\sum_{j=1}^\m \dpar{\psi_j}{x_k}(x_0)\, X_j\in\cV_1
\quad \text{for $k=1,\ldots,\m$}.
\] 
Setting $(E_1,\ldots,E_k)$ as the canonical basis of $\R^k$,
by \eqref{eq:dpsi_jL_j} we define
\[
v_k=L_{x_0}(E_k)=\sum_{j=1}^\m (L_{x_0})_j(E_k) e_j=
\sum_{j=1}^\m \dpar{\psi_j}{x_k}(x_0) e_j\in H_0.
\]
Being $H_0$ a horizontal subgroup, it is in particular commutative, therefore
\[
[v_k,v_s]=\sum_{j,l=1}^\m \dpar{\psi_j}{x_k}(x_0)\dpar{\psi_l}{x_s}(x_0)[e_j,e_l]=0.
\]
This proves that 
\[
[Y_k,Y_s]=\sum_{j,l=1}^\m \dpar{\psi_j}{x_k}(x_0)\dpar{\psi_l}{x_s}(x_0)[X_j,X_l]=0,
\]
due to the isomorphism between the Lie product on $\G$
and $\Lie(\G)$, see Proposition~\ref{pr:GAlgGroupOpe}.
We have shown that 
\[
\cA_p\Sigma=\spn\set{Y_1,\ldots,Y_k}
\]
is commutative, hence $\Psi(x_0)$ is an algebraically regular point
and the homogeneous tangent space
$A_p\Sigma=\exp \cA_p\Sigma$ is a horizontal subgroup.
\end{proof}

\br As a consequence of the previous theorem, all $C^1$ smooth
Legendrian submanifolds in the Heisenberg group are horizontal submanifolds.
\er

\section{Transversal points and transversal submanifolds}\label{Sect:transversal}
This section is devoted to a class of submanifolds containing a specific type of algebraically regular point. 
We start with the following definition.
\begin{Def}[Vertical subgroup]\label{d:vertical_subgroup}\rm
We say that a homogeneous subgroup $N\subset\G$ is a {\em vertical subgroup} if
\beq\label{d:N}
N=N_\ell\oplus H^{\ell+1}\oplus\cdots\oplus H^\iota
\eeq
for some $\ell\in\set{1,\ldots,\iota}$ and a linear subspace $N_\ell\subset H^\ell$.
\ed
One may easily observe that any vertical subgroup is also a normal subgroup
of $\G$.
\begin{Def}[Transversal points and transversal submanifolds]\rm
Let $\Sigma\subset\G$ be a $C^1$ smooth submanifold. 
A {\em transversal point} $p$ of $\Sigma$ is an algebraically regular point, whose homogeneous tangent space is a vertical subgroup.
The submanifold $\Sigma$ is {\em transversal} if it contains at least one {\em transversal point}.
\ed
Transversal points can be characterized by their degree. To see this, 
we introduce the following integer valued functions $\ell_{\cdot},r_{\cdot}:\set{1,\ldots,\q}\to \N$. 
For every $\n=1,\ldots,\q$, the inequalities
\begin{equation}\label{d:ell}
\left\{\begin{array}{ll}
\ell_\n=\iota & \text{if }1\le \n\leq \h_\iota\\
\displaystyle\sum_{j=\ell_\n+1}^\iota \h_j <\n\leq \sum_{j=\ell_\n}^\iota \h_j\quad & \text{if } \h_\iota<\n\le \q
\end{array}\right.
\end{equation}
uniquely define the integer $\ell_\n\in\set{1,\ldots,\iota}$. Thus, we also define
\beq\label{d:rn}
\rr_\n:=\left\{\begin{array}{ll} 
\n & \text{if }  1\le \n\le \h_\iota\\  
\ds \n-\sum_{j=\ell_\n+1}^\iota  \h_j & \text{if }  \h_\iota<\n\le \q
\end{array}\right.
\end{equation}
for every $\n=1,\ldots,\q$, where $\rr_\n\ge1$. We finally set  
\beq\label{d:Q_n}
Q_\n=\ell_\n\,\rr_\n +\sum_{j=\ell_\n+1}^\iota j\,\h_j\,,
\eeq
where the sum is understood to be zero only in the case $1\le \n\le\h_\iota$,
that is $\ell_\n=\iota$. 

If $N\subset\G$ is an $\n$-dimensional vertical subgroup of the form \eqref{d:N},
it is not difficult to observe that the degree at every point of $N$ equals $Q_\n$ given in \eqref{d:Q_n} with 
\[
\dim N_\ell=\rr_\n\qandq \ell=\ell_\n.
\]
From formula \eqref{eq:n_d_Sigma}, taking into account Proposition~\ref{pr:translation}, it is not difficult
to realize that
\beq\label{eq:Q_nMax}
Q_\n=\max_{\Sigma\in\cS_\n(\G)} d(\Sigma).
\eeq
The set $\cS_\n(\G)$ denotes the family of $\n$-dimensional submanifolds of class $C^1$ 
that are contained in $\G$. The integer $d(\Sigma)$ is the degree of $\Sigma$
introduced in \eqref{d:degSigma}.

We are now in the position to prove the following characterization.
\bpr\label{pr:transv_maxdeg}
A point $p$ of an $\n$-dimensional $C^1$ smooth submanifold $\Sigma\subset\G$ is transversal if and only if $d_\Sigma(p)=Q_\n$.
\epr
\begin{proof}
If $p$ is transversal, using left translations we may assume that it
coincides with the origin. Using the coordinates of Theorem~\ref{t:specialcoord}
and applying formula \eqref{eq:homtansp}, the fact
that $A_0\Sigma$ is a transversal subgroup gives
\beq\label{eq:A0SigmaTransversal}
A_0\Sigma=\spn\set{e_{\m_{\ell-1}+1},\ldots,e_{\m_{\ell-1}+\rr},e_{\m_\ell+1},e_{\m_\ell+2},\ldots,e_\q}.
\eeq
We have assumed that $A_0\Sigma$ has the form of
\eqref{d:N} and $\dim N_\ell=\rr$.
From \eqref{eq:n_d_Sigma} we immediately get
\[
d_\Sigma(0)=\rr \ell +\sum_{j=\ell+1}^\iota j\,\h_j\,,
\]
where it must be $\rr=\rr_\n$ and $\ell=\ell_\n$, from \eqref{d:ell}
and \eqref{d:rn}. We have proved that $d_\Sigma(0)=Q_\n$.
It is not restrictive to assume $p=0$ also for the converse implication.
In this case we only know that $d_\Sigma(0)=Q_\n$.
Again, referring to the special coordinates of Theorem~\ref{t:specialcoord}
and the corresponding formula \eqref{eq:n_d_Sigma}, 
the previous equality implies that
\beq\label{eq:alpha_jtransv}
\left\{\begin{array}{ll} 
\alpha_j=0 &  \text{if }j<\ell_\n\\
\alpha_j=\rr_\n &  \text{if } j=\ell_\n \\ 
\alpha_j=\h_j  &  \text{if } j>\ell_\n 
\end{array}\right..
\eeq
Applying formula \eqref{eq:homtansp}, we have shown that
$A_0\Sigma$ must be a vertical subgroup.
\end{proof}
%
%
%
\begin{comment} 
If it were known the weaker statement $d(\Sigma)\leq Q_\n$
implies $\dim_\cH(\Sigma)\le Q_\n$, then 
we would have
\[
Q_\n=\max_{\Sigma\in\cS_\n(\G)} \dist_\cH(\Sigma)
\]
where $\dim_\cH(A)$ denotes the Hausdorff dimension of 
a subset $A\subset\G$ with respect to any homogeneous distance of $\G$.
\end{comment}
\br
The previous proposition and formula \eqref{eq:Q_nMax} show that
any transveral point has maximum degree.
\er
We finally observe that with the assumptions of Corollary~\ref{c:matriceCcont}, when $0\in\Sigma$ is transversal, the matrix $C$ of \eqref{eq:matriceCcont} becomes
\begin{equation}\label{eq:matriceCtransv}
C=\left(\begin{array}{c|c|c|c|c}
\vdots  & \vdots & \cdots & \cdots & \ast\\
\hline
\ast  & \ast & \cdots & \cdots & \ast\\
\hline
I_{\rr_\n}+o(1)  & o(1) & \cdots & \cdots & 0  \\
o(1)  & \ast & \cdots & \cdots & \ast\\
\hline
o(1)   & I_{\h_{\ell_\n+1}}+o(1) & o(1) & \cdots & o(1)\\
\hline
\vdots  & o(1) & \ddots & \cdots & \vdots\\
\hline
\vdots  & \vdots & \cdots & \ddots & o(1) \\
\hline
\vdots  & \vdots & \cdots & o(1)& Id_{\h_\iota}+o(1)
\end{array}\right)\,,
\end{equation}
where $\rr_\n$ and $\ell_\n$ are defined in \eqref{d:ell} and \eqref{d:rn}, respectively. Indeed Proposition~\ref{pr:transv_maxdeg} shows that
$d_\Sigma(0)=Q_\n$ holds and this implies the validity of 
the conditions \eqref{eq:alpha_jtransv}.
%
%
%

%
%
%
%
%
%
\section{Blow-up at points of maximum degree}
%
%
%
%
%
%
%

The general structure of \eqref{eq:matriceCcont} will be also used in the proof of Theorem~\ref{t:LocExpSurf}. The following theorem represents the main technical tool of this paper.
%
%
%
%
\begin{The}[Blow-up]\label{t:LocExpSurf}
Let $\Sigma\subset\G$ be a $C^1$ smooth submanifold of topological dimension $\n$ and degree $\NN$. Let $p\in\Sigma$ be an algebraically regular point of maximum degree $\NN$ and let $A_p\Sigma$ be
its homogeneous tangent space. 
We assume that one of the following assumptions holds:
\begin{enumerate}
\item
$p$ is a horizontal point, 
\item
$\G$ has step two, 
\item
$\Sigma$ is a one dimensional submanifold,
\item
$p$ is a transversal point.
\end{enumerate}
For the translated submanifold 
\[
\Sigma_p=p^{-1}\Sigma,
\]
we introduce the $C^1$ smooth homeomorphism $\eta:\R^\n\to\R^\n$ by
\begin{equation}\label{eq:etat}
\eta(t)=\bigg(\frac{|t_1|^{b_1}}{b_1}\sgn(t_1),\ldots,\frac{|t_p|^{b_\n}}{b_\n}\sgn(t_\n)\bigg),
\end{equation}
where each $b_i$ is defined in \eqref{e:degSigma}.  If $\psi$ denotes
the mapping of Theorem~\ref{t:specialcoord} applied to the 
translated submanifold $\Sigma_p$, we define the $C^1$ smooth mapping 
\beq\label{d:Gammaeta}
\Gamma=\psi\circ\eta
\eeq
and we define the subset of indexes $I\subset\set{1,\ldots,\q}$ such that
\beq\label{e:Indexes}
A_0\Sigma_p=\spn\set{e_l: l\in I}=\spn\set{e_1,\ldots,e_{\alpha_1},e_{\m_1+1},\ldots,e_{\m_1+\alpha_2},\ldots,e_{\m_{\iota-1}+\alpha_\iota}},
\eeq
then the following local expansion holds
\beq\label{eq:estim}
\Gamma_s(t)=\left\{\begin{array}{ll} 
\frac{|t_{s-\m_{d_s-1}+\mu_{d_s-1}}|^{d_s}}{d_s}
\sgn(t_{s-\m_{d_s-1}+\mu_{d_s-1}}) & \mbox{if $s\in I$} \\
o(|t|^{d_s}) & \mbox{if $s\notin I$}
\end{array}\right..
\eeq
\end{The}
\begin{proof}
Taking into account Proposition~\ref{pr:translation}, the translated manifold
$\Sigma_p$ has the same degree of $\Sigma$, therefore
\[
d_{\Sigma_p}(0)=d_\Sigma(p)=\NN.
\]
Thus, the origin $0\in\Sigma_p$ is a point of maximum degree for $\Sigma_p$.
By Theorem~\ref{t:specialcoord}, following its notation, there exists a special graded
basis $(e_1,\ldots,e_\q)$, along with a $C^1$ smooth embedding
$\Psi:U\to\Sigma_p$ with $\Psi(0)=0\in\G$ and 
\beq\label{d:defPsi}
\Psi(y)=\sum_{j=1}^\q \psi_j(y) e_j,
\eeq
that satisfies both conditions \eqref{e:matriceC} and \eqref{eq:psicoord}.
For our purposes, it is not restrictive to assume that $\Psi$ is a $C^1$ diffeomorphism.
We also introduce the basis $(X_1,\ldots,X_\q)$ of $\Lie(\G)$ such that 
$X_i(0)=e_i$ for all $i=1,\ldots,\q$ and consider graded coordinates $(x_i)$ of a point $p$,
such that $p=\sum_{i=1}^\q x_i e_i\in\G$. With respect to these coordinates, the vector fields  
\begin{equation}\label{X_i}
X_i=\sum_{l=1}^n a_i^l\,\der_{x_l}
\end{equation}
satisfy the following conditions 
\begin{equation}\label{eq:ailX}
a_i^l=\left\{\begin{array}{ll}
\delta_i^l & 	d_l\leq d_i \\
\mbox{\scriptsize  polynomial of homogeneous degree $d_l-d_i$} & d_l>d_i
\end{array}\right..
\end{equation}
The homogeneity here refers to intrinsic dilations of the group, namely
\beq\label{e:homoga_il}
a_i^l(\delta_rx)=r^{d_l-d_i}a_i^l(x)
\eeq
for all $r>0$ and $x\in\G$, see e.g.\ \cite{Stein93}.
We can further assume that there exists $c_1>0$ sufficiently small such that the domain $U$ of the above diffeomorphism $\Psi$ is defined on $(-c_1,c_1)^\n$.
The continuous functions $C^s_i$ in \eqref{e:derjXj} can be assumed to be defined on
a common interval $(-c_1,c_1)$, where $C^s_i(0)$ is the $(s,i)$ entry of the matrix \eqref{e:matriceC}. For the sequel, it is convenient to recall formula \eqref{e:derjXj} here
\begin{equation}\label{derjgamma}
(\der_i\psi)(y)=\sum_{s=1}^\q C^s_i(\psi(y)) X_s(\psi(y))\quad \mbox{for all}\quad i=1,\ldots,\n
\end{equation}
for all $y\in(-c_1,c_1)^\n$. Thus, from \eqref{eq:etat} and \eqref{d:Gammaeta} we have the partial derivatives 
\begin{equation}\label{eq:partjG}
\der_{t_i}\Gamma(t)=|t_i|^{b_i-1}\,(\der_i\psi)(\eta(t))=|t_i|^{b_i-1}\sum_{l,s=1}^\q C^l_i(\Gamma(t))\, a_l^s(\Gamma(t))\,\der_{x_s}
\end{equation}
for all $i=1,\ldots,\n$, where we have used both \eqref{X_i} and \eqref{derjgamma}.

The main point is to prove by induction the validity of the following statement. For each $j=1,\ldots,\iota$, if $0\le\alpha_j<\h_j$ there holds
\begin{equation}\label{eq:induct0}
\Gamma_s(t)=o(|t|^j) \qs{for} \m_{j-1}+\alpha_j<s\leq \m_j.
\end{equation}
Notice that in the case $\alpha_j=\h_j=\m_j-\m_{j-1}$ there is nothing to prove and the statement is automatically satisfied.

Let us first establish the case $j=1$.
If $\alpha_1=0$, in all of the four assumptions where this condition applies, 
we have $b_i\ge2$ for each $i=1,\ldots,\n$,
therefore \eqref{eq:partjG} gives 
\[
\nabla\Gamma_s(0)=0 \quad \mbox{for every $s=1,\ldots,\q$}.
\] 
If $0<\alpha_1<\m_1$ and $\alpha_1<s\le\m_1$, again in all four assumptions, 
due to \eqref{e:matriceC}, we get
\[
\der_{x_i}\Gamma_s(0)=0\qs{for all $i=1,\ldots,\alpha_1$.}
\]
In view of \eqref{eq:partjG}, the previous equalities extend to all $i=\alpha_1+1,\ldots,\n$,
being $b_i\ge2$. In both cases, the vanishing of $\Gamma_s(0)$ and $\nabla\Gamma_s(0)$ 
for any $s=\alpha_1+1,\ldots,\m_1$ and in all of our four assumptions proves our inductvie assumption \eqref{eq:induct0} for $j=1$.

Now, we assume by induction the validity of \eqref{eq:induct0} for all $j=1,\ldots,k-1$, where $2\le k\le\iota$. We wish to prove this formula for $j=k$, in the nontrivial case $0\le\alpha_k<\h_k$.
Let us write the general formula \eqref{eq:partjG} for partial derivatives 
\begin{equation}\label{eq:derGammas}
\der_{t_i}\Gamma_s(t)=|t_i|^{b_i-1}\LLs C^s_i(\Gamma(t))+\sum_{l: d_l<d_s} C^l_i(\Gamma(t))\, a_l^s(\Gamma(t))\RRs,
\end{equation}
where $s=\m_{k-1}+1,\ldots,\m_k$. 
We consider the following possibilities: 
\[
b_i<k, \quad b_i=k\quad \mbox{ and }\quad b_i>k.
\] 
Let us begin with the case $b_i<k$. 
If $\alpha_k>0$ and consider $\m_{k-1}+\alpha_k<s\leq \m_k$, then
the structure of \eqref{eq:matriceCcont} and the fact that $b_i<k$ yield
\beq\label{eq:C_i^s}
C^s_i\equiv0. 
\eeq
If $\alpha_k=0$ and the fourth assumption holds, then the special structure of $C$,
see \eqref{eq:matriceCtransv}, implies that $\alpha_j=0$ for all $j=1,\ldots,k-1$. This gives  
$b_i\ge k+1$ for all $i=1,\ldots,\n$. Taking into account the form \eqref{eq:etat} of $\eta$
and the composition \eqref{d:Gammaeta}  we clearly have
\[
\Gamma_s(t)=O(|t|^{k+1})=o(|t|^k)
\]
for all $s=1,\ldots,\q$ and in particular \eqref{eq:induct0} is established.
If $\alpha_k=0$ and the first assumption holds, then the form \eqref{eq:matriceCcontHor} always gives 
\beq\label{eq:C_i^sall}
C^s_i\equiv0 \quad\mbox{for}\quad \m_1\le \m_{k-1}<s\le \q\qandq i=1,\ldots,\n. 
\eeq
If $\alpha_k=0$ and the second assumption holds, then $\iota=2$ and we only have
the case $k=2$, namely $\alpha_2=0$. From the form \eqref{eq:matriceCcontiota2}, 
then 
\beq\label{eq:C_i^s2step}
C^s_i\equiv0 \quad\mbox{for}\quad \m_1<s\le \q\qandq i=1,\ldots,\n. 
\eeq
\begin{comment}
It seems that the argument in the second assumption can extend to the
general case, with no assumption on the matrix $C(0)$. In fact, also in
the general case the assumption $\alpha_k=0$ implies that on the block
on the left side of the diagonal block of degree $k$, we have all zeros.
Compare with the form of \eqref{eq:matriceCcont}.
\end{comment}
If $\alpha_k=0$ and the third assumption holds, then $\n=1$ and the
condition $b_i<k$ gives
\beq\label{e:b1N}
b_1=\NN<k=d_s\quad \mbox{for all  $s=\m_{k-1}+1,\ldots,\m_k$},
\eeq
so that the form \eqref{eq:matriceCcontn1} yields
\beq\label{eq:C_i^D=1}
C^s_1\equiv0 \quad\mbox{for}\quad \m_{k-1}<s\le \m_k. 
\eeq
We are interested in the case $s=\m_{k-1}+1,\ldots,\m_k$  and $i=1,\ldots,\mu_{k-1}$, 
therefore the vanishing of $C_i^s$ joined with \eqref{eq:derGammas} gives 
\beq\label{e:T123}
\begin{split}
\der_{t_i}\Gamma_s(t)&=|t_i|^{b_i-1}\sum_{l: d_l<k} C^l_i(\Gamma(t))\, a_l^s(\Gamma(t)) \\
&=|t_i|^{b_i-1}\sum_{l: d_l<b_i<k} C^l_i(\Gamma(t))\, a_l^s(\Gamma(t))+
|t_i|^{b_i-1}\sum_{l: d_l=b_i<k} C^l_i(\Gamma(t))\, a_l^s(\Gamma(t)) \\
&+|t_i|^{b_i-1}\sum_{l: b_i<d_l<k} C^l_i(\Gamma(t))\, a_l^s(\Gamma(t))=T_1+T_2+T_3.
\end{split}
\eeq
We have denoted by $T_1,T_2$ and $T_3$ the first, second and third addend, respectively. To study $T_1$, we use the graph form of $\psi$ given by \eqref{eq:psicoord}.
In fact, whenever $\alpha_j>0$ we have the identity
\beq\label{e:bsmj-1}
b_{s-\m_{j-1}+\mu_{j-1}}=j
\eeq
for $\m_{j-1}<s\leq \m_{j-1}+\alpha_j$ and $j=1,\ldots,\iota$, hence \eqref{eq:etat} 
and \eqref{d:Gammaeta} yield
\beq\label{e:specialgraphform}
\Gamma_s(t)=\frac{|t_{s-\m_{j-1}+\mu_{j-1}}|^j}{j}
\sgn(t_{s-\m_{j-1}+\mu_{j-1}})=\frac{|t_{s-\m_{j-1}+\mu_{j-1}}|^{d_s}}{d_s}
\sgn(t_{s-\m_{j-1}+\mu_{j-1}}).
\eeq
Each polynomial $a^s_l$ in the sum of $T_1$ has homogeneous degree $k-d_l$, hence 
it does not depend on the variables $x_i$, with $i>\m_{k-1}$. 
As a consequence of \eqref{e:specialgraphform}, for all $s=\m_{k-1}+1,\ldots,\m_k$, the homogeneity \eqref{e:homoga_il} of $a_l^s$, when joined with our inductive assumption also implies that
\[
a^s_l(\Gamma(t))=a^s_l(\Gamma_1(t),\ldots,\Gamma_{\m_{k-1}}(t))=O(|t|^{k-d_l}).
\]
\begin{comment}
This can be proved as follows
\[
a_l^s(\delta_{|t|}\pa{\delta_{1/|t|}\Gamma(t)})=|t|^{d_s-d_l} a_l^s\pa{\delta_{1/|t|}\Gamma(t)}
\]
since $\delta_{1/|t|}\Gamma(t)=O(1)$ as $t\to 0$.
\end{comment}
This immediately shows that $T_1(t)=O(|t|^k)=o(|t|^{k-1})$.
We now consider the second addend
\[
T_2(t)=|t_i|^{b_i-1}\sum_{l: d_l=b_i<k} C^l_i(\Gamma(t))\, a_l^s(\Gamma(t))
\]
and set $j=b_i$. The conditions $d_l=b_i<k$ give
\beq\label{e:ranges_il}
\mu_{j-1}<i\le\mu_j\qandq \m_{j-1}<l\le\m_j.
\eeq
We consider the general case where $0\le\alpha_k<\h_k$.
Since $b_i=j$ we have $\alpha_j>0$, therefore taking into account \eqref{eq:matriceCcont}, 
for $\m_{j-1}<l \le \m_{j-1}+\alpha_j$ it follows that 
\beq\label{e:deltaCi^l}
C_i^l=\delta_{i-\mu_{j-1}}^{l-\m_{j-1}}+o_i^l(1)
\eeq
where $o_i^l(1)$ vanish at the origin. When $\m_{j-1}+\alpha_j<l \le \m_j$, we have
\[
C_i^l=o_i^l(1)
\]
and $o_i^l(1)$ vanish at zero. In view of \eqref{e:deltaCi^l}, for $i$ and $l$ in the ranges 
\eqref{e:ranges_il}, we set
\[
\m_{j-1}<l_{ij}:=i-\mu_{j-1}+\m_{j-1}\le \m_{j-1}+\alpha_j,
\]
therefore we obtain the expression
\beq\label{e:T_2new}
T_2(t)=|t_i|^{b_i-1}\LLs\sum_{\substack{l: d_l=b_i<k\\ l\neq l_{ij} }} o^l_i(1)\, a_l^s(\Gamma(t)) +a_{l_{ij}}^s(\Gamma(t))\RRs.
\eeq
Arguing as before, formulae \eqref{e:specialgraphform} and the inductive assumption
imply that
\[
|t_i|^{b_i-1} a_l^s(\Gamma(t))=|t_i|^{j-1} O(|t|^{k-d_l})=|t_i|^{j-1} O(|t|^{k-j})=
O(|t|^{k-1}).
\]
It follows that 
\beq\label{eq:T_2}
T_2(t)=o(|t|^{k-1}) +|t_i|^{b_i-1}a_{l_{ij}}^s(\Gamma(t)).
\eeq
The behavior of the second addend in the previous equality requires a special study,
that precisely relies on the group structure that is assumed on $A_0\Sigma_p$. Taking into account the definition of the set of indexes $I$ defined through \eqref{e:Indexes}, in view of \cite[Lemma~2.5]{Mag13Vit}, if the group operation is given by the polynomial formula
\[
xy=x+y+Q(x,y) 
\]
with respect to our fixed graded coordinates, then the polynomial $Q_s$, with $s\notin I$,
is given by the formula
\[
Q_s(x,y)=\sum_{v: d_v<k, v\notin I} x_v R_{sv}(x,y) +y_v U_{sv}(x,y).
\]
Both polynomials $R_{sv}$ and $U_{sv}$ have homogeneous of degree  $k-d_v$.
Since we have $\m_{j-1}< l_{ij}\le\m_{j-1}+\alpha_j$, the condition $l_{ij}\in I$ gives
\[
\dpar{Q_s}{y_{l_{ij}}}(x,0)=a^s_{l_{ij}}(x)=\sum_{v: d_v\le k-j, v\notin I}x_v\dpar{R_{sv}}{y_{l_{ij}}}(x,0),
\] 
where we have used the relationship between left invariant vector fields and group operation,
along with the fact that $v\neq l_{ij}$ for all $v\notin I$.
As we have already observed, $a_{l_{ij}}^s$ only depends on $(x_1,\ldots,x_{\m_{k-1}})$ and by our inductive assumption \eqref{eq:induct0}
\[
\Gamma_v(t)=o_v(|t|^{d_v})\qs{whenever} \mbox{$d_v<k$ and $v\notin I$}.
\]
Precisely, for all of these $v's$, we have $o_v(|t|^{d_v})/|t|^{d_v}\to 0$ as $t\to0$ and there holds
\[
a_{l_{ij}}^s(\Gamma(t))=\sum_{v: d_v\le k-j, v\notin I}o_v(|t|^{d_v})\, \dpar{R_{sv}}{y_{l_{ij}}}(\Gamma(t),0),
\]
Again, the inductive assumption gives $\der_{y_{l_{ij}}}R_{sv}(\Gamma(t),0)=O(|t|^{k-d_v-j})$, that is
\[
o_v(|t|^{d_v})\, \dpar{R_{sv}}{y_{l_{ij}}}(\Gamma(t),0)=o(|t|^{k-j}),
\]
therefore $a^s_{l_{ij}}(\Gamma(t))=o(|t|^{k-j})$. We have finally proved that
\[
T_2(t)=o(|t|^{k-1}).
\]
The treatment of the addend 
\[
T_3=|t_i|^{b_i-1}\sum_{l: b_i<d_l<k} C^l_i(\Gamma(t))\, a_l^s(\Gamma(t))
\]
in \eqref{e:T123} strongly relies on our special
four assumptions. Without these assumptions, it is not clear whether for instance the factors $C_i^l(\Gamma(t))$ for $b_i<d_l<k$ behave like $o(|t|^{d_l-b_i})$, since $C_i^l$ are only continuous.

If the first assumption holds, then the special form \eqref{eq:matriceCcontHor} of $C$
immediately proves that there cannot exist nonvanishing coefficients $C_i^l$ 
whenever $b_i<d_l$, hence $T_3\equiv0$.
If the second assumption holds, then $1\le b_i<d_l<k$ implies
$k\ge 3$, that conflicts with the 2-step assumption on $\G$, therefore $T_3\equiv0$.
%
\begin{comment}
First ideas for the general case. From \eqref{eq:matriceCcont}, we distinguish the case where 
$C_i^l$ vanishes for $m_{d_l-1}+\alpha_l<l\le m_{d_l}$ and the case
where it does not since $m_{d_l-1}<l\le m_{d_l-1}+\alpha_l$.
Is there any hope that for these $l'$s we are able to establish a better behavior of 
$a_l^s(\Gamma(t))$, due to the assumption that $A_p\Sigma$ is a subgroup?
Another more promising technique is to write explicitly the coefficients
$C_i^l$ for the entries  $m_{d_l-1}<l\le m_{d_l-1}+\alpha_l$, where they do
not vanish a priori. In fact, these coefficients can be written in terms of the
partial derivatives of $\Gamma$.
\end{comment}
%
If the third assumption holds, then $\n=1=i$ and \eqref{e:b1N} gives
\[
b_1=\NN<d_l
\]
that joined with the special form \eqref{eq:matriceCcontn1} gives 
$C_1^l\equiv0$, therefore $T_3\equiv0$ also in this case.
In the fourth assumption, where $p$ is a transveral point, we consider
the integer  $\ell_\n$ defined in \eqref{d:ell}.
By definition \eqref{d:rn}, according to \eqref{eq:matriceCtransv}, we have
\[
\alpha_{\ell_\n}=\rr_n\ge 1\qandq b_i\ge \ell_\n,
\]
therefore $k>\ell_\n$. This implies that $\alpha_k=\h_k$, hence the 
inductive assumption is automatically satisfied.
Collecting all of the previous cases, we conclude that in any of the four assumptions
for $b_i<k$, we have that either the inductive assumption \eqref{eq:induct0} is satisfied or
we have
\[
\der_{t_i}\Gamma_s(t)=o(|t|^{k-1}).
\]
In the case $b_i=k$, then $\alpha_k>0$ and the condition $\m_{k-1}+\alpha_k<s\le\m_k$ joined with the form of \eqref{eq:matriceCcont} yields
\[
C_i^s(\Gamma(t))=o(1),
\]
therefore \eqref{eq:derGammas} gives
\[
\der_{t_i}\Gamma_s(t)=|t_i|^{k-1}\LLs o(1)+\sum_{l: d_l<k} C^l_i(\Gamma(t))\, a_l^s(\Gamma(t))\RRs.
\]
In the previous sum the condition $d_s=k>d_l$ yields $a^s_l(0)=0$, therefore also in the
case $b_i=k$ we have
\[
\der_{t_i}\Gamma_s(t)=o(|t|^{k-1}).
\]
When $b_i>k$, there obviously holds 
\beqas
\der_{t_i}\Gamma_s(t)&=&|t_i|^{b_i-1}\LLs C^s_i(\Gamma(t))+\sum_{l: d_l<d_s} C^l_i(\Gamma(t))\, a_l^s(\Gamma(t))\RRs \\
&=&|t_i|^{b_i-1} O(1)=o(|t|^{k-1}).
\eeqas
Joining all the previous results, it follows that $\nabla \Gamma_s=o(|t|^{k-1})$, hence
\[
\Gamma_s(t)=o(|t|^k),
\]
proving the induction step. This proves our claim \eqref{eq:estim}.
\end{proof}

%
%
%
%
%
%
\section{Measure theoretic area formula in homogeneous groups}\label{sect:meastheo}
%
%
%
%
%
%
%
We introduce some preliminary results and notions that will
be needed in the next sections.
The symbol $\G$ always denotes a homogeneous group equipped with a homogeneous distance $d$.
%
%
\subsection{Differentiation of measures in homogeneous groups}
We denote by $\cF_b$ the family of closed balls in $\G$ having positive radius.
The properties of the homogeneous distance give $\diam(B(x,r))=2r$ for all $x\in\G$ and $r>0$. As a consequence, if $\mu:\cP(X)\to[0,+\infty]$ is a measure, then 
one easily realizes that the family of sets
\beq\label{eq:Smu_zetab}
\cS_{\mu,\zeta_{b,\alpha}}=\cF_b\sm\{S\in\cF_b: \zeta_{b,\alpha}(S)=\mu(S)=0\ \mbox{or}\ \zeta_{b,\alpha}(S)=\mu(S)=+\infty \}=\cF_b,
\eeq
where we have defined 
\[
\zeta_{b,\alpha}:\cF_b\to[0,+\infty),\quad \zeta_{b,\alpha}(S)=\frac{\diam(S)^\alpha}{2^\alpha}.
\]
%
%
%
\begin{Def}[Carath\'eodory construction]\label{def:sphericalH}\rm 
Let $\cF\subset\cP(\G)$ denote a nonempty family of closed subsets
and fix $\alpha>0$. If $\delta>0$ and $E\subset\G$, we define
\begin{equation}\label{d:phialpha}
\phi^\alpha_\delta(E)=\inf \bigg\lbrace\sum_{j=0}^\infty \frac{\diam(B_j)^\alpha}{2^\alpha}: E\subset \bcup_{j\in\N} B_j ,\, \diam(B_j)\le\delta,\, B_j\in\cF \bigg\rbrace\,,
\end{equation}
where the diameter $\diam B_j$ is computed with respect to the distance $d$ on $\G$.
If $\cF$ coincides with the family of closed balls $\cF_b$, then we set
\beq\label{d:S_0^alpha}
\cS^\alpha_0(E)=\sup_{\delta>0}\phi^\alpha_\delta(E)
\eeq
to be the {\em $\alpha$-dimensional spherical measure} of $E$.
In the case $\cF$ is the family of all closed sets and $k\in\set{1,2,\ldots,\q-1}$, we define the Hausdorff measure
\beq
\cH_{|\cdot|}^k=\cL^k(\set{x\in\G: |x|\le 1})  \sup_{\delta>0}\phi^k_\delta(E)
\eeq
where $\cL^k$ denotes the Lebesgue measure 
and $|\cdot|$ is the norm arising from the fixed graded scalar product on
$\G$.
\end{Def}

Observing that $\cF_b$ covers any subset finely, according to the terminology in
\cite[2.8.1]{Federer69} and that condition \eqref{eq:Smu_zetab} holds, we can 
apply Theorem~11 in \cite{Mag30} to the metric space $(\G,d)$,
establishing the following result.
\begin{The}\label{the:metricspherical}
Let $\alpha>0$ and let $\mu$ be a Borel regular measure over $\G$ 
such that there exists a countable open covering of $\G$, whose elements have $\mu$ finite measure. If $B\subset A\subset \G$ are Borel sets, then $\theta^\alpha(\mu,\cdot)$ is Borel on $A$. In addition, if $\cS^\alpha(A)<+\infty$
and $\mu\res A$ is absolutely continuous with respect to $\cS^\alpha\res A$, then we have
\begin{equation}\label{eq:spharea}
 \mu(B)=\int_B \theta^\alpha(\mu,x)\,d\cS_0^\alpha(x)\,.
\end{equation}
\end{The}
The {\em spherical Federer density} $\theta^\alpha(\mu,\cdot)$ in \eqref{eq:spharea}
was introduced in \cite{Mag30}. We will use its explicit representation
\begin{equation}\label{eq:FDens} 
\theta^\alpha(\mu,x)=\inf_{\ep>0}\; \sup\bigg\{\frac{2^\alpha \mu(\B)}{\diam(\B)^\alpha}: x\in\B\in\cF_b,\ \diams\B<\ep\bigg\}\,.
\end{equation}

\subsection{Intrinsic measure and spherical factor}\label{sect:intrinsic_measure}
The next definition introduces the intrinsic measure associated to a submanifold
in a homogeneous group, see \cite{Mag13Vit}. For hypersurfaces in Carnot
groups this measure is precisely the h-perimeter measure with respect to the sub-Riemannian
structure of the group.
\begin{Def}[Intrinsic measure]\label{d:SRmeasure}\rm
Let $\Sigma\subset\G$ be an $\n$-dimensional submanifold of class  $C^1$ and degree $\NN$.
We consider our fixed graded left invariant Riemannian metric $g$ on $\G$. To present a coordinate free version of this measure, we fix an auxiliary Riemannian metric $\tilde g$ on $\G$.
Let $\tau_\Sigma$ be a $\tilde g$-unit tangent $\n$-vector field on $\Sigma$, namely,
\[
\|\tau_\Sigma(p)\|_{\tilde g}=1 \quad \text{for each} \quad p\in\Sigma.
\]
We consider its corresponding $\NN$-tangent $\n$-vector field, defined as follows
\beq\label{d:tauSigmaN}
\tau_{\Sigma,\NN}^{\tilde g}(p):=\pi_{p,\NN}(\tau_{\Sigma}(p))\quad\text{for each} \quad p\in\Sigma.
\eeq
Then we define the {\em intrinsic measure} of $\Sigma$ in $\G$ as follows
\beq\label{d:intmeas}
\mu_{\Sigma}=\|\tau_{\Sigma,\NN}^{\tilde g}\|_g\; \sigma_{\tilde g},
\eeq
where $\sigma_{\tilde g}$ is the $\n$-dimensional Riemannian 
measure induced by $\tilde g$ on $\Sigma$.
This can be also seen as the $\n$-dimensional Hausdorff measure 
with respect to the Riemannian distance induced by $\tilde g$
and restricted to $\Sigma$.
\ed
\br\label{r:degreeTangent}
By definition of pointwise degree \eqref{d:degree_point}, we realize that under the assumptions of Definition~\ref{d:SRmeasure} a point $p\in\Sigma$ has maximum degree $\NN$ if and only if 
\[
\tau_{\Sigma,\NN}^{\tilde g}(p)=\pi_{p,\NN}(\tau_\Sigma(p))\neq0,
\]
as it follows from the definition of pointwise $\NN$-projection, see \eqref{d:projM}.
\er
\bpr\label{pr:intmeaslocalchart}
If $H\subset\R^\n$ is an open subset and $\Phi:H\to \G$ is a $C^1$ smooth local
chart for an $\n$-dimensional $C^1$ smooth submanifold $\Sigma$ of degree $\NN$, then
\beq\label{e:intmeaslocalchart}
\mu_\Sigma\lls \Phi(H)\rrs=\int_H \|\pi_{\Phi(y),\NN}\lls \der_{y_1}\Phi(y)\wedge\cdots\wedge\der_{y_\n}\Phi(y)\rrs\|_g\,dy.
\eeq
\epr
\begin{proof}
By our local chart, using \eqref{d:tauSigmaN} we can write 
\[
\tau^{\tilde g}_{\Sigma,\NN}(\Phi(y)):=\frac{\pi_{\Phi(y),\NN}\lls \der_{y_1}\Phi(y)\wedge\cdots\wedge\der_{y_\n}\Phi(y)\rrs}{\| \der_{y_1}\Phi(y)\wedge\cdots\wedge\der_{y_\n}\Phi(y)\|_{\tilde g}},
\]
therefore the integral
\[
\int_{\Phi(H)} \|\tau^{\tilde g}_{\Sigma,\NN}(p)\|_g\, d\sigma_{\tilde g}(p),
\]
after the standard change of variables $p=\Phi(y)$, becomes equal to
\[
\int_H \left\|\frac{\pi_{\Phi(y),\NN}\lls \der_{y_1}\Phi(y)\wedge\cdots\wedge\der_{y_\n}\Phi(y)\rrs}{\| \der_{y_1}\Phi(y)\wedge\cdots\wedge\der_{y_\n}\Phi(y)\|_{\tilde g}}\right\|_g
\| \der_{y_1}\Phi(y)\wedge\cdots\wedge\der_{y_\n}\Phi(y)\|_{\tilde g}\, dy,
\]
therefore concluding the proof of \eqref{e:intmeaslocalchart}.
\end{proof}
The relationship between intrinsic meausure and spherical measure requires 
some geometric constants that can be associated to the homogeneous distance
that defines the spherical measure. These constants may change, depending on the
sections we consider of the metric unit ball.
%
%
%
%
\begin{Def}[Spherical factor]\label{d:sphericalfactor}\rm
Let $S\subset\G$ a linear subspace and consider a fixed homogeneous 
distance $d$ on $\G$. If $|\cdot|$ denotes our fixed graded scalar product on $\G$,
then the {\em spherical factor} of $d$, with respect to $S$, is the number
\[
\beta_d(S)=\max_{d(u,0)\le1} \cH^\n_{|\cdot|}\lls\B(u,1)\cap S\rrs,
\]
where $\B(u,1)=\set{v\in\G: d(v,u)\le1}$.
\ed
\section{Proof of the upper blow-up theorem}
This section is devoted to the proof of the upper blow-up for the intrinsic measure of submanifolds (Theorem~\ref{t:UpBlwC}). 

\begin{proof}[Proof of Theorem~\ref{t:UpBlwC}]
We consider the special coordinates obtained in Theorem~\ref{t:specialcoord}
for the translated manifold $\Sigma_p=p^{-1}\Sigma$.
This assumption is possible by Proposition~\ref{pr:translation},
since algebraic regularity along with the first and the fourth assumptions
are automatically transferred to the origin of $\Sigma_p$.
We then follow notations of Theorem~\ref{t:LocExpSurf}.
In some parts of the proof the identification of $\G$ with $\R^\n$ with 
respect to the above mentioned coordinates will be understood.
For instance, the algebraic tangent space
$A_0\Sigma_p$ defined in \eqref{eq:homtansp}
equals $A_p\Sigma$ by Proposition~\ref{pr:translation}
and it can be also identified with $\R^\n$.

Let $\Psi$ be defined as in the proof of Theorem~\ref{t:LocExpSurf} and define the translated mapping $\Phi:(-c_1,c_1)^p\to\Sigma$
as follows
\beq\label{d:PhiPsi}
\Phi(y)=p\Psi(y).
\eeq
We are going to use the local expansion \eqref{eq:estim} in order to compute the Federer's density,
that is defined as follows
\beq\label{d:NFedDens}
\theta^\NN(\mu_\Sigma,p)=\inf_{r>0}\sup_{\substack{z\in \B(p,\tilde r) \\ 0<\tilde r<r}}\frac{\mu_\Sigma(\B(z,\tilde r))}{\tilde r^\NN}.
\eeq
Taking $r>0$ sufficiently small and $z\in B(p,\tilde r)$, in view of \eqref{e:intmeaslocalchart}, we have
\begin{equation}
\frac{\mu_\Sigma(\B(z,\tilde r))}{\tilde r^\NN}=\tilde r^{-\NN}
\int_{\Phi^{-1}(\B(z,\tilde r))}\|\pi_{\Phi(y),\NN}\lls \der_{y_1}\Phi(y)\wedge\cdots\wedge\der_{y_\n}\Phi(y)\rrs\|_g\,dy .
\end{equation}
Taking into account the relations
\[
\NN=\sum_{i=1}^\n b_i=\sum_{j=1}^\iota  j\,\alpha_j,
\]
and the ``induced dilations'' $\sigma_r$ introduced in \eqref{d:dilintrinsic}, the change of 
variable $y=\sigma_r(t)$ implies that
\beq\label{e:rescale_mu_tilde}
\frac{\mu_\Sigma(\B(z,\tilde r))}{\tilde r^\NN}=
\int_{\sigma_{1/\tilde r}(\Phi^{-1}(\B(z,\tilde r)))}
\|\pi_{\Phi(y),\NN}\lls \der_{y_1}\Phi(\sigma_{\tilde r}y)\wedge\cdots\wedge\der_{y_\n}\Phi(\sigma_{\tilde r}y)\rrs\|_g \,dy \,.
\eeq
Our first claim is the uniform boundedness of the following rescaled sets
\[
\sigma_{1/\tilde r}\big(\Phi^{-1}(\B(z,\tilde r))\big)=
\sigma_{1/\tilde r}\big(\Psi^{-1}(\B(p^{-1}z,\tilde r))\big)
\]
as $\tilde r<r$ and $d(p,z)\le\tilde r$ with $\tilde r$ sufficiently small. There holds
\begin{equation}\label{e:rescaledSet}
\sigma_{1/\tilde r}\big(\Phi^{-1}(\B(z,\tilde r))\big)=\left\{y\in\R^\n:
\delta_{1/\tilde r}(z^{-1}p)\delta_{1/\tilde r}(\Psi(\sigma_{\tilde r}y))\in\mathbb B(0,1)\right\}.
\end{equation}
We first observe that
\[
\zeta(\tau)=
\bigg(\sgn(\tau_1)\sqrt[b_1]{b_1|\tau_1|},\ldots,\sgn(\tau_p)\sqrt[b_p]{b_p|\tau_p|}\bigg)
\]
is the inverse of $\eta$, hence in view of \eqref{d:Gammaeta} and \eqref{d:defPsi} we have 
\beq\label{e:psirtilde}
\psi(\sigma_{\tilde r}y)=\Gamma(\zeta(\sigma_{\tilde r}y))=\Gamma(\tilde r\,\zeta(y))\,.
\eeq
In view of \eqref{e:bsmj-1}, we can write \eqref{eq:estim} as follows
\beq\label{e:estimVar}
\Gamma_s(t)=\left\{\begin{array}{ll} 
\eta_{s-\m_{d_s-1}+\mu_{d_s-1}}(t)  & \mbox{if $s\in I$} \\
o(|t|^{d_s}) & \mbox{if $s\notin I$}
\end{array}\right.,
\eeq
therefore whenever $s\in I$ we get
\beq\label{e:Gammasrtilde}
\begin{split}
\Gamma_s(\zeta(\sigma_{\tilde r}y))&=(\eta\circ \zeta)_{s-\m_{d_s-1}+\mu_{d_s-1}}(\sigma_{\tilde r}y) \\
&=(\sigma_{\tilde r}y)_{s-\m_{d_s-1}+\mu_{d_s-1}}\\
&=(\tilde r)^{b_{s-\m_{d_s-1}+\mu_{d_s-1}}}y_{s-\m_{d_s-1}+\mu_{d_s-1}}\\
&=(\tilde r)^{d_s}y_{s-\m_{d_s-1}+\mu_{d_s-1}}.
\end{split}
\eeq
As a result, taking into account that $d\pa{\delta_{1/\tilde r}(z^{-1}p),0}\le1$, an element $y\in\R^\n$ of \eqref{e:rescaledSet} satisfies the condition
\[
\begin{split}
&y_1e_1+\cdots +y_{\alpha_1}e_{\alpha_1}+
\frac{\Gamma_{\alpha_1+1}(r\zeta(y))}{\tilde r}e_{\alpha_1+1}+\cdots
+\frac{\Gamma_{\m_1}(\tilde r\zeta(y))}{\tilde r} e_{\m_1}\\
&+y_{\alpha_1+1}e_{\m_1+1}+\cdots +y_{\mu_2}e_{\m_1+\alpha_2}+
\frac{\Gamma_{\m_1+\alpha_2+1}(\tilde r\zeta(y))}{(\tilde r)^2}e_{\m_1+\alpha_2+1}
+\cdots+\frac{\Gamma_{\m_2}(\tilde r\zeta(y))}{(\tilde r)^2} e_{\m_2}\\
&\;\,\vdots\qquad\qquad\quad \vdots\qquad\qquad\quad \vdots\qquad\qquad\quad \vdots\qquad\qquad\quad \vdots\qquad\qquad\quad \vdots\qquad\qquad\quad \vdots \\
&+y_{\mu_{\iota-1}+1}e_{\m_{\iota-1}+1}+\cdots +y_\n e_{\m_{\iota-1}+\alpha_\iota}+\frac{\Gamma_{\m_{\iota-1}
+\alpha_{\iota-1}+1}(\tilde r\zeta(y))}{(\tilde r)^\iota}
e_{\m_{\iota-1}+\alpha_\iota+1}+\cdots\\
&\cdots +\frac{\Gamma_{\m_\iota}(\tilde r\zeta(y))}{(\tilde r)^\iota} e_{\m_\iota} \in\B(0,2).
\end{split}
\]
Since $\B(0,2)$ is also bounded with respect to the fixed Euclidean norm on $\G$ and 
$(e_1,\ldots,e_\q)$ is an orthonormal basis the previous expression implies 
the existence of a bounded set $V\subset A_p\Sigma$ such that 
\beq\label{e:unifBoundedness}
\sigma_{1/\tilde r}\big(\Phi^{-1}(\B(z,\tilde r))\big)\subset V
\eeq
for $r>0$ sufficiently small, $0<\tilde r<r$ and $d(z,p)\le \tilde r$.
We notice that the previous sums can be also written as follows
\[
\sum_{l=1}^\n y_l e_{\m_{b_l-1}+l-\mu_{b_l-1}}+\sum_{l\notin I} \frac{\Gamma_l(\tilde r\zeta(y))}{(\tilde r)^{d_l}} e_l\in\B(0,2).
\]
The uniform boundedness \eqref{e:unifBoundedness} joined with \eqref{e:rescale_mu_tilde} implies that 
$\theta^\NN(\tilde\mu\res\Sigma,p)<+\infty$, hence there exist a sequence $\{r_k\}\subset(0,+\infty)$ converging to zero and a sequence of elements $z_k\in\B(p,r_k)$ such that
\[
\theta^\NN(\mu_\Sigma,p)=\lim_{k\to\infty} \int_{\sigma_{1/r_k}(\Phi^{-1}(\B(z_k,r_k)))}
\|\pi_{\Phi(y),\NN}\lls \der_{y_1}\Phi(\sigma_{r_k}y)\wedge\cdots\wedge\der_{y_\n}\Phi(\sigma_{r_k}y)\rrs\|_g \,dy.
\]
Possibly extracting a subsequence, there exists $u_0\in\B(0,1)$ such that
\begin{equation}\label{e:deltak}
\delta_{1/r_k}(z_k^{-1}p)\to u_0^{-1}\in\B(0,1).
\end{equation}
We define the following subsets of the algebraic tangent space
\[
F_k=\sigma_{1/r_k}\big(\Phi^{-1}(\B(z_k,r_k))\big)\qandq F(u_0)=\B(u_0,1)\cap A_p\Sigma.
\]
Our second claim is the validity of the following limit
\begin{equation}\label{e:F_k}
 \lim_{k\to\infty}{\bf 1}_{F_k}(w)=0
\end{equation}
for each $w\in A_p\Sigma\sm F(u_0)$.
Arguing by contradiction, if there exists a sequence of positive integers $j_k$ such that 
\[
{\bf 1}_{F_{j_k}}(w)=1
\] 
for every $k\in\N$, then \eqref{e:rescaledSet} gives
\beq
\sum_{l=1}^\n w_l e_{\m_{b_l-1}+l-\mu_{b_l-1}}+\sum_{l\notin I} \frac{\Gamma_l(r_{j_k}\zeta(w))}{(r_{j_k})^{d_l}} e_l\in\delta_{1/r_{j_k}}(p^{-1}z_{j_k})\B(0,1),
\eeq
since the previous element precisely coincides with
$\delta_{1/r_{j_k}}(\Psi(\sigma_{r_{j_k}}w))$.
The estimate \eqref{e:estimVar} joined with the limit \eqref{e:deltak}, as $k\to\infty$ give
\[
\sum_{l=1}^\n w_l e_{\m_{b_l-1}+l-\mu_{b_l-1}}\in \B(u_0,1)\cap A_p\Sigma=F(u_0),
 \]
 that is a contradiction.
We now define 
\beq
I_k=\int_{F_k}\|\pi_{\Phi(y),\NN}\lls \der_{y_1}\Phi(\sigma_{r_k}y)\wedge\cdots\wedge\der_{y_\n}\Phi(\sigma_{r_k}y)\rrs\|_g\,dy,
\eeq
along with 
\beq\label{d:Jik}
\begin{split}
J_{1,k}&=\int_{F_k\cap F(u_0)}\|\pi_{\Phi(y),\NN}\lls \der_{y_1}\Phi(\sigma_{r_k}y)\wedge\cdots\wedge\der_{y_\n}\Phi(\sigma_{r_k}y)\rrs\|_g\,dy, \\
J_{2,k}&=\int_{F_k\sm F(u_0)}\|\pi_{\Phi(y),\NN}\lls \der_{y_1}\Phi(\sigma_{r_k}y)\wedge\cdots\wedge\der_{y_\n}\Phi(\sigma_{r_k}y)\rrs\|_g\,dy,
\end{split}
\eeq
so that $I_k=J_{1,k}+J_{2,k}$ for each $k\ge0$. Taking the limit of the following inequality
\begin{equation}\label{eq:I_kle}
J_{1,k}\le \int_{F(u_0)}\|\pi_{\Phi(y),\NN}\lls \der_{y_1}\Phi(\sigma_{r_k}y)\wedge\cdots\wedge\der_{y_\n}\Phi(\sigma_{r_k}y)\rrs\|_g\,dy,
\end{equation}
we obtain 
\beq\label{e:ineqJ1k}
\limsup_{k\to\infty} J_{1,k}\le {\mathcal H}^\n_{|\cdot|}(F(u_0))\,
\|\pi_{\Phi(y),\NN}\lls \der_{y_1}\Phi(0)\wedge\cdots\wedge\der_{y_\n}\Phi(0)\rrs\|_g.
\eeq
Joining \eqref{d:Jik} with \eqref{e:unifBoundedness}, we also get
\beq
J_{2,k}\le\int_{V\sm F(u_0)}{\bf 1}_{F_k}(y)\|\pi_{\Phi(y),\NN}\lls \der_{y_1}\Phi(\sigma_{r_k}y)\wedge\cdots\wedge\der_{y_\n}\Phi(\sigma_{r_k}y)\rrs\|_g\,dy.
\eeq
The boundedness of $V$ and \eqref{e:F_k} joined with the classical Lebesgue's convergence theorem imply that
\beq\label{e:limJ2k}
\lim_{k\to\infty}J_{2,k}=0.
\eeq
In view of \eqref{e:ineqJ1k} and \eqref{e:limJ2k}, we have proved that
\beq
\theta^\NN(\mu_\Sigma,p)\le 
{\mathcal H}^\n_{|\cdot|}\lls\B(u_0,1)\cap A_p\Sigma\rrs\,\|\pi_{\Phi(y),\NN}\lls \der_{y_1}\Phi(0)\wedge\cdots\wedge\der_{y_\n}\Phi(0)\rrs\|_g,
\eeq
where $u_0\in\B(0,1)$, therefore the definition of spherical factor yields
\beq\label{e:ineqthetaN}
\theta^\NN(\mu_\Sigma,p)\le \beta_d(A_p\Sigma)\,\|\pi_{\Phi(y),\NN}\lls \der_{y_1}\Phi(0)\wedge\cdots\wedge\der_{y_\n}\Phi(0)\rrs\|_g.
\eeq 
Our third claim is the validity of the equality in \eqref{e:ineqthetaN}.
Let $v_0\in\B(0,1)$ be such that 
\beq\label{e:betav_0}
\beta_d(A_p\Sigma)=\cH^\n_{|\cdot|}\lls\B(v_0,1)\cap A_p\Sigma\rrs,
\eeq
define $v_{\tilde r}=p\delta_{\tilde r}v_0\in\B(p,\tilde r)$ for $\tilde r>0$ and fix $\lambda>1$. We observe that
\[
\sup_{0<\tilde r<r}\frac{\mu_\Sigma\lls\B(v_{\tilde r},\lambda \tilde r)\rrs}
{(\lambda \tilde r)^\NN}
\le\sup_{\substack{u\in\B(p,r')\\ 0<r'<\lambda r}}\frac{\mu_\Sigma\lls\B(u,r')\rrs}{({r'})^\NN}
\]
for each $r>0$ sufficiently small. From the definition of spherical Federer density \eqref{d:NFedDens},
it follows that
\beq\label{e:ineqmutilde}
\limsup_{\tilde r\to0^+}\frac{\mu_\Sigma\lls\B(v_{\tilde r},\lambda \tilde r)\rrs}{(\lambda \tilde r)^\NN}\le\theta^\NN(\mu_\Sigma,p)\,.
\eeq
We wish to write a formula for $\mu_\Sigma\lls\B(v_{\tilde r},\lambda \tilde r)\rrs$, therefore we consider \eqref{e:rescale_mu_tilde} and apply \eqref{e:rescaledSet}, replacing $\tilde r$ with $\lambda \tilde r$ and $z$ with $v_{\tilde r}$.
It follows that the set
\begin{equation}
E_{\tilde r}=\delta_{1/(\lambda \tilde r)}\lls\Phi^{-1}(\B(v_{\tilde r},\lambda \tilde r))\rrs
=\Big\{y\in\R^\n: \delta_{1/\tilde r}(\Psi(\sigma_{\lambda\tilde r}y))\in\mathbb B(v_0,\lambda)\Big\}
\end{equation}
gives the equality
\beq
\frac{\mu_\Sigma\lls\B(v_{\tilde r},\lambda \tilde r)\rrs}{(\lambda \tilde r)^\NN}
=\int_{E_{\tilde r}}\|\pi_{\Phi(y),\NN}\lls \der_{y_1}\Phi(\sigma_{\lambda \tilde r}y)\wedge\cdots\wedge\der_{y_\n}\Phi(\sigma_{\lambda\tilde r}y)\rrs\|_g\,dy \,.
\eeq
Setting $\tilde E_{\tilde r}=\sigma_\lambda(E_{\tilde r})$ and performing the change of variables $y=\sigma_{1/\lambda}\tilde y$, we get
\beq\label{e:mutildeintegral}
\frac{\mu_\Sigma\lls\B(v_{\tilde r},\lambda \tilde r)\rrs}{(\lambda \tilde r)^\NN}
=\frac{1}{\lambda^\NN}\int_{\tilde E_{\tilde r}}
\|\pi_{\Phi(y),\NN}\lls \der_{y_1}\Phi(\sigma_{\tilde r}\tilde y)\wedge\cdots\wedge\der_{y_\n}\Phi(\sigma_{\tilde r}\tilde y)\rrs\|_g\,d\tilde y \,,
\eeq
where we have defined 
\[
\tilde E_{\tilde r}=
\Big\{y\in\R^\n: \delta_{1/\tilde r}(\Psi(\sigma_{\tilde r}y))\in\mathbb B(v_0,\lambda)\Big\}.
\]
Now, we fix $1<\tilde \lambda<\lambda$, the subset
\beq\label{e:Hrtilde}
H_{\tilde r}=
\Big\{y\in\R^\n: \delta_{1/\tilde r}(\Psi(\sigma_{\tilde r}y))\in B(v_0,\tilde\lambda)\Big\}
\eeq
and observe that \eqref{e:psirtilde}, \eqref{e:estimVar} and \eqref{e:Gammasrtilde}, in view of $\Psi(y)=\sum_{j=1}^\q\psi_j(y) e_j$, show that
\beq
\delta_{1/\tilde r}(\Psi(\sigma_{\tilde r}y))=\sum_{l=1}^\n y_l e_{\m_{b_l-1}+l-\mu_{b_l-1}}+\sum_{l\notin I} \frac{\Gamma_l(\tilde r\zeta(y))}{(\tilde r)^{d_l}} e_l
\eeq
for each $y\in\R^\n$ converges to 
\beq
\sum_{l=1}^\n y_l e_{\m_{b_l-1}+l-\mu_{b_l-1}}\in A_p\Sigma \quad \text{as} \quad \tilde r\to 0^+.
\eeq
As a result, for any $y\in B(v_0,\tilde \lambda)\cap A_p\Sigma$ there holds
\[
\lim_{\tilde r\to0^+} {\bf 1}_{H_{\tilde r}\cap B(v_0,\tilde\lambda)}(y)=1.
\]
Thus, taking into account that \eqref{e:ineqmutilde}, \eqref{e:mutildeintegral}, \eqref{e:Hrtilde}, the following limit superior
\[
\limsup_{\tilde r\to0^+}
\frac{1}{\lambda^\NN}\int_{B(v_0,\tilde \lambda)\cap A_p\Sigma}
{\bf 1}_{H_{\tilde r}\cap B(v_0,\tilde\lambda)}(\tilde y)
\|\pi_{\Phi(y),\NN}\lls \der_{y_1}\Phi(\sigma_{\tilde r}\tilde y)\wedge\cdots\wedge\der_{y_\n}\Phi(\sigma_{\tilde r}\tilde y)\rrs\|_g\,d\tilde y \
\]
is not greater than $\theta^\NN(\mu_\Sigma,p)$.
Then Lebesgue's convergence theorem gives
\[
\frac{1}{\lambda^\NN}\cH_{|\cdot|}^\n(B(v_0,\tilde\lambda)\cap A_p\Sigma)\,
\|\pi_{\Phi(y),\NN}\lls \der_{y_1}\Phi(0)\wedge\cdots\wedge\der_{y_\n}\Phi(0)\rrs\|_g\le 
\theta^\NN(\mu_\Sigma,p).
\]
Letting first $\tilde \lambda\to1^+$ and then $\lambda\to1^+$, due to 
\eqref{e:betav_0}, we get
\beq
\beta_d(A_p\Sigma)\,
\|\pi_{\Phi(y),\NN}\lls \der_{y_1}\Phi(0)\wedge\cdots\wedge\der_{y_\n}\Phi(0)\rrs\|_g\le
\theta^\NN(\mu_\Sigma,p).
\eeq
Joining this inequality with \eqref{e:ineqthetaN} we get
a formula for the Federer density
\beq
\theta^\NN(\mu_\Sigma,p)=\beta_d(A_p\Sigma)\,
\|\pi_{\Phi(y),\NN}\lls \der_{y_1}\Phi(0)\wedge\cdots\wedge\der_{y_\n}\Phi(0)\rrs\|_g.
\eeq
Finally, by \eqref{e:matriceC} and \eqref{d:PhiPsi} we observe that
\[
\begin{split}
\pi_{\Phi(y),\NN}\lls \der_{y_1}\Phi(0)\wedge\cdots\wedge\der_{y_\n}\Phi(0)\rrs=&X_1\wedge\cdots\wedge X_{\alpha_1}\wedge X_{\m_1+1}\wedge\cdots \wedge X_{\m_1+\alpha_2}\wedge\cdots \\
&\cdots\wedge X_{\m_{\iota-1}+1}\wedge\cdots\wedge X_{\m_{\iota-1}+\alpha_\iota},
\end{split}
\]
which has unit norm with respect to $\|\cdot\|_g$. This completes our proof.
\end{proof}

\section{Sections of convex balls and spherical factor}\label{sect:sectConvex} 

In this section, we are concerned with finding formulas that make 
the spherical factor easier to compute. 
We consider homogeneous distances whose metric unit ball is convex.
\bt\label{t:ConvSect_n-dim}
Let $H$ be a $\q$-dimensional Hilbert space with $\q\ge2$ and let 
$C$ be a compact and convex set, whose interior is nonempty and it contains the origin.
Let $S$ denote an $\n$-dimensional subspace of $H$ and consider $V=S^\bot$ its orthogonal subspace. Then the subset $D=\{v\in V: C\cap (v+S)\neq\emptyset\}$
is convex and $\psi:D\to[0,+\infty)$, defined as follows
\[
\psi(v)=\big[\cH_{|\cdot|}^\n\lls C\cap (v+S)\rrs\big]^{1/\n}
\]
is concave on $D$.
\et
\begin{proof}
It is easy to observe that $D$ is convex and with nonempty interior in $V$.
Let us consider $v,w\in D$ and $\theta\in(0,1)$. Convexity of $C$ gives
\[
\theta \ls (v+S)\cap C\rs+(1-\theta)\ls (w+S)\cap C \rs\subset 
\ls [\theta v+(1-\theta) w]+S\rs\cap C\,,
\]
that it can be rewritten as follows
\beq\label{e:inclCgammau}
\ls (\theta v+S)\cap \theta C\rs+\ls \ls(1-\theta) w+S\rs \rs\cap (1-\theta)C \rs\subset \ls [\theta v+(1-\theta)w]+S\rs\cap C\,.
\eeq
Our point is to use the the classical Brunn-Minkowski inequality, in dimension $\n$. We observe 
the following equality of sets
\[
 (\gamma u+S)\cap \gamma C=\gamma u+\ls S\cap (\gamma C- \gamma u)\rs
\]
for every $\gamma\in\R$ and $u\in H$, hence we define 
\beq 
C(\gamma,u)=S\cap \ls \gamma  C-\gamma u\rs.
\eeq
Thus, the inclusion \eqref{e:inclCgammau} can be written as follows
\[
\big[\theta v+C(\theta,v)\big]+\big[(1-\theta)w+C\lls(1-\theta),w\rrs\big]\subset \lls [\theta v+(1-\theta)w]+S\rrs\cap C\,.
\]
It follows that
\begin{equation}\label{e:convexityIncl}
\theta v+(1-\theta)w+C(\theta,v)+C\ls(1-\theta),w\rs\subset 
\lls [\theta v+(1-\theta)w]+S\rrs\cap C\,.
\end{equation}
The Brunn-Minkowski inequality in $S$ gives
\beq\label{e:BrunnMinkowski}
\big[ \cH_{|\cdot|}^\n\lls C(\theta,v)+C\ls(1-\theta),w)\rrs\big]^{1/\n}\geq
\big[ \cH_{|\cdot|}^\n\lls C(\theta,v)\rrs\big]^{1/\n}
+\big[ \cH_{|\cdot|}^\n\lls C\lls(1-\theta),w\rrs\rrs\big]^{1/\n}.
\eeq
Taking into account that
\[
 C(\theta,u)=\theta\, C(1,u) \quad \mbox{and}\quad   C\lls(1-\theta),u\rrs=(1-\theta)\, C(1,u),
\]
the inequality \eqref{e:BrunnMinkowski} joined with the inclusion \eqref{e:convexityIncl}, we obtain
\[
 \psi\ls \theta v+(1-\theta)w\rs\ge  \theta\, \big[\cH_{|\cdot|}^\n\ls C(1,v)\rs\big]^{1/\n}+(1-\theta)\,\big[ \cH_{|\cdot|}^\n\ls C(1,w)\rs\big]^{1/\n}\,.
\]
Finally, we observe that
\[
\psi(u)=\big[\cH_{|\cdot|}^\n\ls C(1,u)\rs\big]^{1/\n}
\]
for each $u\in D$, hence completing the proof.
\end{proof}
The next lemma provides a factorization of a homogeneous group $\G$
with respect to a vertical subgroup $N$, although a {\em complementary subgroup} $V$ such that $V\oplus N=\G$ may not exist. We will consider $V$ as simple homogeneous subspace of $\G$.
\bl\label{l:hom_splitting}
If $\G$ is a homogeneous group, $V\subset\G$ is a homogeneous subspace
and $N\subset\G$ is a vertical subgroup such that $V\oplus N=\G$, then the mapping
\[
V\times N\to \G,\quad (v,h)\to vh
\]
is an analytic diffeomorphism. Furthermore, its inverse mapping 
$T:\G\to V\times N$ is defined by 
\[
T(x)=\pa{P_V(x),\Pi_N(x)}
\]
where $P_V:\G\to V$ is the linear projection onto $V$ with
respect to the direct sum $\G=V\oplus N$ and
$\Pi_N(x)=P_V(x)^{-1}x$.
\el
\begin{proof}
The mapping $F(v,h)=vh$ has the property that
\[
\der_v F(0,0)|_{V}=\Id_V\qandq \der_h F(0,0)|_{N}=\Id_N
\]
so the assumption that $V\oplus N=\G$ implies that $dF(0,0)$
is invertible and $F$ is locally invertible around the origin.
The homogeneity of $F$, i.e.
\[
F(\delta_rv, \delta_r h)=\delta_r F(h,v)
\]
shows that $F$ is surjective. We now consider 
\[
vh=wk,
\]
where $v,w\in V$ and $h,k\in N$. 
By the BCH formula \eqref{eq:BCH}, we have
\beq\label{eq:vwN}
v+h+\sum_{j=2}^\iota c_j(v,h)=w+k+\sum_{j=2}^\iota c_j(w,k)
\eeq
where the fact that $N$ is vertical gives
\[
h+\sum_{j=2}^\iota c_j(v,h)\in N \qandq  k+\sum_{j=2}^\iota c_j(w,k)\in N.
\]
Applying $P_V$ to the equality \eqref{eq:vwN}, it follows that 
\[
v=P_V(vh)=P_V(wk)=w.
\]
We have shown that any element $x\in\G$ can be uniquely written as the product
\[
P_V(x)\Pi_N(x),
\]
therefore concluding the proof.
\end{proof}
\bl
If $\G$ is a homogeneous group, $V\subset\G$ is a homogeneous subspace
and $N\subset\G$ is a vertical subgroup such that $V\oplus N=\G$, then 
for every $v\in V$ we have
\beq\label{eq:v+N=vN}
v+N=vN.
\eeq
\el
\begin{proof}
From \eqref{eq:BCH} and the fact that $N$ is an ideal with respect
to the Lie algebra structure of $\G$, there holds
\[
vn=v+n+\sum_{j=2}^\iota c_j(v,n)\in h+N
\]
being $c_j(v,n)\in N$ for every $j=2,\ldots,\iota$. This shows that
$vN\subset v+N$. Conversely, considering $v+n\in\G$ and applying
Lemma~\ref{l:hom_splitting} we get
\[
v+n=P_V(v+n)\Pi_N(v+n)=v\Pi_N(v+n)=v\tilde n
\]
where $\tilde n=\Pi_N(v+n)\in N$.  We have then established the opposite
inclusion.
\end{proof}
An important feature of vertical subgroups is the following
left invariance property of the Euclidean Hausdorff measure.
\bl\label{l:vertical_meas_translation}
If $\G$ is a homogeneous group and $N\subset\G$ is an $\n$-dimensional vertical subgroup, then for every $p\in\G$ and every measurable set $A\subset N$, we have
\[
\cH_{|\cdot|}^\n(A)=\cH^\n_{|\cdot|}(l_p(A)),
\]
where $l_p:\G\to\G$ denotes the left translation by $p$.
\el
\begin{proof}
We consider a graded basis $(e_1,\ldots,e_\q)$, such that
\[
V=\spn\set{e_1,\ldots,e_{\q-\n}} \qandq N=\spn\set{e_{\q-\n+1},\ldots,e_\q}
\]
and the associated graded coordinates in $\G$, setting
\[
p=\sum_{j=1}^\q x_j e_j\qandq n=\sum_{j=1}^\n \zeta_j e_{\q-\n+j}\in N.
\]
We express the left translation explicitly
\beq\label{eq:translation_left_n}
l_pn=\sum_{j=1}^{\q-\n} x_j e_j+\sum_{j=\q-\n+1}^\q(x_j+\zeta_{j-\q+\n}) e_j+\sum_{j=\m+1}^\q c_j(\bar x^{d_j-1},\bar \zeta^{d_j-1}) e_j.
\eeq
According to the Baker-Campbell-Hausdorff formula, the functions $c_j$
are polynomials that only depends on variables of degree less than $d_j$.
We have defined
\beq\label{eq:barxn}
\bar x^{d_j-1}=\sum_{d_i\le d_j-1} x_i \tilde b_i\qandq 
\bar \zeta^{d_j-1}=\sum_{\substack{d_{\q-\n+i}\le d_j-1\\ \q-\n+1\le i\le\q}} \zeta_i \tilde b_{\q-\n+i},
\eeq
where $(\tilde b_1,\ldots,\tilde b_\q)$ is the canonical basis of $\R^\q$.
Let us remark that 
$\bar \zeta^{d_j-1}=0$ for $d_j\le d_{\q-\n+1}$.
We decompose $p$ into the sum
\[
p=v+w, \quad \text{where} \quad v=\sum_{j=1}^{\q-\n} x_j e_j\qandq
w=\sum_{j=\q-\n+1}^\q x_j e_j.
\]
With this notation formula \eqref{eq:v+N=vN} gives $pN=v+N$.
Since $v\notin N$, this decomposition improves \eqref{eq:translation_left_n}, giving
\[
l_pn=\sum_{j=1}^{\q-\n} x_j e_j+\sum_{j=\q-\n+1}^\q(x_j+\zeta_j) e_j
+\sum_{j=k_\n}^\q c_j(\bar x^{d_j-1},\bar \zeta^{d_j-1}) e_j,
\]
with $k_\n=\max\set{\q-\n+1,\m+1}$.
We now consider the projection
\[
T:v+N\to\R^\n,\quad 
v+\sum_{j=\q-\n+1}^\q \zeta_j e_j \lra \sum_{j=\q-\n+1}^\q \zeta_j b_{j-\q+\n},
\]
where $(b_1,\ldots,b_\n)$ is the canonical basis of $\R^\n$ 
and $J:\R^\n\to N$, defined as 
\[
J(\zeta_1,\ldots,\zeta_\n)= \sum_{j=1}^\n\zeta_je_{\q-\n+j}.
\]
The composition $F:\R^\n\to\R^\n$ defined 
as $F=T\circ l_p\circ J$ can be written as follows
\[
F(\zeta)=\sum_{j=1}^\n(x_{\q-\n+j}+\zeta_j)\, b_j+
\sum_{j=k_\n-\q+\n}^\n c_{\q-\n+j}(\bar x^{d_{\q-\n+j}-1},\bar \zeta^{d_{\q-\n+j}-1}) b_j.
\]
As a consequence of the previous formulae, for every $j,l=1,\ldots,\n$ 
we get
\[
\dpar{F_j}{\zeta_l}=\delta^j_l+\dpar{c_{\q-\n+j}(\bar x^{\q-\n+j},\cdot)}{\zeta_l}.
\]
In the special case $l\ge j$, due to \eqref{eq:barxn} the function $\zeta\to c_{\q-\n+j}(\bar x^{\q-\n+j},\zeta)$ only depends
on $\zeta_i$ with $d_{\q-\n+i}\le d_{\q-\n+j}-1$, therefore
$i<j\le l$. We have proved that
\[
\dpar{F_j}{\zeta_l}=\delta^j_i \quad \text{whenever $l\ge j$},
\]
hence the Jacobian of $F$ is one. Since both $T$ and 
$J$ are isometries one easily observes that image measures satisfy
\[
T_\sharp\cH^\n_{|\cdot|}=\cL^\n\qandq J_\sharp\cL^\n=\cH^\n_{|\cdot|}.
\]
Since $F$ preserves the Lebesgue measure $\cL^\n$, the following equalities conclude the proof, that is
\[
\begin{split}
\cH^\n_{|\cdot|}\pa{l_p(A)}&=T_\sharp\cH^\n_{|\cdot|}\pa{F\circ J^{-1}(A)}=\cL^n\pa{F(J^{-1}(A))}\\
&=\cL^n\pa{J^{-1}(A)}=J_\sharp\cL^\n(A)=\cH^\n_{|\cdot|}(A). \qedhere
\end{split}
\]
\end{proof}
\begin{proof}[Proof of Theorem~\ref{t:homConvBall}]
According to Definition~\ref{d:vertical_subgroup}, we set
\[
N=N_\ell\oplus H^{\ell+1}\oplus\cdots\oplus H^\iota,
\]
for some $\ell\in\set{1,2,\ldots,\iota}$. Applying Lemma~\ref{l:vertical_meas_translation}, we get
\beq\label{eq:l_pMeasPreserving}
\cH_{|\cdot|}^\n\ls N\cap \B(z,1)\rs =\cH_{|\cdot|}^\n\ls \B(0,1)\cap z^{-1}N\rs\,.
\eeq
To study the previous function with respect to $z$, we define
\[
 a(z)=\cH_{|\cdot|}^\n\ls \B(0,1)\cap zN\rs\,.
\]
We set $V=N^\bot$, that can be written as follows
\[
V=H^1\oplus\cdots\oplus H^{\ell-1}+S_\ell.
\]
This is a homogeneous subspace of $\G$ need not be a subgroup,
however Lemma~\ref{l:hom_splitting} gives the mappings
\[
P_V:\G\to V\qandq \Pi_N:\G\to N,
\]
such that for $y\in\G$ we have 
\[
y=P_V(y)\Pi_N(y),
\]
where $P_V$ is the projection onto $V$ with respect to the 
direct sum $\G=V\oplus N$. As a consequence, we can write
\begin{equation}\label{eq:az}
 a(z)=\cH_{|\cdot|}^{\n-1}\ls \B(0,1)\cap zN\rs=\cH_{|\cdot|}^{\n-1}\ls \B(0,1)\cap P_V(z)N\rs\,.
\end{equation}
We are in the assumptions to apply \eqref{eq:v+N=vN}, hence $P_V(z)N=P_V(z)+N$.
Our special representation of $\G$ also allows us to have  
\beq\label{eq:B01-1}
\B(0,1)^{-1}=-\B(0,1).
\eeq
It follows that 
\[
  a(z)=\cH_{|\cdot|}^{\n-1}\ls \B(0,1)\cap \ls P_V(z)+N\rs\rs
  =\cH_{|\cdot|}^{\n-1}\ls \B(0,1)\cap\ls-P_V(z)+N\rs\rs
\]
and the property $-P_V(z)=P_V(z^{-1})$ yields
\[
  a(z)=\cH_{|\cdot|}^{\n-1}\ls \B(0,1)\cap \ls P_V(z^{-1})+N\rs\rs
  =\cH_{|\cdot|}^{\n-1}\ls \B(0,1)\cap\ls P_V(z^{-1})N\rs\rs.
\]
Thus, by \eqref{eq:az} we get 
$a(z)=\cH_{|\cdot|}^{\n-1}\ls \B(0,1)\cap\ls z^{-1}N\rs\rs=a(z^{-1})=a(-z)$, hence $a$ is even and for every $t\in\R$ we may define the even function
\[
b(t)=\Big[\cH_{|\cdot|}^{\n-1}\Big( \B(0,1)\cap\ls tv+N\rs\Big)\Big]^{1/(\n-1)}.
\]
Due to Theorem~\ref{t:ConvSect_n-dim} the function $t\to b(t)=\sqrt[n-1]{a(tv)}$ is also concave on the compact interval 
\[
 I=\{t\in\R: \B(0,1)\cap (tv+N)\neq\emptyset\}.
\]
Using again \eqref{eq:B01-1} one easily observes that $I$ is an even interval, therefore
\[
 \beta_d(N)=\max_{z\in\B(0,1)}\cH_{|\cdot|}^{\n-1}\ls(\B(z,1)\cap N(v)\rs=\cH_{|\cdot|}^{\n-1}\ls N\cap\B(0,1)\rs,
\]
concluding the proof.
\begin{comment}
Being $N(v)\cap\der \B(0,1)$ locally parametrized by Lipschitz mappings on an $(n-2)$ dimensional open set,
then we obviously have $\cH_{|\cdot|}^{n-1}\ls N(v)\cap\der \B(0,1)\rs=0$, concluding the proof.
\end{comment}
%
%
%
\begin{comment}
Notice that we have used the fact that
\[
 N(v)\cap\der \B(0,1)=\der\ls N(v)\cap\B(0,1)\rs
\]
and this is true, since the interior of $N(v)\cap\B(0,1)$ is nonempty.
The subset $\der\ls N(v)\cap\B(0,1)\rs$ is the boundary of a convex set in $N$ having nonempty interior,
hence it is locally parametrized by a Lipschitz mapping.
\end{comment}
\end{proof}
The next examples show homogeneous distances with convex metric unit ball.

\bex\label{ex:homdistEuclBall}\rm
In any homogeneous group one can find a homogeneous distance 
with convex unit ball. Indeed \cite[Theorem~2]{HebSik90} shows that
this distance can be found in such a way its corresponding metric unit ball
is the Euclidean ball with suitably small radius
and that all layers $H^j$ in the decomposition of $\G$ are orthogonal.
\eex
\bex\label{ex:htype}\rm
In any H-type group with direct decomposition $\G=H^1\oplus H^2$, the well known Cygan-Kor\'anyi norm
\[  \|x\|=\sqrt[4]{|x_1|^4+16|x_2|^2}, \]
where $(x_1,x_2)\in H^1\times H^2$ and $x=x_1+x_2$,
clearly yields a convex metric unit ball, see \cite{Cygan81} for more information.
This ``homogeneous norm'' defines the associated homogeneous distance 
$d(x,y)=\|x^{-1}y\|$.
\eex
%
%
%
%
\bex\label{ex:dinfty} \rm
Setting $\ep_1=1$ and suitabliy small $\ep_i>0$, one can always construct a nonsmooth homogeneous distance defining
\[
\|x\|_\infty=\max\{\ep_i|x_i|^{1/i}: 1\le i\le\q\}
\]
and then $d(x,y)=\|x^{-1}y\|_\infty$ for $x,y\in\G$, see for instance \cite{SerraCassano2016}.
One can realize that
\[ 
\B=\{x\in\G: d(x,0)\le1\}=\{x\in\G:\,\ep_i^i|x_i|\le1\,\mbox{for all $i$} \}
\]
is a convex set.
\eex

Formula \eqref{eq:bdN} simplifies the computation of the spherical factor,
that however may depend on the subgroup $N$. The next section
deals with special symmetric distances where this dependence disappears.

%
%
%
%
%
%
\section{Special symmetries for homogeneous distances}\label{sect:Hsym}
%
%
%
%
%
%
%

This section studies classes of homogeneous distances for which the
spherical factor becomes a geometric constant. We  focus our attention on spherical factors for vertical subgroups of a fixed dimension, that appear as h-tangent spaces of some transversal submanifold.
\begin{Def}[$\n$-vertically symmetric distance]\label{d:hsym} \rm
Let $\G$ be a $\q$-dimensional stratified group let $\n\le\q$.
We consider $\ell_\n$ as defined in \eqref{d:ell} and denote $\ell=\ell_\n$.
Considering $\rr_\n$ as in \eqref{d:rn}, we introduce the integer
\[
\jmath=\left\{\begin{array}{ll}
0 & \text{if $\rr_\n=\dim H^\ell$} \\
\rr_\n & \text{if $\rr_\n<\dim H^\ell$}
\end{array}\right..
\]
If $\jmath=0$, then $d$ is automatically {\em $\n$-vertically symmetric}.
In the case $\jmath>0$, referring to the fixed graded scalar product, 
we assume that there exists a family $\cF_\ell\subset O(H^\ell)$ 
of isometries in $H^\ell$ such that
for any couple of $\jmath$-dimensional subspaces $S_1,S_2\subset H^\ell$ 
there exists $J\in\cF_\ell$ that satisfies the condition
\[
J(S_1)=S_2.
\]
Taking into account that $H^i$ and $H^j$ are orthogonal for $i\neq j$,
we introduce the class of isometries
\[
 \cO=\{T\in O(\G): T|_{H^j}=\Id_{H^j} \ \text{for all $j\neq\ell$ and}\  T|_{H^\ell}\in\cF_\ell\}. 
 \]
We denote by $P_V:\G\to V$ the orthogonal projection onto $V$.
We denote by 
\[
\B(0,1)=\{y\in\G: d(y,0)\le1\}
\]
the metric unit ball with respect to the homogeneous distance $d$ on $\G$.
Defining the subspaces
\[
U=H^1\oplus\cdots\oplus H^{\ell-1},\quad V=H^1\oplus\cdots\oplus H^\ell \qandq W=H^{\ell+1}\oplus\cdots\oplus H^\iota,
\]
setting $W=\set{0}$ in the case $\ell=\iota$,
we say that $d$ is {\em $\n$-vertically symmetric} if in addition
the following properties hold. If $P_V:\G\to V$ and $P_U:\G\to U$ 
denote the projections with respect to the direct decomposition 
\[
\G=V\oplus W\qandq  \G=U\oplus (H^\ell\oplus\cdots \oplus H^\iota),
\]
respectively, then we have that
\begin{enumerate}
\item 
$P_V(\B(0,1))=\B(0,1)\cap V=\{h\in V: \psi(P_U(v),|P_{H^\ell}(v)|)\le r_0\}$ for some $\psi:U\times [0,+\infty)\to[0,+\infty)$ that is monotone nondecreasing on the second variable and $r_0>0$,
\item
$T\ls \B(0,1)\rs=\B(0,1)$ for all $T\in\cO$.
\end{enumerate}
In the case $\ell=1$, we have both $U=\set{0}$ and 
$\psi:[0,+\infty)\to[0,+\infty)$ is any monotone nondecreasing function.
\end{Def}
\br
One may observe that any $(\q-1)$-vertically symmetric distance 
is {\em vertically symmetric} in the sense introduced in \cite{Mag31}.
\er
\br
Considering the homogeneous distance $d$ of Example~\ref{ex:homdistEuclBall},
it is not difficult to realize that $d$ is $\n$-vertically symmetric
for any $\n=1,\ldots,\q-1$.
As a consequence, the next theorem will show that its spherical factor for any
class of vertical subgroups of fixed dimension becomes a geometric constant.
\er
We have now all the tools to prove that the main symmetry result of this section.
\begin{The}\label{t:constBeta}
If $d$ is an $\n$-vertically symmetric distance, then the spherical factor 
\[
\beta_d(N)
\]
is constant for every $\n$-dimensional vertical subgroup $N\subset\G$.
\end{The}
\begin{proof}
We first consider the integers $\ell_\n$ and $\rr_\n$ defined
in \eqref{d:ell} and \eqref{d:rn}, respectively.
If $\rr_\n=\h_{\ell_\n}$, namely $\jmath=0$ in Definition~\ref{d:hsym}, then $d$ is $\n$-vertically symmetric by definition.
Indeed in this case the only $\n$-dimensional vertical subgroup is
\[
N_0=H^{\ell_\n}\oplus\cdots\oplus H^\iota,
\]
therefore the spherical factor
is obviously constant and equal to $\beta(d,N_0)$.
Let us consider the case $0<\jmath=\rr_\n<\h_{\ell_\n}$, where 
\[
\n=\jmath+\h_{\ell_\n+1}+\cdots+h_\iota.
\]
To simplify notation, in the sequel we will write $\ell$ in place of $\ell_\n$.
We fix $z\in\B(0,1)$ and consider two arbitrary $\n$-dimensional 
vertical subgroups
\[
N_1=S_1\oplus W \qandq N_2=S_2\oplus W,
\]
where $S_1$ and $S_2$ are $\jmath$-dimensional subspaces of $H^{\ell}$ and
\[
W=H^{\ell+1}\oplus\cdots\oplus H^\iota.
\]
By the $\n$-vertical symmetry of $d$, there exists an isometry $J:H^\ell\to H^\ell$ with 
\[
J(S_1)=S_2.
\]
This defines the isometry $T:\G\to\G$ such that 
\[
T|_{H^j}=\Id_{H^j} \qandq T|_{H^\ell}=J
\]
for each $j\neq\ell$. We now define the subspace
\[
V_1=H^1\oplus\cdots\oplus H^{\ell-1}\oplus Z_1
\]
such that $Z_1$ is orthogonal to $S_1$ and
\[
Z_1\oplus S_1=H^\ell.
\]
We consider the orthogonal projection $P_{V_1}:\G\to V_1$, that is also
the linear projection associated to the direct sum
\[
\G=V_1\oplus N_1.
\]
Defining the nonlinear projection
\[
\Pi_{N_1}(x)=P_{V_1}(x)^{-1}x,
\]
Lemma~\ref{l:hom_splitting} ensures the following unique product decomposition
\beq\label{eq:z^-1decomp}
z^{-1}=P_{V_1}(z^{-1})\Pi_{N_1}(z^{-1})\quad\text{with}\quad 
\Pi_{N_1}(z^{-1})\in N_1.
\eeq
Setting $P_{V_1}(z^{-1})=v_1$, $\Pi_{N_1}(z^{-1})=h_1$ and
taking into account \eqref{eq:l_pMeasPreserving}, \eqref{eq:z^-1decomp} and
\eqref{eq:v+N=vN},  we get
\begin{equation*}
\cH_{|\cdot|}^\n(\B(z,1)\cap N_1)=\cH^\n_{|\cdot|}(\B(0,1)\cap (v_1+N_1)).
\end{equation*}
Since the previously defined mapping $T$ belongs to $\cO$, in view of the property (2) of Definition~\ref{d:hsym}, we obtain that
\[
\cH_{|\cdot|}^\n(\B(z,1)\cap N_1)=\cH_{|\cdot|}^\n\lls \B(0,1)\cap(Tv_1+T(N_1))\rrs.
\]
It is not difficult to realize that
\[
T(N_1)=T(S_1\oplus W)=J(S_1)\oplus W=S_2\oplus W=N_2,
\]
due to the definition of $T$ and the inclusions $S_1\subset H^\ell$.
We have proved that
\beq\label{eq:Bz1N_1}
\cH_{|\cdot|}^\n(\B(z,1)\cap N_1)=\cH^\n_{|\cdot|}\lls \B(0,1)\cap(Tv_1+N_2)\rrs.
\eeq
We wish to check whether $(Tv_1)^{-1}$ is a suitable element of $\B(0,1)$.
To do this, we define the subspace 
\[ 
V=H^1\oplus\cdots \oplus H^\ell
 \]
and consider the orthogonal decompositions
\[
z^{-1}=v_1+\eta_1=v+\eta,
\]
where $P_V(z^{-1})=v\in V$, $v_1$ and $v$ are orthogonal to $\eta_1\in N_1$ and
$\eta\in W$, respectively.
The previous equality gives
\[
s_1=v-v_1=\eta_1-\eta\in V\cap (S_1\oplus W)=S_1\subset H^\ell,
\]
hence $v_1$ and $s_1$ are orthogonal.
We write the orthogonal decomposition
\[
v_1=w_1+z_1 \quad \text{with $w_1\in H^1\oplus\cdots\oplus H^{\ell-1}$ and $z_1\in Z_1$}
\]
therefore $v=w_1+z_1+s_1$. Since $s_1$ is orthogonal to $w_1$,
it is also orthogonal to $z_1$, obtaining
\[
|P_{H^\ell}v|=|z_1+s_1|=\sqrt{|z_1|^2+|s_1|^2}\ge|z_1|. 
\]
By the property (1) of Definition~\ref{d:hsym}, 
since $P_V(z^{-1})=v\in P_V(\B(0,1))$, we have
$\psi(u,t)\ge0$ that is monotone nondecreasing with respect to $t$ and
such that
\[
v\in \B(0,1)\cap V=\{y\in V: \psi(P_U(y),|P_{H^\ell}(y)|)\le r_0\},
\]
where $U= H^1\oplus\cdots\oplus H^{\ell-1}$. From the monotonicity of $\psi$ we have
\[
\psi(P_U(v_1),|P_{H^\ell}(v_1)|)=\psi(P_U(v),|z_1|)\le \psi(P_U(v),|P_{H^\ell}(v)|)\le r_0.
\]
Moreover, taking into account that
\[
Tv_1=w_1+Jz_1=P_U(v_1)+Jz_1=P_U(Tv_1)+Jz_1,
\]
there holds 
\[
\psi(P_U(Tv_1),|P_{H^\ell}(Tv_1)|)=\psi(w_1,|Jz_1|)=\psi(w_1,|z_1|)
=\psi(P_U(v_1),|P_{H^\ell}(v_1)|)\le r_0.
\]
We have proved that  
\[
Tv_1\in \B(0,1)\cap V\subset \B(0,1)
\]
and clearly $v_0=(Tv_1)^{-1}\in V\cap\B(0,1)$. 
Due to \eqref{eq:l_pMeasPreserving}, it follows that
\beq\label{eq:J(v_1)N_2)}
\cH^\n_{|\cdot|}\pa{\B(v_0,1)\cap N_2}=\cH^\n_{|\cdot|}\pa{\B(0,1)\cap v_0^{-1}N_2}=\cH^\n_{|\cdot|}\pa{\B(0,1)\cap (Tv_1)N_2}.
\eeq
Since $T$ is an isometry the images $Tv_1$ and $N_2=T(N_1)$
are orthogonal, being so $v_1$ and $N_1$, hence we may apply \eqref{eq:v+N=vN}, getting
\[
\cH^\n_{|\cdot|}\pa{\B(v_0,1)\cap N_2}=\cH^\n_{|\cdot|}\pa{\B(0,1)\cap (Tv_1+N_2)}.
\]
The previous equality joined with \eqref{eq:Bz1N_1} yields
\[
\cH_{|\cdot|}^\n(\B(z,1)\cap N_1)=\cH^\n_{|\cdot|}\pa{\B(v_0,1)\cap N_2}\leq\beta_d(N_2)
\]
and the arbitrary choice of $z\in\B(0,1)$ yields $\beta_d(N_1)\le\beta_d(N_2)$. Exchanging the role of $N_1$ with that of $N_2$, we conclude the proof.
\end{proof}
The intriguing aspect of the previous theorem is that the metric unit ball is not 
assumed to be convex. For instance the sub-Riemannian ball in the Heisenberg group is not convex, but it is 2-vertically symmetric distance, as pointed out in \cite{Mag31}. 

Examples~\ref{ex:htype} and \ref{ex:dinfty} suggest a special form of a homogeneous distance, that are $\n$-vertically symmetric for all $\n=1,\ldots,\q-1$. This more manageable notion of distance is presented in the next definition.
\begin{Def}[Multiradial distance]\label{d:multirad}\rm
The assumptions of Theorem~\ref{t:constBeta} also hold when more generally a homogeneous distance $d:\G\times\G\to\R$ satisfies the following conditions.
There exists $\ph:[0,+\infty)^\iota\to[0,+\infty)$, which is continuous and monotone nondecreasing on each single variable, with
\beq\label{eq:dmulti}
d(x,0)=\ph(|x_1|,\ldots,|x_\iota|),
\eeq
$x_j=P_{H^j}(x)$ and $P_{H^j}:\G\to H^j$ is the canonical projection
with respect to the direct sum decomposition of $\G$ into subspaces $H^j$.
The function $\ph$ is also assumed to be {\em coercive} in the sense that
\[
\ph(x)\to+\infty \qs{as} |x|\to+\infty.
\]
Let us stress that the symbol $|\cdot|$ indicates the Euclidean norm
arising from the fixed graded scalar product, see Section~\ref{sect:notions}.
\ed
\bpr\label{pr:multiradial}
If $d:\G\times\G\to[0,+\infty)$ is multiradial, then it is also
$\n$-vertically symmetric for every $\n=1,\ldots,\iota$. 
\epr
\begin{proof}
We represent the metric unit ball as follows
\beq\label{eq:Bph}
\B(0,1)=\set{x\in\G: \ph(|x_1|,\ldots,|x_\iota|)\le1},
\eeq
observing that in general it need not be convex. 
In the nontrivial case where 
\[
\n=\jmath+\h_{\ell+1}+\cdots+\h_\iota
\]
and $0<\jmath<\h_\ell$, we consider two $j$-dimensional subspaces
$S_1, S_2\subset H^\ell$. Then we can find
an Euclidean isometry $J:H^\ell\to H^\ell$ with $J(S_1)=S_2$,
hence we set $\cF_\ell=O(H^\ell)$.
We construct the isometry $T:\G\to\G$ as follows
\[
T|_{H^\ell}=J \qandq  T|_{H^j}=\Id_{H^j}
\]
for all $j\neq\ell$, taking into account that our fixed scalar product is
such that all subspaces $H^j$ are orthogonal to each other.
The special representation \eqref{eq:Bph} of the metric unit ball clearly gives
\[
T\pa{\B(0,1)}=\B(0,1).
\]
Let us consider the projection $P_V:\G\to V$ with $V=H^1\oplus\cdots\oplus H^\ell$
and observe that
\[
P_V(\B(0,1))=\B(0,1)\cap V=\set{x\in\G: d(x,0)=\ph(|x_1|,\ldots,|x_\ell|,0,\ldots,0)\le1}
\]
in view of the nondecreasing monotonicity of $\ph$ with respect to each single variable.
If we set 
\[
\psi(u,t)=\ph(|P_{H^1}u|,\ldots,|P_{H^{\ell-1}}u|,t,0,\ldots,0),
\]
with $u\in H^1\oplus\cdots\oplus H^{\ell-1}$, then also property (1) 
of Definition~\ref{d:hsym} is established. We have proved that  
$d$ is $\n$-vertically symmetric.
\end{proof}
\br
It is not difficult to observe that the distances of Examples~\ref{ex:dinfty} and \ref{ex:htype} are both multiradial distances.
Thus, as a consequence of Proposition~\ref{pr:multiradial} joined with Theorem~\ref{t:constBeta}, the subsequent formulae \eqref{eq:transv_area_nsymmetric},
\eqref{eq:nonhoriz_nsymmetric} and \eqref{eq:coarea} hold for these distances.
\er
%
\begin{comment}
\bex\rm
If $\|x\|=\sqrt[4]{|x_1|^4+16|x_2|^2}$ denotes the homogeneous norm
already introduced in Example~\ref{ex:htype}
for an H-type group $\G=H^1\oplus H^2$ and 
$\n=1,\ldots,\q-1$, we wish to show that 
$\beta_d(N)$ is constant for every vertical subgroup
of dimension $\n$, where $d(x,y)=\|x^{-1}y\|$.
We set $\h_j=\dim H^j$ with $j=1,2$ and consider first the case
$1\le \n<\h_2$. For $S_1,S_2\subset H^2$ we can clearly find an
Euclidean isometry $J:H^2\to H^2$ with $J(S_1)=S_2$,
therefore we choose $T:\G\to\G$ with
\[
T|_{H^1}=\Id_{H^1} \qandq  T|_{H^2}=J.
\]
Defining the metric unit ball $\B=\set{x\in \G: \|x\|\le1}$, one
easily notices that both $T(\B)\subset\B$ and $T^{-1}(\B)\subset\B$,
hence $T(\B)=\B$. Being $\ell=2$ and $U=H^1$, property (1) of Definition~\ref{d:hsym} follows by setting
\[
\psi(u,t)=\sqrt[4]{|u|^4+16 t^2}
\]
for all $u\in U$ and $t\ge0$. Indeed in this special case $V=\G$ and we have
\[
P_V(\B(0,1))=\B(0,1)=\set{x\in\G: \psi(P_U(x),|P_{H^2}(x)|)\le 1}.
\]
Considering the family $\cO$ of isometries $T$ as above, we finally established 
that $d$ is $\n$-vertically symmetric.
The case $\n=\h_2$ is trivial by definition and when $\n>\h_2$ 
an analogous proof can be achieved.
As a consequence, in all the previous cases we can apply Theorem~\ref{t:constBeta},
getting a constant spherical factor $\beta_d$ on all vertical subgroups of the same dimension. A more general proof is presented in the next example.
\eex
\end{comment}

%
%
%
%
%
\section{Area formulae in homogeneous groups}\label{sect:SphericalMeasure}
%
%
%
%
%
%
%

This section is devoted to a number of applications arising from our main results.
We provide a comprehensive treatment of integral formulae for submanifolds with respect to a homogeneous distance. 

Throughout $\G$ denotes an arbitrary homogeneous group and $\Sigma\subset\G$ is an $\n$-dimensional $C^1$ smooth submanifold of degree $\NN$. Its {\em characteristic set} and its {\em subset of maximum degree} are defined by
\beq\label{eq:S_Sigma-R_Sigma}
\cC_\Sigma=\set{p\in\Sigma: d_\Sigma(p)<\NN}\qandq \cM_\Sigma=\set{p\in\Sigma: d_\Sigma(p)=\NN},
\eeq
respectively. We also fix the intrinsic measure $\mu_\Sigma$,
along with the Riemannian metrics $\tilde g$ and $g$,
as in Definition~\ref{d:SRmeasure}.
We will use the spherical measure $\cS^\NN_0$ of \eqref{d:S_0^alpha},
that does not contain any geometric constant.

%
%
%
%
%
%
\subsection{Transversal and non-horizontal submanifolds}\label{sect:TransvSubmanifolds}
%
%
%
%
%
%

We show how the upper blow-up theorem immediately provides
a general integral formula that relates spherical measure and intrinsic measure
for all transversal submanifolds in any homogeneous group. 
Due to the results of Sections~\ref{sect:Hsym} and \ref{sect:sectConvex},
we will also investigate how symmetry conditions on the homogeneous distance provide simpler integral formulae.

\bt[Transversal submanifolds]\label{t:TransvSub}
If $\Sigma\subset\G$ is an $\n$-dimensional transversal submanifold
of degree $\NN$, then for every Borel set $B\subset \Sigma$ we have
\beq\label{eq:transv_area}
\mu_\Sigma(B)=\int_B \|\tau^{\tilde g}_{\Sigma,\NN}(p)\|_g\, d\sigma_{\tilde g}(p)
=\int_B \beta_d(A_p\Sigma)\, d\cS^\NN_0(p),
\eeq
where $\tilde g$ is any fixed Riemannian metric.
If $d$ is $\n$-vertically symmetric, then for any homogeneous tangent
space $V$ of $\Sigma$, having degree $\NN$, the metric 
factor $\beta_d(V)$ equals a geometric constant $\omega_d(\n,\NN)$
and defining $\cS^\NN_d=\omega_d(\n,\NN)\, \cS^\NN_0$,
there holds
\beq\label{eq:transv_area_nsymmetric}
\cS^\NN_d\res\Sigma(B)=\int_B \|\tau^{\tilde g}_{\Sigma,\NN}(p)\|_g\, d\sigma_{\tilde g}(p).
\eeq
\et
\begin{proof}
From the definitions of \eqref{eq:S_Sigma-R_Sigma}, by
Theorem~1.2 of \cite{Mag26TV} we get $\cS^\NN_0(\cC_\Sigma)=0$.
The definition of intrinsic measure \eqref{d:intmeas} joined with
Remark~\ref{r:degreeTangent} yield $\mu_\Sigma(\cC_\Sigma)=0$.
This allows us to restrict our attention to points of $\cM_\Sigma$.
For every Borel set $E\subset \cM_\Sigma$, each point $p\in E$
has maxium degree, therefore Proposition~\ref{pr:transv_maxdeg}
implies that it is a transversal point.
Then we are in the position to apply part (4) of Theorem~\ref{t:UpBlwC} to each $p\in E$, getting formula \eqref{e:UpBlwSub}. 
The everywhere finiteness of the spherical Federer density
$\theta^\NN(\mu_\Sigma,\cdot)$ shows that 
$\mu_\Sigma\res E$ is absolutely continuous with respect $\cS^\NN_0\res E$
and the measure theoretic area formula \eqref{eq:spharea} applied to
$\mu_\Sigma$ yields
\[
\mu_\Sigma(E)=\int_E \beta_d(A_p\Sigma)\,d\cS_0^\NN(p).
\]
This formula joined with the negligibility of $\cC_\Sigma$
immediately leads us to \eqref{eq:transv_area}.
If in addition $d$ is $\n$-vertically symmetric,
Theorem~\ref{t:constBeta} shows that $\beta_d(N)$
is constant on all $\n$-dimensional vertical subgroups.
We denote this constant by $\omega_d(\n,\NN)$ 
and define the rescaled spherical measure $\cS^\NN_d=\omega_d(\n,\NN)\cS^\NN_0$.
Due to Proposition~\ref{pr:transv_maxdeg}, any homogeneous tangent space $A_p\Sigma$ is a vertical subgroup whenever $p\in\cM_\Sigma$, then
\eqref{eq:transv_area} translates into \eqref{eq:transv_area_nsymmetric}.
\end{proof}
When $k\le \m$, a $k$-codimensional $C^1$ smooth submanifold $\Sigma$
is non-horizontal if and only if it is transversal.
In equivalent terms, $\Sigma$ is non-horizontal if 
and only if its degree is $Q-k$. Let us fix a Riemannian metric $\tilde g$, and consider a
$\tilde g$-unit tangent $(\q-k)$-vector $\tau_\Sigma(p)$ of $\Sigma$ at $p$.
We may define a {\em $\tilde g$-normal} of $\Sigma$ at $p$ as follows
\[
{\bf n}_{\tilde g}(p)=\pm\ast_{\tilde g}(\tau_\Sigma(p)).
\]
Here $\ast_{\tilde g}$ denotes the Hodge operator with respect to the Riemannian metric $\tilde g$ and to the fixed orientation.
We consider the linear mappings $\tilde g^*,g^*:T\G\to T^*\G$
associated to the Riemannian metrics $\tilde g$ and the
fixed left invariant metric $g$. 
Defining the linear mappings $\tilde g^*_k,g^*_k:\Lambda_k(T\G)\to\Lambda^k(T\G)$
canonically associated to $\tilde g^*$ and $g^*$, respectively,
we define the {\em horizontal $k$-normal at $p$} with respect to $\tilde g$ and $g$ as follows
\[
{\bf n}_{g,\tilde g,H}(p)=\pi^0_{p,k}\pa{(g^*_k)^{-1}\tilde g^*_k\pa{{\bf \tilde n}(p)}}.
\]
The projection $\pi^0_{p,k}$ is defined in \eqref{eq:Pi_pM0}.
We assume now that the volume measure $\vol_{\tilde g}$ is left invariant
and define the unique geometric constant $c(g,\tilde g)>0$ such that
\[
c(g,\tilde g)\, \vol_{\tilde g}=\vol_ g.
\]
Since the class of $(\q-k)$-vertically symmetric distances is larger than
the one in \cite[Section~6]{Mag12A},
using Theorem~\ref{t:TransvSub} we obtain an area formula
with constant spherical factor for a large family of distances. 
\bc[Non-horizontal submanifolds]\label{c:Nonhoriz}
Let $1\le k\le m$ and let $\Sigma\subset\G$ be a $C^1$ smooth 
submanifold of codimension $k$ and degree $Q-k$.
Then for every Borel set $B\subset \Sigma$ there holds
\beq\label{eq:nonhoriz}
\mu_\Sigma(B)=c(g,\tilde g) \int_B \|{\bf n}_{g,\tilde g,H}(p)\|_g\, d\sigma_{\tilde g}(p)
=\int_B \beta_d(A_p\Sigma)\, d\cS^{Q-k}_0(p),
\eeq
where $\tilde g$ is any Riemannian metric, whose volume measure is left invariant.
In the case $d$ is $(\q-k)$-vertically symmetric, 
then for any homogeneous tangent space $V$ of $\Sigma$, having degree $Q-k$, the metric 
factor $\beta_d(V)$ equals the geometric constant $\omega_d(\q-k,Q-k)$
and defining $\cS^{Q-k}_d=\frac{\omega_d(\q-k,Q-k)}{c(g,\tilde g)}\, \cS^{Q-k}_0$,
there holds
\beq\label{eq:nonhoriz_nsymmetric}
\cS^{Q-k}_d\res\Sigma(B)=\int_B \|{\bf n}_{g,\tilde g,H}(p)\|_g\, d\sigma_{\tilde g}(p).
\eeq
\ec
\begin{proof}
The claims are a straightforward consequence of \eqref{eq:transv_area} and
of \eqref{eq:transv_area_nsymmetric}, joined with formula (12) of \cite{Mag12A}.
\end{proof}
\br
It is well known for instance that any $C^1$ smooth hypersurface $\Sigma\subset\G$ is automatically a non-horizontal submanifold, therefore formulae
\eqref{eq:nonhoriz} and \eqref{eq:nonhoriz_nsymmetric} hold
for any $C^1$ smooth hypersurface of a stratified group $\G$.
The local isoperimetric inequality applied to a suitably ``small'' open subset $U\subset\Sigma$ of $\Sigma$ shows that it must have positive $\cS^{Q-1}_0$ measure. 
Since characteristic points are $\cS^{Q-1}_0$ negligible \cite{Mag5},
the piece $U$ must contain non-characteristic points. 
This shows that $\Sigma$ has degree $Q-1$. 
\er
In any homogeneous group $\G$ equipped with graded coordinates $x_j$,
the corresponding basis of left invariant vector fields $X_1,\ldots,X_\q$ has the form
\[
X_j=\der_{x_j}+\sum_{d_l>d_j} a_{jl} X_l,
\]
and it is automatically assumed to be orthonormal with respect to
the fixed left invariant metric $g$.
The dual basis of left invariant differential forms has the form
\[
\xi_j=dx_j+\sum_{d_l>d_j} b_{jl}\, dx_l.
\]
The special form of the left invariant differential forms $\xi_j$ implies that
\[
\xi_1\wedge \xi_2\wedge \cdots \wedge \xi_\q=dx_1\wedge dx_2\wedge \cdots \wedge dx_\q,
\]
therefore using the Euclidean metric in place of $\tilde g$ yields
\[
c(g,\tilde g)=1.
\]
This simplifies formula \eqref{eq:nonhoriz} and it allows one to use the Euclidean metric to compute the spherical measure of a submanifold.

\begin{Exa}\rm
Using graded coordinates $x_j$ in a homogeneous group $\G$
we consider the standard Euclidean metric $\tilde g$ given by $\delta_{ij}$,
along with our fixed left invariant Riemannian metric $g$.
Let $\Sigma$ be a non-horizontal submanifold of codimension $k$
with unit normal ${\bf n}={\bf n}_{\tilde g}$ with repect to the Euclidean metric $\tilde g$. Setting 
${\bf n}_g={\bf n}_{g,\tilde g}$ and ${\bf n}_{g,H}={\bf n}_{g,\tilde g,H}$, there holds
\[
\|{\bf n}_{g,H}\|_g=\sqrt{\sum_{1\le j_1<j_2<\cdots<j_k\le \m}
\ban{{\bf n},X_{j_1}\wedge \cdots \wedge X_{j_k}}_E^2},
\]
where $\ban{\cdot,\cdot}_E$ denotes the Euclidean scalar product
with respect to the graded coordinates $x_j$.
Indeed by definition of ${\bf n}_g$ and ${\bf n}_{g,H}$, there holds
\[
\ban{{\bf n}_g,X_{j_1}\wedge \cdots \wedge X_{j_k}}_g=
\ban{{\bf n},X_{j_1}\wedge \cdots \wedge X_{j_k}}_E.
\]
Since $c(g,\tilde g)=1$, formula \eqref{eq:nonhoriz_nsymmetric} gives
a more explicit area formula
\[
\cS^{Q-k}_d(B)=\int_B \sqrt{\sum_{1\le j_1<j_2<\cdots<j_k\le \m}
\ban{{\bf n},X_{j_1}\wedge \cdots \wedge X_{j_k}}_E^2}\; d\cH^{\q- k}_E(p)
\]
for every Borel set $B\subset \Sigma$ and for aby $(\q-k)$-vertically symmetric homogeneous distances $d$, where $\cH^{\q-k}_E$ denotes the $(\q-k)$-dimensional Hausdorff measure with respect to the Euclidean
distance and we have set
\[
\cS^{Q-k}_d=\omega_d(\q-k,Q-k)\, \cS^{Q-k}_0.
\] 
In particular, since all smooth hypersurfaces are non-horizontal submanifolds,
the previous formula for $k=1$ yields the well known formula for the
spherical measure $\cS^{Q-1}$ of hypersurfaces
\[
\cS^{Q-1}_d(\Sigma)=\int_\Sigma 
\sqrt{\sum_{j=1}^{\m}
\ban{{\bf n},X_j}_E^2}\; d\cH^{\q-1}_E(p)
\]
that is here extended to the largest class of suitably symmetric distances,
namely the $(q-1)$-vertically symmetric distances, according to
Definition~\ref{d:hsym}.
\end{Exa}
The general notion of vertically symmetric distance has effects also 
on the known coarea formulae, extending the one of \cite{Mag12A}
to this class of distances.

\bc[Coarea formula]\label{c:Coarea}
Let $1\le k\le \m$, let $f:A\to\R^k$ be a Riemannian Lipschitz map,
where $A\subset\G$ is measurable and consider a homogeneous distance 
$d$ on $\G$ that is $(\q-k)$-vertically symmetric.
By Theorem~\ref{t:constBeta}, the geometric constant 
$\omega_d(\q-k,Q-k)=\beta_d(V)$ does not depend on the choice of
the vertical subgroup $V\subset\G$ of dimension $\q-k$.  
Thus, defining the rescaled spherical measure
\[
\cS^{Q-k}_d=\omega_d(\q-k,Q-k)\, \cS^{Q-k}_0,
\]
for any nonnegative measurable function $u:A\to\R$, we have
\beq\label{eq:coarea}
\int_A u(x) J_{g,H}f(x)\, d\vol_g(x) =\int_{\R^k}\pa{\int_{f^{-1}(t)}
 u(x)\, \cS^{Q-k}_d(x)}dt,
\eeq
where $\vol_g$ is the Riemannian volume measure on $\G$ and 
\[
J_{g,H}(x)=\|\pi_{x,k}(\nabla f_1(x)\wedge\nabla f_2(x)\wedge \cdots\wedge \nabla f_k(x)) \|
\]
is the horizontal Jacobian at every differentiability point $x\in A$ of $f$.
\ec
The arguments to establish this corollary are the same of Theorem~1.2 in \cite{Mag12A}, joined with the area formula \eqref{eq:nonhoriz_nsymmetric}.
Considering the Riemannian metric $\tilde g$ and the spherical
measure $\cS^{Q-k}_d$ of Corollary~\ref{c:Nonhoriz}, one immediately
realizes that 
\beq\label{eq:coarea_gtilde}
\int_A u(x) J_{g,H}f(x)\, d\vol_{\tilde g}(x) =\int_{\R^k}\pa{\int_{f^{-1}(t)}
 u(x)\, \cS^{Q-k}_d(x)}dt,
\eeq
under the assumptions of Corollary~\ref{c:Coarea}.

%
%
%
%
%
%
\subsection{Submanifolds in two step groups}\label{sect:HorRectHomGroups}
%
%
%
%
%
%
This section presents some results to compute the spherical measure
of submanifolds in two step groups.
The main point here is that multiradial distances yield constant spherical factor in any two step homogeneous group.

\bpr\label{pr:MetricStep2}
If $\G$ is a step two homogeneous group and $d$ is a multiradial distance, then for every $\n$-dimensional homogeneous subspace $V\subset \G$ we have
\beq\label{eq:betanStep2}
\beta_d(V)=\cH_{|\cdot|}^\n(\B\cap V),
\eeq
where $1\le \n\le\q-1$ and $\B=\set{x\in\G: d(x,0)\le 1}$.
\epr
\begin{proof}
Denote $\B=\B(0,1)$, choose $z\in\B$ and write $V=V_1\oplus V_2$ with
$V_j\subset H^j$, being $V$ a homogeneous subspace of $\G$.
Then the assumptions on $d$ ensure that $\B$ is defined as 
in \eqref{eq:Bph}, therefore we get
\[
V\cap\B(z,1)=\set{v\in V: \ph(|P_{H^1}(z^{-1}v)|,|P_{H^2}(z^{-1}v|)\le 1 }.
\]
The BCH formula \eqref{eq:BCH} yields
\[
V\cap\B(z,1)=\set{v_1+v_2\in V: \ph\pa{|v_1-z_1|,\pal{v_2-z_2-\frac12[z_1,v_1]}}\le 1 }
\]
with $z_j=P_{H^j}(z)$ and $v_j=P_{H^j}(v)$.
From the coercivity of $\ph$ we can define 
\[
r_1=\sup\set{t\ge0: \ph(t,0)\le 1},
\]
then considering an orthogonal system of coordinates on $V$
and denoting by $\cL^\n$ the corresponding Lebesgue measure on $V$, 
Fubini's theorem yields
\beq\label{eq:H^nB(0,1)}
\begin{split}
\cH^\n_{|\cdot|}(V\cap\B(z,1))&=\cL^\n(V\cap\B(z,1)) \\
&=\int_{V_1\cap B_E(z_1,r_1)} \cL^{\n_2}\pa{\set{v_2\in V_2:
v_1+v_2\in \B(z,1)}} dv_1
\end{split}
\eeq
where $B_E(z,r)=\set{x\in\G: |x-z|<r}$ and $\n_j=\dim V_j$.
From the nondecreasing monotonicity of $\ph$ with respect to each variable, defining
\[
\rho(x)=\sup\set{t\ge0:\ph(|x|,t)\le 1}
\]
whenever $|x|<r_1$ gives the monotonicity
\beq\label{eq:monotwj}
\rho(w_2)\le \rho(w_1) \qs{for} |w_1|\le |w_2|<r_1.
\eeq
As a consequence, using the integral representation \eqref{eq:H^nB(0,1)} for $z=0$ we get the formula
\beq\label{eq:H^nrho}
\begin{split}
\cH^\n_{|\cdot|}(V\cap \B(0,1))&=\int_{V_1\cap B_E(0,r_1)}
\cL^{\n_2}\pa{\set{v_2\in V_2: \ph(|v_1|,|v_2|)\le 1}} dv_1 \\
&=\int_{V_1\cap B_E(0,r_1)} \cL^{\n_2}\lls V_2\cap B_E(0,\rho(v_1))\rrs dv_1,
\end{split}
\eeq
that will be used later.
Defining the function $\Psi(z,v_1)=z_2+\frac12[z_1,v_1]$, we have
\[
\begin{split}
\cH^\n_{|\cdot|}(V\cap\B(z,1))&=\int_{V_1\cap B_E(z_1,r_1)} \cL^{\n_2}\pa{V_2\cap B_E\lls\Psi(z,v_1),\rho(v_1-z_1)\rrs} dv_1 \\
&=\int_{(V_1-z_1)\cap B_E(0,r_1)} \cL^{\n_2}\pa{V_2\cap B_E\lls\Psi(z,v_1),\rho(v_1)\rrs} dv_1 \\
&\le \int_{V_1\cap B_E(0,r_1)} \cL^{\n_2}\pa{V_2-\Psi(z,v_1)\cap B_E\lls0,\rho(v_1)\rrs} dv_1.
\end{split}
\]
Taking into account the last inequality, Theorem~\ref{t:ConvSect_n-dim}
and formula \eqref{eq:H^nrho}, we get
\[
\cH^\n_{|\cdot|}(V\cap\B(z,1))\le \int_{V_1\cap B_E(0,r_1)} \cL^{\n_2}\pa{V_2\cap B_E\lls0,\rho(v_1)\rrs} dv_1=\cH^{\n}_{|\cdot|}(V\cap \B),
\]
therefore concluding the proof.
\end{proof}

\bpr\label{pr:constantmetfact2step}
Let $\G$ be a step two homogeneous group and let $d$ be a multiradial distance.
Fix two integers $\n_1\le \h_1$ and $\n_2\le\h_2$. Then the spherical factor
with respect to $d$ is constant on all homogeneous subspaces $V=V_1\oplus V_2\subset \G$ with $\dim V_1=\n_1$ and $\dim V_2=\n_2$.
\epr
\begin{proof}
Let us consider two homogeneous subspaces $V$ and $W$ with direct decompositions
$V_1\oplus V_2$ and $W_1\oplus W_2$, respectively, where
\[
\dim V_1=\dim W_1=\n_1\qandq \dim V_2=\dim W_2=\n_2.
\]
Considering two isometries $J_1:H^1\to H^1$ and $J_2:H^2\to H^2$ such that 
\[
J_1(V_1)=W_1\qandq J_2(V_2)=W_2,
\]
respectively, then for every $x_1\in H^1$ and $x_2\in H^2$ we define
\[
T(x_1+x_2)=J_1(x_1)+J_2(x_2).
\]
The mapping $T:\G\to\G$ is an isometry since $H_1$ is orthogonal to $H_2$
with respect to our fixed graded scalar product on $\G$.
Moreover, the shape of the metric unit ball $\B$ gives the equalities
\[
\begin{split}
T(\B\cap V)&=T\pa{\set{x_1+x_2\in\G: x_1\in V_1,\ x_2\in V_2,\ \ph(|x_1|,|x_2|)\le 1 }} \\
&=\set{T(x_1)+T(x_2)\in\G: x_1\in V_1,\ x_2\in V_2,\ \ph(|x_1|,|x_2|)\le 1 } \\
&=\set{T(x_1)+T(x_2)\in\G: x_1\in V_1,\ x_2\in V_2,\ \ph(|T(x_1)|,|T(x_2)|)\le 1 } \\
&=\set{y_1+y_2\in\G: y_1\in W_1,\ y_2\in W_2,\ \ph(|y_1|,|y_2|)\le 1 } \\
&=\B\cap W.
\end{split}
\]
As a consequence, formula \eqref{eq:betanStep2} concludes the proof.
\end{proof}

\bt[Intrinsic measure in two step groups]\label{t:area2steps}
If $\G$ has step two, $\cS_0^\NN\pa{\cC_\Sigma}=0$ and $p'\in\Sigma$ 
is algebraically regular for all $p'\in \cM_\Sigma$, then for every Borel set 
$B\subset \Sigma$ we have
\beq\label{eq:area2step}
\mu_\Sigma(B)=\int_B \|\tau^{\tilde g}_{\Sigma,\NN}(p)\|_g\, d\sigma_{\tilde g}(p)
=\int_B \beta_d(A_p\Sigma)\, d\cS^\NN_0(p),
\eeq
where $\tilde g$ is any fixed Riemannian metric.
If $d$ is multiradial, then for any homogeneous tangent
space $V$ of $\Sigma$, having degree $\NN$, the metric 
factor $\beta_d(V)$ equals a geometric constant $\omega_d(\n,\NN)$.
Thus, defining $\cS^\NN_d=\omega_d(\n,\NN)\, \cS^\NN_0$, the following area formula holds
\beq\label{eq:area2stepmultiradial}
\cS^\NN_d\res\Sigma(B)=\int_B \|\tau^{\tilde g}_{\Sigma,\NN}(p)\|_g\, d\sigma_{\tilde g}(p).
\eeq
\et
\begin{proof}
Our assumptions joined with definition \eqref{eq:S_Sigma-R_Sigma}, 
the formula for the intrinsic measure \eqref{d:intmeas} and 
Remark~\ref{r:degreeTangent} yield the conditions 
\beq\label{eq:NegligNN}
\mu_\Sigma(\cC_\Sigma)=\cS^\NN_0(\cC_\Sigma)=0.
\eeq
If $E\subset \cM_\Sigma$ is any Borel set, we may apply part (2) of Theorem~\ref{t:UpBlwC} to each $p\in E$, since all of these points are algebraically regular.
This allows us to establish \eqref{e:UpBlwSub}. 
In particular, the spherical Federer density
$\theta^\NN(\mu_\Sigma,\cdot)$ is everywhere finite on $E$, hence
$\mu_\Sigma\res E$ is absolutely continuous with respect $\cS^\NN_0\res E$.
We are then in the condition to apply the measure theoretic area formula \eqref{eq:spharea}, obtaining that
\[
\mu_\Sigma(E)=\int_E \beta_d(A_p\Sigma)\,d\cS_0^\NN(p)\,,
\]
therefore \eqref{eq:area2step} holds. The previous equality
joined with \eqref{eq:NegligNN} leads us to \eqref{eq:area2step}.

Let us now assume that $d$ is multiradial, and consider any
point $p\in\cM_\Sigma$. 
Taking into account Proposition~\ref{pr:translation}, Remark~\ref{r:localdegreealpha_i} and Theorem~\ref{t:specialcoord} there exist two positive integers $\alpha_1\le\h_1$ and $\alpha_2\le\h_2$ such that 
\[
\alpha_1+\alpha_2=\n\qandq \alpha_1+2\alpha_2=\NN.
\]
The previous conditions uniquely define both 
\[
\alpha_1=2\n-\NN \qandq \alpha_2=\NN-\n. 
\]
From Proposition~\ref{pr:homtansp}, any 
homogeneous tangent space $V$ at a point of $\cM_\Sigma$
has the direct decomposition
\[
V=V_1\oplus V_2,\quad V_1\subset H^1,\quad V_2\subset H^2, \quad \dim V_1=\alpha_1\qandq \dim V_2=\alpha_2.
\]
We are now in the position to apply Proposition~\ref{pr:constantmetfact2step}
to all these homogeneous tangent spaces $V$, getting a constant
that we denote by $\omega_d(\n,\NN)$, such that 
\[
\beta_d(V)=\omega_d(\n,\NN).
\]
Thus formula \eqref{eq:area2step} and the definition
$\cS^\NN_d=\omega_d(\n,\NN)\, \cS^\NN_0$ conclude the proof.
\end{proof}
A special application of the previous theorem is the following.
\bc\label{c:2stepC^11}
If $\Sigma$ is a $C^{1,1}$ smooth submanifold in a two step group $\G$,
then formula \eqref{eq:area2step} holds. In addition, if $d$ is multiradial, 
we can find $\omega_d(\n,\NN)$ and define $\cS^\NN_d$ as in Theorem~\ref{t:area2steps} such that \eqref{eq:area2stepmultiradial} holds.
\ec

\begin{proof}
By \cite[Corollary~1.2]{Mag12B} the condition $\cS_0^\NN\pa{\cC_\Sigma}=0$
is satisfied. Due to the $C^{1,1}$ smoothness we may apply \cite[Lemma~3.9]{Mag13Vit}, that joined with Proposition~\ref{pr:homtansp} and Proposition~\ref{pr:translation} proves that every point of $\cM_\Sigma$ 
is algebraically regular. Thus, arguing as in the proof of Theorem~\ref{t:area2steps} our claim is immediately achieved.
\end{proof}
A concrete application of the previous corollary is given in the next example.
\bex\rm
Let us consider the submanifold $\Sigma=\Phi(U)$,
where $U\subset\R^\n$ is an open, bounded and connected set.
Assume that $\tilde U$ is an open set with $\overline U\subset\tilde U$
and $\Phi:\tilde U\to\G$ is a $C^{1,1}$ smooth embedding
into an H-type group $\G$. Consider $\G$ equipped with the multiradial distance
\[
d(x,0)=\sqrt[4]{|x_1|^4+16|x_2|^2}, 
\]
as in Example~\ref{ex:htype} and let $|\cdot|$ the fixed Euclidean norm 
on $\G$. Let $y=(y_1,\ldots,y_\n)$ be the coordinates of $\Sigma$, let $\NN$ be its degree and define the constant
\[
\omega_d(\n,\NN)=\cH^\n_{|\cdot|}\pa{\set{x\in\G: d(x,0)\le1}\cap V}
\]
where $V=V_1\oplus V_2$ is any homogeneous subspace
of $\G$ with $\dim V_1=\n_1$, $\dim V_2=\n_2$ and
$\NN=\n_1+2\n_2$. By Proposition~\ref{pr:constantmetfact2step}
the geometric constant $\omega_d(\n,\NN)$ is independent of 
the choice of $V$, with the previous dimensional constraints.
Defining the spherical measure $\cS^\NN_d=\omega_d(\n,\NN)\cS^\NN_0$,
as a consequence of Corollary~\ref{c:2stepC^11} and of formula
\eqref{e:intmeaslocalchart}, we get
\beq\label{eq:SNarea2step}
\cS^\NN_d\pa{\Sigma}=\int_U \|\pi_{\Phi(y),\NN}\lls \der_{y_1}\Phi(y)\wedge\cdots\wedge\der_{y_\n}\Phi(y)\rrs\|_g\,dy,
\eeq
where $g$ is the left invariant Riemannian metric associated
to the fixed Euclidean norm on $\G$, as indicated in Section~\ref{sect:Graded}.
Clearly, whenever we have a multiradial distance $d$ on a two step
homogeneous group the area formua \eqref{eq:SNarea2step} holds,
with a different geometric constant $\omega_d(\n,\NN)$ defining
$\cS^\NN_d$.
\eex

%
%
%
%
%
%
\subsection{Curves in homogeneous groups}\label{sect:MeasCurves}
%
%
%
%
%
%

In this section
we see how the formula for the spherical measure of curves
takes a simpler form for multiradial distances, hence 
improving previous some results of \cite{Mag21Korte}.

\bt[Intrinsic measure of curves]\label{t:curves}
Let $\Sigma\subset\G$ be a $C^1$ smooth embedded curve
of degree $\NN$, let $\tilde g$ be a fixed Riemannian metric and 
consider a Borel set $B\subset\Sigma$. There holds
\beq\label{eq:measure_curves}
\int_B \|\tau^{\tilde g}_{\Sigma,\NN}(p)\|_g\, d\sigma_{\tilde g}(p)
=\int_B \beta_d(A_p\Sigma)\, d\cS^\NN_0(p).
\eeq
If the homogeneous distance $d$ on $\G$ is multiradial, then the
spherical factor $\beta_d(L)$ is constant on all one dimensional subspaces
$L\subset H^\NN$. Thus, denoting this constant by 
$\omega_d(1,\NN)$ and setting $\cS^\NN_d=\omega_d(1,\NN)\, \cS^\NN_0$,  for evey Borel set $B\subset\Sigma$ we get
\beq\label{eq:measure_curves_multiradial}
\cS^\NN_d\res\Sigma(B)=\int_B \|\tau^{\tilde g}_{\Sigma,\NN}(p)\|_g\, d\sigma_{\tilde g}(p).
\eeq
\et
\begin{proof}
Taking into account the definitions \eqref{eq:S_Sigma-R_Sigma},
Theorem~1.1 of \cite{Mag21Korte} gives 
\beq\label{eq:negligib_curves}
\cS_0^\NN(\cC_\Sigma)=0.
\eeq
From the definition of intrinsic measure and taking into account
Remark~\ref{r:degreeTangent}, one also notices that
$\mu_\Sigma(\cC_\Sigma)=0$.
At any point $p$ of $\cM_\Sigma$, the $\NN$-projection $\pi_{p,\NN}(\tau_\Sigma(p))$ is obviously a vector, hence the homogeneous tangent space $A_p\Sigma$ is automatically a one dimensional subgroup of $\G$.
This shows that any point of $\cM_\Sigma$ is algebraically regular.
As a consequence, considering any Borel set $E\subset\cM_\Sigma$,
we apply part (3) of Theorem~\ref{t:UpBlwC} at each point $p\in E$, getting 
\beq\label{eq:betaA_pSigma}
\theta^\NN(\mu_\Sigma,x)=\beta_d(A_p\Sigma).
\eeq
In particular, the finiteness of the spherical Federer density $\theta^\NN(\mu_\Sigma,\cdot)$ on $E$ yields the absolute continuity of $\mu_\Sigma\res E$ with respect to $\cS_0^\NN\res E$.
Joining the measure theoretic area formula \eqref{eq:spharea} with the
negligibility condition \eqref{eq:negligib_curves} our claim \eqref{eq:measure_curves} follows.

We now assume the $d$ is multiradial and
set $\B=\set{x\in\G: d(x,0)\le1}$. Let us consider 
a one dimensional subspace $L$ of $H^\NN$ and observe that
\[
L\cap\B(z,1)=\set{v\in L:  z^{-1}v\in\B }.
\]
Since $d$ is multiradial, we have
\[
L\cap\B(z,1)
=\set{v\in L: \ph(|P_{H^1}(z^{-1}v)|,\ldots,|P_{H^\iota}(z^{-1}v|)\le 1 }.
\]
The nondecreasing monotonicity of $\ph$ with respect to each variable shows that
\[
L\cap \B(z,1)\subset\Big\{v\in V: 
\ph\lls\underbrace{0,\ldots,0}_{\NN-1\ \text{zeros}},|v-P_{H^\NN}(z)|,0,\ldots,0\rrs\le 1\Big\}=\zeta_\NN+C, 
\]
where $\zeta_\NN=P_{H^\NN}(z)$ and 
$J_\NN=\set{v\in L: \ph(0,\ldots,0,|v|,0,\ldots,0)\le 1}$.
The previous inclusion yields
\[
\cH^1_{|\cdot|}\pa{L\cap \B(z,1)}\le 
\cH^1_{|\cdot|}\pa{\zeta_1+J_\NN}=\cL^1(J_\NN).
\]
Observing that   
\[
\cH^1_{|\cdot|}\pa{L\cap \B}=
\cH^1_{|\cdot|}\pa{\set{v\in L: \ph(0,\ldots,0,|v|,0,\ldots,0)\le 1}}
=\cL^1(J_\NN),
\]
we have shown that $\beta_d(L)=\cL^1(J_\NN)$, independently
of the choice of the linear subspace $L$ in $H^\NN$.
Finally, defining $\omega_d(1,\NN)=\beta_d(L)$ and 
$\cS^\NN_d=\omega_d(1,\NN)\cS^\NN_0$, formula \eqref{eq:measure_curves} takes the form \eqref{eq:measure_curves_multiradial}, concluding the proof.
\end{proof}

\br
Notice that \eqref{eq:measure_curves_multiradial} holds
also for 1-vertically symmetric distances, assuming that 
the curve is transversal, see \eqref{eq:transv_area_nsymmetric}.
In particular, we have the following area formula 
\beq\label{eq:measure_curves_Heisenberg}
\cS^2_d\res\Sigma(B)=\int_B \|\tau^{\tilde g}_{\Sigma,\NN}(p)\|_g\, d\sigma_{\tilde g}(p).
\eeq
for any $C^1$ smooth nonhorizontal curve $\Sigma\subset\H^n$
of the $(2n+1)$-dimensional Heisenberg group
and for any homogeneous distance.
Indeed, from Definition~\ref{d:hsym} one may easily notice that
in $\H^n$ any homogeneous distance is automatically
1-vertically symmetric. In particular, formula \eqref{eq:measure_curves_Heisenberg} holds
for the sub-Riemannian distance of $\H^n$, that
is not multiradial.
\begin{comment}
se partiamo da uno dei punti della sfera subriemanniana
avente la terza coordinata pi\`u grande (costituiscono una circonferenza)
e mantenendo la stessa altezza avviciniamo tale punto
all'asse verticale il punto esce dalla sfera e quindi aumenta
la distanza dall'origine pur avendo diminuito il modulo
della proiezione del punto sullo spazio orizzontale.
\end{comment}
\er

%
%
%
%
%
%
\section{Spherical measure of horizontal submanifolds}\label{sect:HorizSubmanifolds}
%
%
%
%
%
%
%

Our aim in this section is to establish integral formulae for the spherical measure of horizontal submanifolds and to relate it to their Hausdorff measure.

\bpr\label{pr:MetricHoriz}
If $d$ is a multiradial distance on $\G$, then there exists
a geometric constant $\omega_d(\n,\n)$ such that 
\beq\label{eq:betanV}
\beta_d(V)=\cH^\n_{|\cdot|}(\B\cap V)=\omega_d(\n,\n)
\eeq
for any $\n$-dimensional horizontal subgroup $V\subset H^1$,
where $\B=\set{x\in\G: d(x,0)\le 1}$.
\epr
\begin{proof}
Consider an $\n$-dimensional horizontal subgroup $V\subset H^1$ and the intersection
\[
V\cap\B(z,1)=\set{v\in V: z^{-1}v\in\B }.
\]
Since $d$ is multiradial, we have 
\[
V\cap\B(z,1)=\set{v\in V: \ph(|P_{H^1}(z^{-1}v)|,\ldots,|P_{H^\iota}(z^{-1}v|)\le 1 }
\]
and the monotonicity properties of $\ph$ give
\[
V\cap \B(z,1)\subset\set{v\in V: 
\ph(|v-P_{H^1}(z)|,0,\ldots,0)\le 1}=\zeta_1+C, 
\]
where $\zeta_1=P_{H^1}(z)$ and $C=\set{v\in V: \ph(|v|,0,\ldots,0)\le 1}$.
It follows that
\[
\cH^\n_{|\cdot|}\pa{V\cap \B(z,1)}\le 
\cH^\n_{|\cdot|}\pa{\zeta_1+C}=\cH^\n_{|\cdot|}(C)
=\cH^\n_{|\cdot|}\pa{\B\cap V}.
\]
This proves \eqref{eq:betanV}, along with the fact that
$\beta_d(V)$ does not depend on the choice of the 
$\n$-dimensional horizontal subgroup $V$.
\end{proof}

\bt[Horizontal submanifolds]\label{t:AreaHoriz}
If $\Sigma\subset\G$ is an $\n$-dimensional horizontal submanifold, 
then for every Borel set $B\subset \Sigma$ we have
\beq\label{eq:horiz_area}
\mu_\Sigma(B)=\int_B \|\tau^{\tilde g}_{\Sigma,\NN}(p)\|_g\, d\sigma_{\tilde g}(p)
=\int_B \beta_d(A_p\Sigma)\, d\cS^\n_0(p),
\eeq
where $\tilde g$ is any fixed Riemannian metric.
If $d$ is multiradial, then for any homogeneous tangent
space $V$ of $\Sigma$, the spherical factor $\beta_d(V)$ 
equals a geometric constant $\omega_d(\n,\n)$ and defining $\cS^\n_d=\omega_d(\n,\n)\, \cS^\n_0$, there holds
\beq\label{eq:horiz_area_nsymmetric}
\cS^\n_d\res\Sigma(B)
=\int_B \|\tau^{\tilde g}_{\Sigma,\NN}(p)\|_g\, d\sigma_{\tilde g}(p).
\eeq
\et
\begin{proof}
From the definition of horizontal submanifold, all points
of $\Sigma$ are algebraically regular. In view of Remark~\ref{r:horiz_degree},
we also observe that all points of $\Sigma$ have degree $\n$. 
As a result, $\Sigma=\cM_\Sigma$ and choosing any Borel set $B\subset\Sigma$, 
we apply part (1) of Theorem~\ref{t:UpBlwC} to each $p\in B$,
getting formula \eqref{e:UpBlwSub}.
The everywhere finiteness of the spherical Federer density on $B$ joined with
the measure theoretic area formula \eqref{eq:spharea} 
lead us to the integration formula \eqref{eq:horiz_area}. 
In the case $d$ is multiradial, Proposition~\ref{pr:MetricHoriz} allows
us to define the geometric constant $\omega_d(\n,\n)=\beta_d(V)$,
independent of the choice of the homogeneous tangent space $V$ 
at any point of $\Sigma$. Then \eqref{eq:horiz_area} immediately
gives \eqref{eq:horiz_area_nsymmetric}, concluding the proof.
\end{proof}

As a consequence of the previous theorem,
joined with the area formula of \cite{Mag}, we will 
find the formula relating spherical measure
and Hausdorff measure on horizontal submanifolds.

A multiradial distance $d$ is fixed from now on.
We consider the set function $\phi^k_\delta$ of \eqref{d:phialpha}
with respect to $d$, where $\cF$ is the family of all closed sets
$\alpha=\n$. In the sequel,
we consider the Hausdorff measure
\beq\label{d:chH^n_d}
\cH^\n_0(E)=\sup_{\delta>0}\phi_\delta^\n(E)
\eeq
for every $E\subset\G$. We will also use the ``normalized Hausdorff measure''
\beq\label{d:omega_dH^n}
\cH^\n_d=\omega_d(\n,\n) \cH^\n_0,
\eeq
where $\omega_d(\n,\n)$ is defined in \eqref{eq:betanV}.

\bl\label{l:inv_multiradial}
Let $d$ be a multiradial distance on $\G$ and let
$V,W\subset H^1$ be horizontal subgroups of dimension $\n$. 
It follows that 
\beq\label{eq:HausIsod}
\cH^\n_0(\B\cap V)=\cH^\n_0(\B\cap W)=1,
\eeq
with $\B=\set{x\in\G: d(x,0)\le 1}$.
\el
\begin{proof}
Let us consider the Euclidean isometry $\widetilde T:V\to W$
with respect to the fixed graded Euclidean norm $|\cdot|$ on $\G$.
The same scalar product defines the multiradial distance, see \eqref{eq:dmulti}.
For each $x,y\in V$ there holds
\[
d(\widetilde Tx,\widetilde Ty)=d((\widetilde Tx)^{-1}\widetilde T(y),0)=
d(\widetilde Ty-\widetilde Tx,0)=d(\widetilde T(y-x),0)
\]
where the last equality follows by the fact that $W$ is commutative.
By definition of $d$, we have
\[
d(\widetilde Tx,\widetilde Ty)=\ph(|\widetilde Ty-\widetilde Tx|,0,\ldots,0)
=\ph(|y-x|,0,\ldots,0)=d(x,y),
\]
where the last equality holds, due to the commutativity of $V$.
Choosing proper orthonormal bases, we extend $\widetilde T$ to an isometry
$T:H^1\to H^1$ with respect to $|\cdot|$ such that 
$T|_V=\widetilde T$. Since $d$ is multiradial it is easy to observe that 
\[
T(\B\cap H^1)=\B\cap H^1.
\]
By definition of $\widetilde T$, it follows that
\beq\label{eq:TBcapV}
T(\B\cap V)=T(\B\cap H^1\cap V)=\B\cap H^1\cap \widetilde T(V)=\B\cap W.
\eeq
We now consider the Hausdorff measures  
\[
\widetilde\cH^\n_V:\cP(V)\to[0,+\infty] \qandq 
\widetilde\cH^\n_W:\cP(W)\to[0,+\infty]
\]
defined in \eqref{d:chH^n_d}, but where the metric space $\G$
is replaced by the horizontal subgroups $V$ and $W$, respectively.
Since $\widetilde T$ is an isometry also with respect to $d$,
taking into account \eqref{eq:TBcapV} and the standard
property of Lipschitz functions with respect to the Hausdorff measure,
we get
\[
\widetilde\cH^\n_W(\B\cap W)=
\widetilde\cH^\n_W(\widetilde T(\B\cap V))=
\widetilde\cH^\n_V(\B\cap V).
\] 
Exploiting the special property of the Hausdorff meausure about restrictions
\[
\widetilde\cH^\n_V=\cH^\n_0|_{\cP(V)} \qandq 
\widetilde\cH^\n_W=\cH^\n_0|_{\cP(W)}
\]
the first equality of \eqref{eq:HausIsod} follows.
Finally, we apply the isodiametric inequality in finite dimensional Banach spaces, see for instance \cite[Theorem 11.2.1]{BuragoZalgaller1988},
and observe that the restriction $\|x\|_d=d(x,0)$ 
for $x\in V$  yields a Banach norm, due to the commutativity of $V$.
By standard arguments, the isodiametric inequality in the Banach space
$(V,\|\cdot\|_d)$ gives $\widetilde \cH_V^\n(\B\cap V)=1$, therefore concluding the proof.
\end{proof}
We can now prove the main result of this section.
\begin{proof}[Proof of Theorem~\ref{t:HausdSpheric}]
Since our argument is local, it is not restrictive to
consider an open set $\Omega\subset\R^\n$ and 
assume that there exists a $C^1$ smooth 
embedding $\Psi:\Omega\to\G$ such that $\Sigma=\Psi(\Omega)$.
\begin{comment}
We choose the graded Euclidean Riemannian metric $\delta$,
that on the fixed graded coordinates $(x_1,\ldots,x_\q)$ of $\G$
yields the constant identity matrix $\delta_{ij}$.
\end{comment}
Joining Proposition~\ref{pr:intmeaslocalchart} 
and Theorem~\ref{t:AreaHoriz}, for every open subset $H\subset\Omega$ there holds
\beq\label{eq:S^nSigma}
\cS^\n_d\res\Sigma\lls \Psi(H)\rrs=\int_H \|\pi_{\Psi(y),\n}\lls \der_{y_1}\Psi(y)\wedge\cdots\wedge\der_{y_\n}\Psi(y)\rrs\|_g\,dy.
\eeq
From the area formula of \cite{Mag}:
\[
\cH^\n_d(\Psi(H))=\int_H \frac{\cH^\n_d(D\Psi(x)(B_E))}{\cL^\n(B_E)} dx,
\]
where $D\Psi(x):\R^\n\to\G$ is the Lie group homomorphism
defining the differential, see \cite{Mag} for more information.
For each $x\in\Omega$ both $\cH^\n_d$ and $\cH^n_{|\cdot|}$ 
are Haar measures on the horizontal subgroup $V_x=D\Psi(x)(\R^\n)\subset H^1$, therefore
\[
\cH^\n_d\res V_x=\frac{\cH^\n_d(V_x\cap\B)}{\cH^\n_{|\cdot|}(V_x\cap \B)}\cH^\n_{|\cdot|}\res V_x.
\]
The Haar property of these measures follows from the commutativity
of $V_x$, hence the BCH yields
$yA=y+A$, whenever $y\in V_x$ and $A\subset V_x$.
By Proposition~\ref{pr:MetricHoriz}
and definition \eqref{d:omega_dH^n} we have
\beq\label{d:gamma_dnH^n}
\cH^\n_d\res V_x=\cH^\n_0(V_x\cap\B)\,\cH^\n_{|\cdot|}\res V_x
=\cH^\n_{|\cdot|}\res V_x,
\eeq
in view of Lemma~\ref{l:inv_multiradial}. We have proved that
\beq\label{eq:areaf_d_constants}
\cH^\n_d(\Psi(H))=\int_H \frac{\cH^\n_{|\cdot|}(D\Psi(x)(B_E))}{\cL^\n(B_E)} dx.
\eeq
The Lie group homomorphism $D\Psi(x):\R^k\to\G$ is defined as the limit
\beq\label{eq:limL(w)}
D\Psi(x)(v)=\lim_{t\to0^+}\delta_{1/t}\pa{\Psi(x)^{-1}\Psi(x+tv)},
\eeq
that exists for all $x\in \Omega$, in view of \cite[Theorem~1.1]{Mag14}.
Exploiting the BCH formula in the limit \eqref{eq:limL(w)} 
and the fact that the image of $D\Psi(x)$ must be in a horizontal subgroup,
we have
\beq\label{eq:dPsi(h)}
D\Psi(x)(h)=d\widetilde\Psi(x)(h) \in V_x\subset\G,
\eeq
where $\widetilde\Psi=P_{H^1}\circ\Psi$ and $h$ is any vector of $\R^\n$.
Finally, the expression
\[
\pi_{\Psi(y),\n}\lls \der_{y_1}\Psi(y)\wedge\cdots\wedge\der_{y_\n}\Psi(y)\rrs
\]
can be more explicitly written as
\[
\sum_{j_1,\ldots,j_\n=1}^\m\pi_{\Psi(y),\n}\lls (\der_{y_1}\Psi(x))^{j_1} X_{j_1}(\Psi(x))\wedge\cdots\wedge\der_{y_\n}\Psi(x))^{j_1} X_{j_1}(\Psi(x))\rrs
\]
where the special polynomial form of the vector fields $X_j$ implies that
the component $(\der_{y_i}\Psi(x))^j$ of $\der_{y_i}\Psi(x)$
with respect to the basis $(X_1(\Phi(x)),\ldots,X_\m(\Phi(x))$
coincides with $\der_{y_1}\Psi^j(x)$. This allows us to conclude that 
\[
\|\pi_{\Psi(x),\n}\lls \der_{y_1}\Psi(x)\wedge\cdots\wedge\der_{y_\n}\Psi(x)\rrs\|_g=J\tilde\Psi(x)=
\frac{\cH^\n_{|\cdot|}(d\tilde\Psi(x)(B_E))}{\cL^\n(B_E)}.
\]
As a consequence, by \eqref{eq:dPsi(h)},
\eqref{eq:areaf_d_constants} and \eqref{eq:S^nSigma}, our claim follows.  
\end{proof}

\bibliography{bibtex-Area}{}
\bibliographystyle{plain}

\end{document}